\documentclass{amsart}

\usepackage[utf8]{inputenc}
\usepackage[symbol]{footmisc}
\usepackage{tikz,amsfonts,amsthm,amsmath,amssymb,xcolor,blindtext,subcaption,multicol,biblatex,xspace,tagging}
\usetikzlibrary{arrows.meta,decorations.pathmorphing,backgrounds,decorations.pathreplacing, decorations.shapes,positioning,fit,petri,calc,bending,matrix,animations,intersections}

\newcommand{\parallelproofrow}[2]{%
  \noindent
  \begin{minipage}[t]{0.48\textwidth}
    \vspace{0pt}#1
  \end{minipage}%
  \hfill
  \begin{minipage}[t]{0.48\textwidth}
    \vspace{0pt}#2
  \end{minipage}%
  \par\medskip
}

\usepackage{booktabs}

\usepackage[shortlabels]{enumitem}

\usepackage[margin=1.2in]{geometry}

\usetag{comments} 

\newcommand\sltwo{\ensuremath{\mathfrak{sl}_2}\xspace}

\newtheorem{theorem}{Theorem}[section]
\newtheorem{conjecture}[theorem]{Conjecture}
\newtheorem{lemma}[theorem]{Lemma}
\newtheorem{corollary}[theorem]{Corollary}
\newtheorem{proposition}[theorem]{Proposition}
\newtheorem{example}[theorem]{Example}
\newtheorem{defthm}{Definition/Theorem}[section]

\theoremstyle{definition}
\newtheorem{definition}[theorem]{Definition} 

\title{Strong \sltwo Actions and Symmetric Chain Decompositions}

\author{Nate Harman}

\author{Rylan Rosas}

\author{Charles Troutman III}

\author{Chris Wu}

\begin{document}

\begin{abstract}
   
   We introduce a notion of a strong $\mathfrak{sl}_2$-action on a poset, and conjecture the existence of such an action is equivalent to the existence of a symmetric chain decomposition of that poset.  We explain the intuition behind the conjecture, go through a number of explicit examples, and prove some special cases and relaxations of the conjecture.  Then we explain a computational experiment where we tested the conjecture on millions of randomly constructed rank-symmetric, rank-unimodal, and strongly Sperner posets.

\end{abstract}

\maketitle

\thispagestyle{empty}
\vspace{-.7cm}
\begin{section}{Introduction}

Given a ranked (i.e. graded) poset $P$, we can renormalize the ranks so that the maximal and minimal ranks are $n$ and $-n$ respectively for some $n$, with steps of size $2$.   With this convention we have the following definition:

\begin{definition}
A \emph{strong $\mathfrak{sl}_2$-action} on $P$ is the data of linear operators $E$, $F$, and $H$ on $\mathbb{C}P$ such that:

\begin{enumerate}

\item $H$ scales elements $x \in P$ according to the rank: $Hx = \omega(x)x$, where $\omega$ is the (renormalized) rank function.

\item $E$ acts upwards along cover relations:  $Ex = \sum_{y \gtrdot x} e_{x,y}y$ for some complex numbers $e_{x,y}$. 

\item $F$ acts downwards along cover relations:  $Fx = \sum_{y \lessdot x} f_{x,y}y$ for some complex numbers $f_{x,y}$. 

\item The operators satisfy the $\mathfrak{sl}_2$ relations:  $EF-FE = H$, $HE-EH = 2E$, and $HF-FH = -2F$

\end{enumerate} 

\end{definition}

The ``strong" modifier is to distinguish this from an earlier notion due to Proctor in \cite{proctor82}  where we only require the raising operator $E$ to act along cover relations, and $F$ just has to lower the rank.  A poset admits an $\mathfrak{sl}_2$-action in the sense of Proctor if and only if it is rank symmetric, unimodal, and satisfies the strong Sperner property -- such posets are called \emph{Peck}.  See Section \ref{Peckdef} for more details.

As an example,  here is a visual depiction of a strong $\mathfrak{sl}_2$-action on the Young lattice $L(3,3)$:

\begin{center}
\includegraphics[scale=0.28]{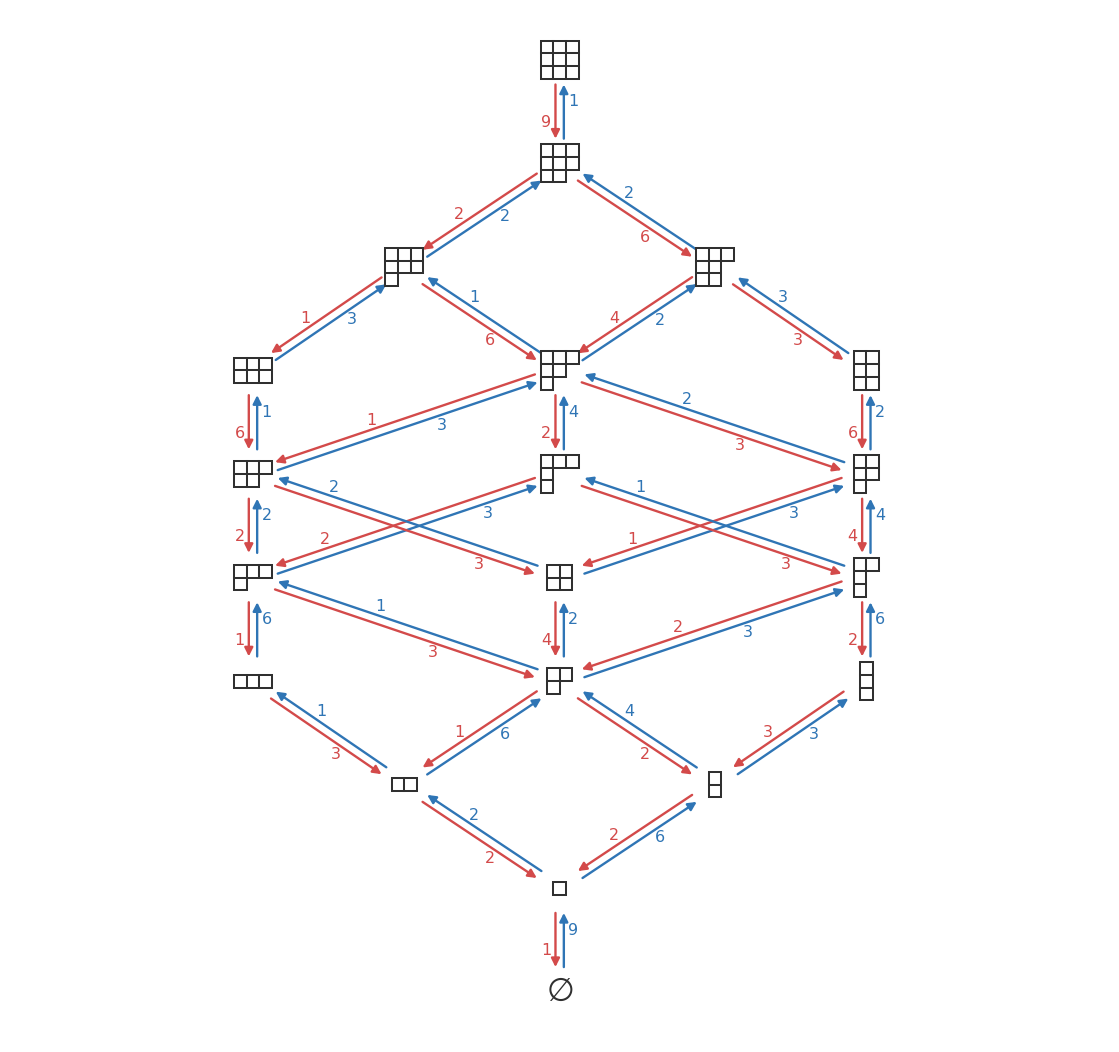}
\end{center}
The blue upwards arrows encode the raising operator $E$, and the red downwards arrows are represent the lowering operator $F$. The operator $H$  acts according to the rank, which from top to bottom go from $9$ to $-9$ in steps of size 2.   Famously, the Young lattices $L(m,n)$ are conjectured to always admit a symmetric chain decomposition -- a covering of the poset by disjoint saturated chains that begin and end in opposite ranks.

In this paper we won't actually look at the $L(m,n)$ posets at all.  Instead we will focus on what we feel is actually the main underlying question about symmetric chain decompositions:

\smallskip

\noindent \textbf{Open Problem:} \emph{Give an algebraic characterization of posets which admit a symmetric chain decomposition.}

\smallskip

\noindent With this broader goal in mind we offer the following conjecture\footnote{We came to this conjecture independently, but it turns out to be equivalent to the $\mathfrak{sl}_2$ case of a beautiful (and not nearly well-known enough) question formulated by Donnelly  for arbitrary semisimple Lie algebras \cite{Donnelly1998Solitary}.  We will discuss his framework more in Section \ref{supportgraphs}. }:

\begin{conjecture}\label{theconjecture}

A ranked poset $P$ admits a strong $\mathfrak{sl}_2$-action if and only if it admits a symmetric chain decomposition. 

\end{conjecture}

One direction is easy, if $P$ admits a symmetric chain decomposition then we can define a $\mathfrak{sl}_2$-action along the chains, see Proposition \ref{SCDimpliesstrong}. The goal of this paper is to provide computational evidence and partial results toward the other direction.  To start we'll provide some intuition for why this might be true by looking at a simple example:

\begin{example} \label{smallestexample}
\ \vspace{-.3cm}
\begin{center}
\includegraphics[scale=0.3]{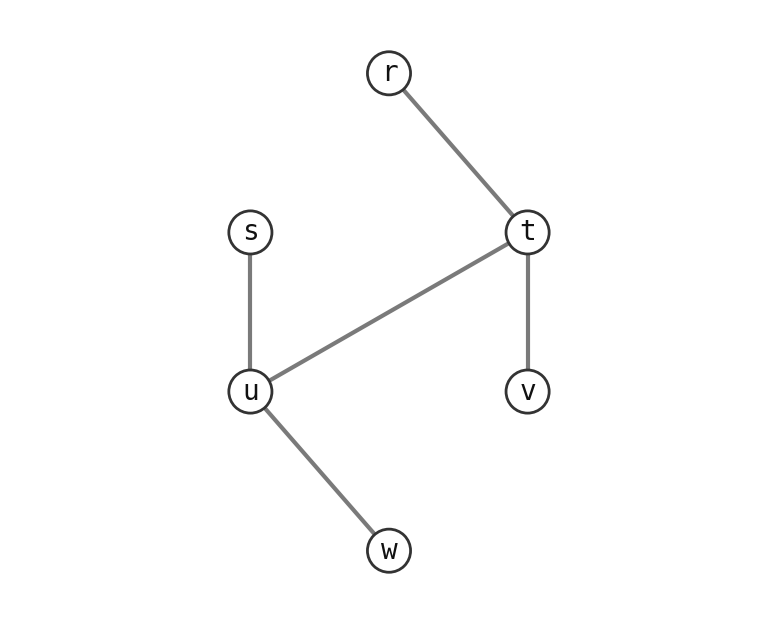}
\end{center}

\end{example}

\parallelproofrow
  {\textbf{Claim:} The above poset does not admit a symmetric
   chain decomposition.}
  {\textbf{Claim:} The above poset does not admit a strong
   $\mathfrak{sl}_2$-action.}

\parallelproofrow
  {\textbf{Proof:} Assume a symmetric chain decomposition exists.}
  {\textbf{Proof:} Assume a strong $\mathfrak{sl}_2$-action exists.}

\parallelproofrow
  {$r$ has no upwards neighbors, so it must be the start of a
   chain of length $4$.}
  {$r$ has no upwards neighbors, so it must be a highest-weight
   vector in the $L(3)$-isotypic component.}

\parallelproofrow
  {$t$ is the unique downwards neighbor of $r$, so it must be in
   a chain of length $4$.}
  {$t$ is the unique downwards neighbor of $r$, so
   $t=cF(r)$ for a nonzero constant $c$, and $t$ is therefore in
   the $L(3)$-isotypic component.}

\parallelproofrow
  {The element $v$ has no downwards neighbors, so it must be the
   end of a chain of length $2$.}
  {The element $v$ has no downwards neighbors, so it must be a
   lowest-weight vector in the $L(1)$-isotypic component.}

\parallelproofrow
  {$t$ is the unique upwards neighbor of $v$. Therefore, $t$ must
   be in a chain of length $2$.}
  {$t$ is the unique upwards neighbor of $v$, so
   $t=dE(v)$ for a nonzero scalar $d$, and therefore $t$ is in
   the $L(1)$-isotypic component.}

\parallelproofrow
  {This is a contradiction: $t$ cannot simultaneously be in
   chains of two different lengths.\hfill$\square$}
  {This is a contradiction: $t$ cannot simultaneously be in two
   different isotypic components.\hfill$\square$}
\vspace{.6cm}

This example is very simple, and there are a number of different arguments one could come up with for both of these facts. However we hope these parallel arguments demonstrate the following informal observation, which underpins why we believe this could be true:

\medskip

\emph{Obstructions to a symmetric chain decomposition seem to translate, often directly,  into obstructions to a strong $\mathfrak{sl}_2$-action}.

\medskip

We have constructed many examples of Peck posets which do not admit a symmetric chain decomposition. Not only have we not found any admitting a strong $\mathfrak{sl}_2$-action, but as a general rule we've found that the easier it is to prove the poset does not have an SCD the easier it is to prove it does not admit a strong $\mathfrak{sl}_2$-action.

\medskip

The remainder of the paper is structured as follows:

\begin{itemize}

\item In Section \ref{background} we set our conventions and cover the necessary background information.  This includes a brief overview of the representation theory of $\mathfrak{sl}_2$, an explanation of the dictionary between matrices and bipartite graphs leading to the Lindstr\"om–Gessel–Viennot lemma,  several equivalent characterizations of a Peck poset, and a discussion of Donnelly's support graph framework.  We also introduce the notion of a fractional symmetric chain decomposition.

\smallskip

\item In Section \ref{combinatorics} we look at how the combinatorics of the poset can constrain a possible strong $\mathfrak{sl}_2$-action, and vice versa.  We give a number of examples of Peck posets which do not admit a symmetric chain decomposition, and explain why these constraints prevent them from having a strong $\mathfrak{sl}_2$-action. We close the section with a proof that the weak Bruhat order for the Coxeter group of type H3 does not admit a strong $\mathfrak{sl}_2$-action. 

\smallskip 

\item In Section \ref{specialcases} we look at special cases and relaxations of the conjecture. We prove the conjecture for several families of rank vectors, and for ``diamond-free" posets. We also show that every poset with $4$ ranks which admits a strong $\mathfrak{sl}_2$-action also admits a fractional symmetric chain decomposition. 

\smallskip

\item In Section \ref{experiments} we discuss the computational experiments we ran, testing the conjecture on millions of randomly generated Peck posets, including thousands of randomly generated worst-case-scenario SCD-critical posets. 

\end{itemize}

\subsection*{Acknowledgments}
Thanks to Joshua Mundinger for many helpful conversations.  This project was started at the 2022 MathILy-EST REU funded by NSF Grant DMS-2149647, and later the first author was supported by NSF grant DMS-2401515. 

\subsection*{AI Acknowledgment}

We used ChatGPT (primarily version 5.6) to write essentially all of the code used in the experiments discussed in Section \ref{experiments}, to find relevant references, to create all of the diagrams and tables in this paper, and to help edit the paper.  All of the proofs, definitions, and conjectures were made by humans.

\end{section}

\begin{section}{Background and Definitions}\label{background}

\subsection{Representations of $\mathfrak{sl}_2$}  For simplicity, we will work over the complex numbers, although much of this paper makes sense over arbitrary fields of characteristic zero.   The Lie algebra $\mathfrak{sl}_2 = \mathfrak{sl}_2(\mathbb{C})$ is the space of $2 \times 2$ complex matrices of trace zero.  It has a standard basis:
$$ e = \begin{bmatrix}
0 & 1\\
0 & 0 
\end{bmatrix}
\hspace{1cm}
 f = \begin{bmatrix}
0 & 0\\
1 & 0 
\end{bmatrix}
\hspace{1cm}
 h = \begin{bmatrix}
1 & 0\\
0 & -1 
\end{bmatrix}$$
which satisfy the relations $[e,f] = h$, $[h,e] = 2e$, and $[h,f] = -2f$ where $[a,b] = ab-ba$ denotes the Lie bracket.  A \emph{representation} of $\mathfrak{sl}_2$ is a complex vector space $V$ equipped with a Lie algebra action of $\mathfrak{sl}_2$.  Concretely, a representation is the data of three linear maps $E,F,H \in End_{\mathbb{C}}(V)$ satisfying $EF-FE = H$, $HE-EH = 2E$ and $HF-FH = -2F$.

\begin{itemize}

\item A subspace $W \subseteq V$ is a \emph{subrepresentation} if it is preserved by $E$, $F$ and $H$.  Note that in this case $W$ is naturally an $\mathfrak{sl}_2$-representation, as is the quotient $V/W$. 

\item We say $V$ is \emph{irreducible} if it has no nonzero proper subrepresentations. 

\item The eigenvalues of $H$ are called \emph{weights}. A vector $\vec{v} \in V$ is called a \emph{weight vector} if it is an eigenvector for $H$.  If we specify the eigenvalue $\lambda$ we'll call $\vec{v}$ a $\lambda$-weight vector. 

\item A \emph{weight basis} for $V$ is a basis consisting of weight vectors.

\item If $\vec{v} \in V$ is a weight vector satisfying $E \vec{v} = 0$ we say $\vec{v}$ is a \emph{highest-weight vector}.

\item If $\vec{v} \in V$ is a weight vector satisfying $F \vec{v} = 0$ we say $\vec{v}$ is a \emph{lowest-weight vector}.

\end{itemize}

As an illustrative example, let $L(n)$ denote the space of homogeneous degree $n$ polynomials in two variables, $x$ and $y$.  This is an $(n+1)$-dimensional space with a natural basis $\{x^n, x^{n-1}y, \dots, xy^{n-1}, y^n\}$.  If we consider the operators
$$E = x \frac{d}{dy} \hspace{1cm} F = y \frac{d}{dx} \hspace{1cm} H = x\frac{d}{dx} - y\frac{d}{dy}$$
acting on $L(n)$ in the natural way it is a standard exercise to verify 

\begin{enumerate}[a)]
\item These operators satisfy the $\mathfrak{sl}_2$ relations and therefore define an $\mathfrak{sl}_2$-representation on $L(n)$.

\item $L(n)$ is irreducible.

\item The standard basis $\{x^n, x^{n-1}y, \dots, xy^{n-1}, y^n\}$ is a weight basis with weights $n, n-2, \dots,$ $-(n-2),$ and $-n$ respectively.

\item Up to scaling,  $x^n$ and $y^n$ are the unique highest-weight and lowest-weight vectors in $L(n)$ respectively.

\end{enumerate}

In fact, these spaces $L(n)$ essentially completely describe the finite dimensional representation theory of $\mathfrak{sl}_2$, as summarized by the following proposition:

\begin{proposition} (For a standard reference, see \cite{humphreys72} chapters $6$ and $7$.)

Let $V$ be a finite dimensional $\mathfrak{sl}_2$-representation. 

\begin{enumerate}[a)]

\item $V$ can be decomposed as a direct sum of irreducible representations, each isomorphic to an $L(k)$: 
$$V  \cong L(k_1) \oplus L(k_2) \oplus \dots \oplus L(k_\ell)$$ 

\item This decomposition is typically not unique, but the coarser decomposition
$$V \cong (L(0)^{\oplus m_0}) \oplus (L(1)^{\oplus m_1}) \oplus \dots \oplus (L(n)^{\oplus m_n})$$
 into isotypic components is unique.  In particular any two decompositions of $V$ into irreducibles have the same multiplicities $m_k$.

\smallskip

\item The multiplicity $m_k$ is equal to the dimension of the space of weight $k$ highest-weight vectors in $V$ (or equivalently, the dimension of the space of weight $-k$ lowest-weight vectors in $V$)

\smallskip

\item The decomposition into isotypic components $V(k) = L(k)^{\oplus m_k}$ coincides with the eigenspace decomposition for the Casimir operator $\Omega = EF + FE + \frac{1}{2}H^2 \in End(V)$.

\end{enumerate}

\end{proposition}

\subsection{Posets and Hasse Diagrams}  A \emph{partially ordered set} or \emph{poset} is a set $P$ equipped with a reflexive, antisymmetric, and transitive relation $\preceq$.  There are additional relations $\prec$, $\succeq$, and $\succ$, which are all naturally defined in terms of $\preceq$.

 Given two elements $a$ and $b$ in a poset $P$ we say $a$ \emph{covers} $b$, written $a \gtrdot b$ if $a \succ b$ and if $c$ satisfies $a \succeq c \succeq b$ then either $c =a$ or $c =b$. In other words, $a$ covers $b$ if $a$ is bigger than $b$ and there is nothing in between them.  The data of which elements cover one another is called the \emph{cover relations} of $P$.  If $P$ is a finite poset, it is a standard exercise to show that $P$ is determined by its cover relations. 
 
We will be representing finite posets via their \emph{Hasse diagrams}. These are graphs drawn in the plane satisfying:

\begin{itemize}

\item The underlying graph has a vertex for each element in $P$, and an edge between $a$ and $b$ if $a \gtrdot b$.

\item The graph is drawn in the plane so that if $a \succ b$ then $a$ is higher in the plane than $b$ (i.e. its $y$-coordinate is larger)

\end{itemize}

As an example, if $P = \{1, 2, 3, 4, 12 \}$ ordered by divisibility then the Hasse diagram is: 

\begin{center}
\includegraphics[scale=0.3]{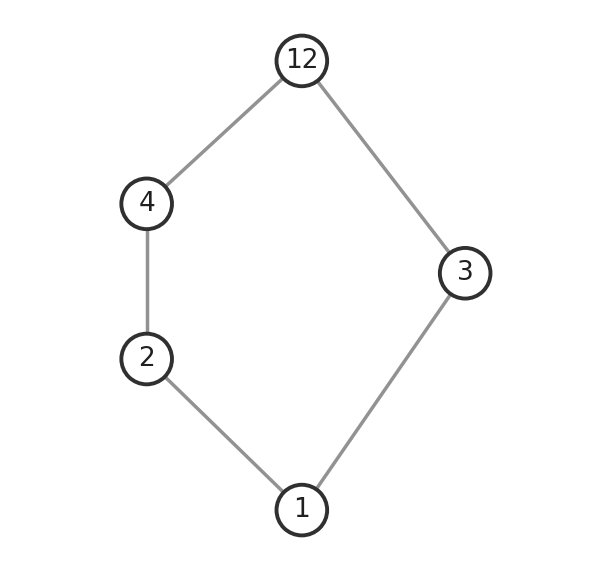}
\end{center}

We often won't really distinguish between the poset itself and its Hasse diagram.  In particular if we use graph theoretic terms like edges,  vertices, paths, or connected components when describing posets, we are always referring to a corresponding feature of the Hasse diagram. 

\subsection{Ranked posets}  A \emph{ranked} poset (sometimes called a \emph{graded} poset) is a poset $(P, \preceq)$ equipped with a rank function $\omega: P \to \mathbb{Z}$ such that 

\begin{enumerate}

\item If $a \succeq b$ then $\omega(a) \ge \omega(b)$. 

\item If $a \gtrdot b$ then $\omega(a) = \omega(b) + 2$.

\item For all $a, b \in P$ $\omega(a) \equiv \omega(b) \mod 2$

\end{enumerate}

Notice that we are allowing negative ranks and are taking steps of size two along cover relations ($\omega(a) = \omega(b) + 2$ instead of $\omega(a) = \omega(b) + 1$).  These are non-standard conventions, but they amount to reindexing from the usual notion. Namely, we scale the rank by a factor of $2$ and possibly shift it.  Note that if the Hasse diagram is connected then condition $3$ follows from the first two.  Ultimately we will want our rank function to agree with the weights of an $\mathfrak{sl}_2$-representation, which is why we make this convention. 

For ranked posets we impose an additional condition on their Hasse diagrams: the height of a vertex in the diagram is equal to the rank of the corresponding poset element. For example if $P = B(3)$ is the Boolean lattice on $3$ elements, there is a natural rank function $\omega(S) = |S| - |S^c|$ and the Hasse diagram looks like:

\begin{center}
\includegraphics[scale=0.3]{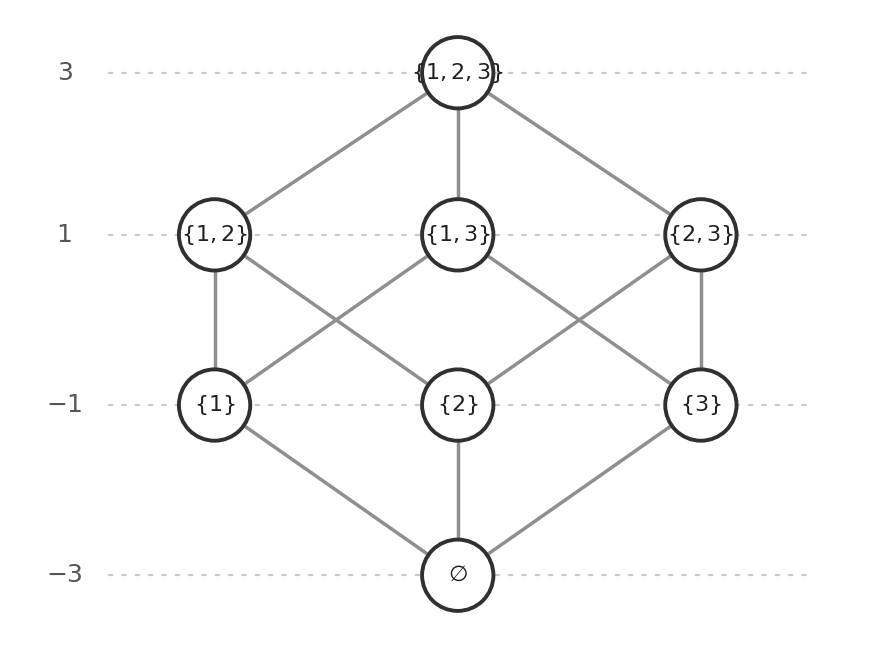}
\end{center}

For a ranked poset $P$ we let $P_n$ denote the set of rank $n$ elements, i.e. those $p \in P$ with $\omega(p) = n$. The \emph{rank vector} records the sizes $|P_n|$ of the nonempty ranks, starting from the highest non-empty rank.   We say that a ranked poset is \emph{rank symmetric} if $|P_n| = |P_{-n}|$ for all $n$. For $P$ connected and rank symmetric we say $P$ is \emph{unimodal} if $|P_n| \ge |P_{n+2}|$ for all $n \ge 0$.  Ranked posets which are supported in either all even or all odd ranks (which is automatic if they are connected) and are rank symmetric and unimodal are called \emph{RSU-posets}.  In the example above, the Boolean lattice $B(3)$ is an RSU-poset with rank vector $[1,3,3,1]$.

Under our conventions, given a ranked poset $P$, a \emph{symmetric chain decomposition (SCD)} of $P$ is a collection of chains (totally ordered subsets of $P$) satisfying:

\begin{enumerate}

\item Each chain is \emph{saturated}, meaning it does not skip any ranks:  For integers $\ell < m < n$ which are congruent modulo $2$, if a chain contains elements of rank $\ell$ and $n$, then it must contain an element of rank $m$ as well.

\smallskip

\item If the top element of a chain is in rank $k$, then the bottom element of that chain must be in rank $-k$.

\smallskip

\item Each element of the poset is contained in exactly one chain in the collection.
\end{enumerate}

\begin{proposition}\label{SCDimpliesstrong}
If a poset $P$ admits an SCD then $P$ admits a strong $\mathfrak{sl}_2$ action.  
\end{proposition}

\begin{proof}
Given SCD on $P$, first set all $E$- and $F$-coefficients along edges not in the SCD equal to zero.  Then for a symmetric chain in the SCD of maximal rank $k$,  identify its elements with the standard basis for $L(k)$ to define an action along that chain.
\end{proof}

\subsection{Matrices and Bipartite Graphs} \label{graphdictionary} There is a natural correspondence between $n \times m$ matrices and bipartite graphs on $\{a_1, \dots , a_n\} \cup \{b_1, \dots b_m \}$ with weighted edges:  Given an $n \times m$ matrix $A$, if an entry $A_{i,j}$ is nonzero we add an edge with that weight from $b_j$ to $a_i$.  

\begin{center}
\includegraphics[scale=0.3]{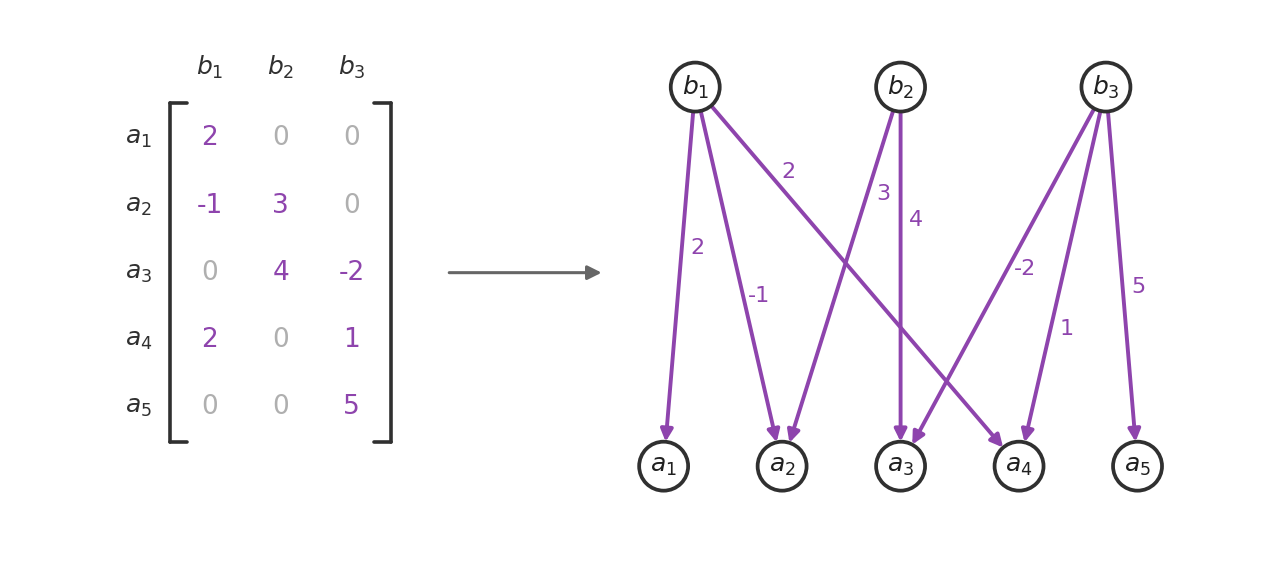}
\end{center}

\subsubsection{Minors and matchings:} Given subsets $I$ and $J$ of the rows and columns respectively,  we can look at the induced subgraph on those corresponding vertices.  If $|I| = |J|$ we can interpret the corresponding minor as an alternating sum over matchings on this induced subgraph.   In particular if such a minor is nonzero then there must be at least one matching between those vertices in the bipartite graph. 

\begin{center}
\includegraphics[scale=0.25]{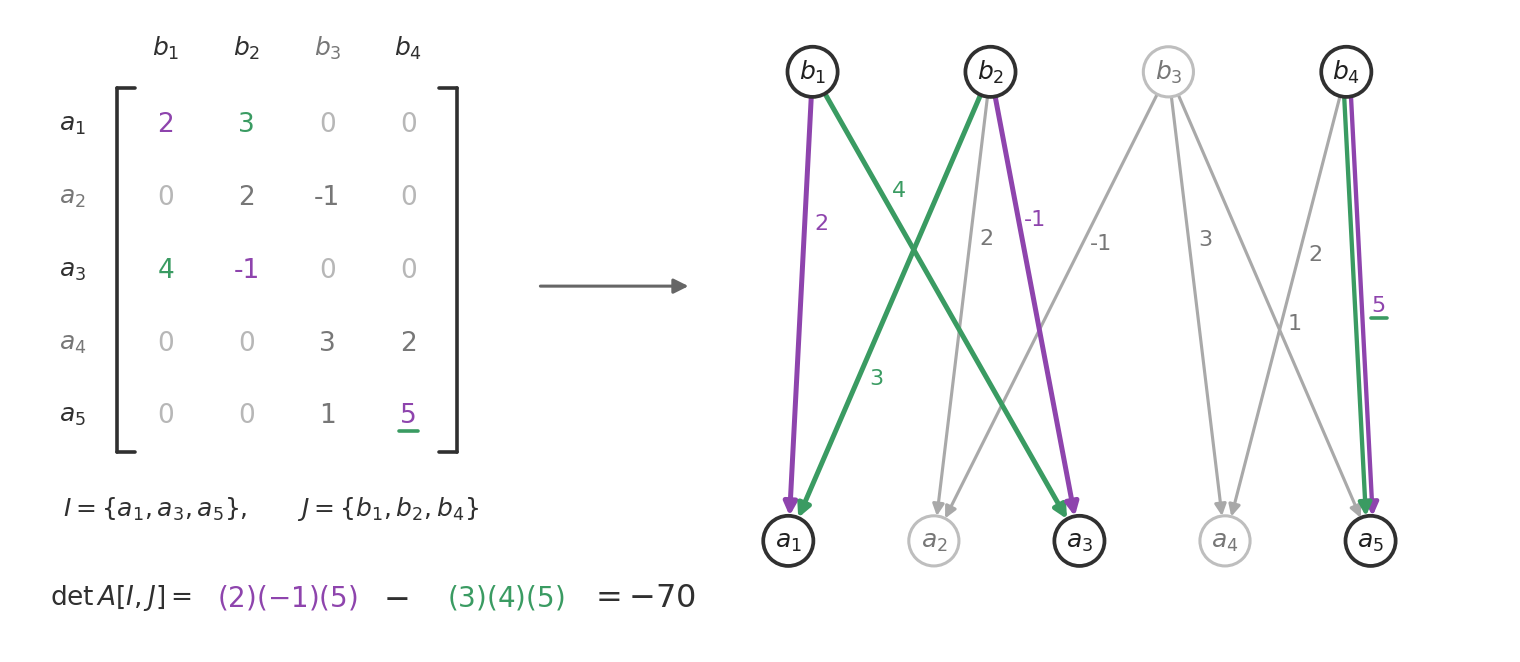}
\end{center}

\subsubsection{Multiplication and graph stacking:} Under this identification, matrix multiplication corresponds to the following operation:  First we stack the two bipartite graphs into a $3$-layer graph.  Then we construct a new bipartite graph just on the two outer layers of vertices,  with edge weights given by the sum over all length 2 paths in the 3-layer graph of the product of weights along those paths.

\begin{center}
\includegraphics[scale=0.25]{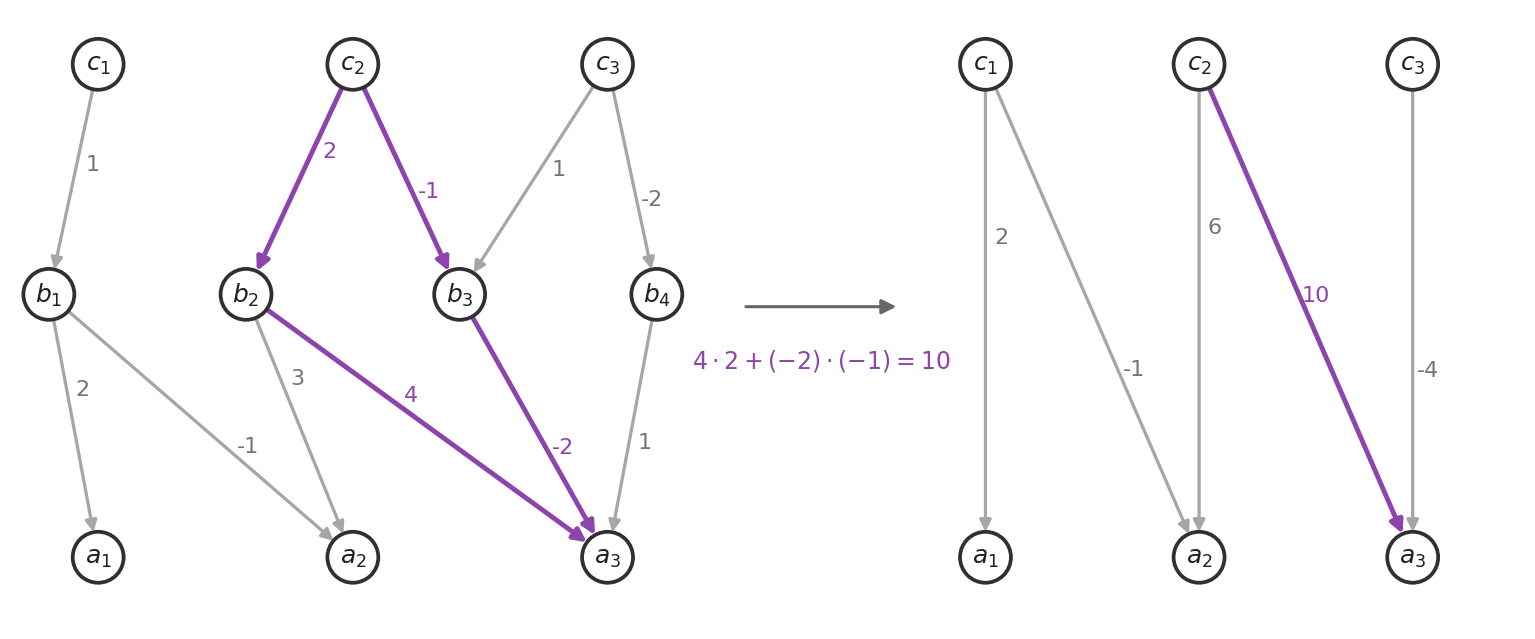}
\end{center}

\subsubsection{Cauchy-Binet and disjoint paths}  Combining these, the Cauchy-Binet formula expresses a minor of a product $AB$ of matrices as a sum of products of minors of $A$ and $B$.    A combinatorial manifestation of this is the Lindstr\"om–Gessel–Viennot lemma (see \cite{Lindstrom1973} and \cite{GesselViennot1985}) which computes a minor of a product of two or more matrices graphically by stacking the bipartite graphs and then taking a signed sum over collections of disjoint paths connecting $I$ to $J$ of the product of weights along those paths.

\begin{center}
\includegraphics[scale=0.25]{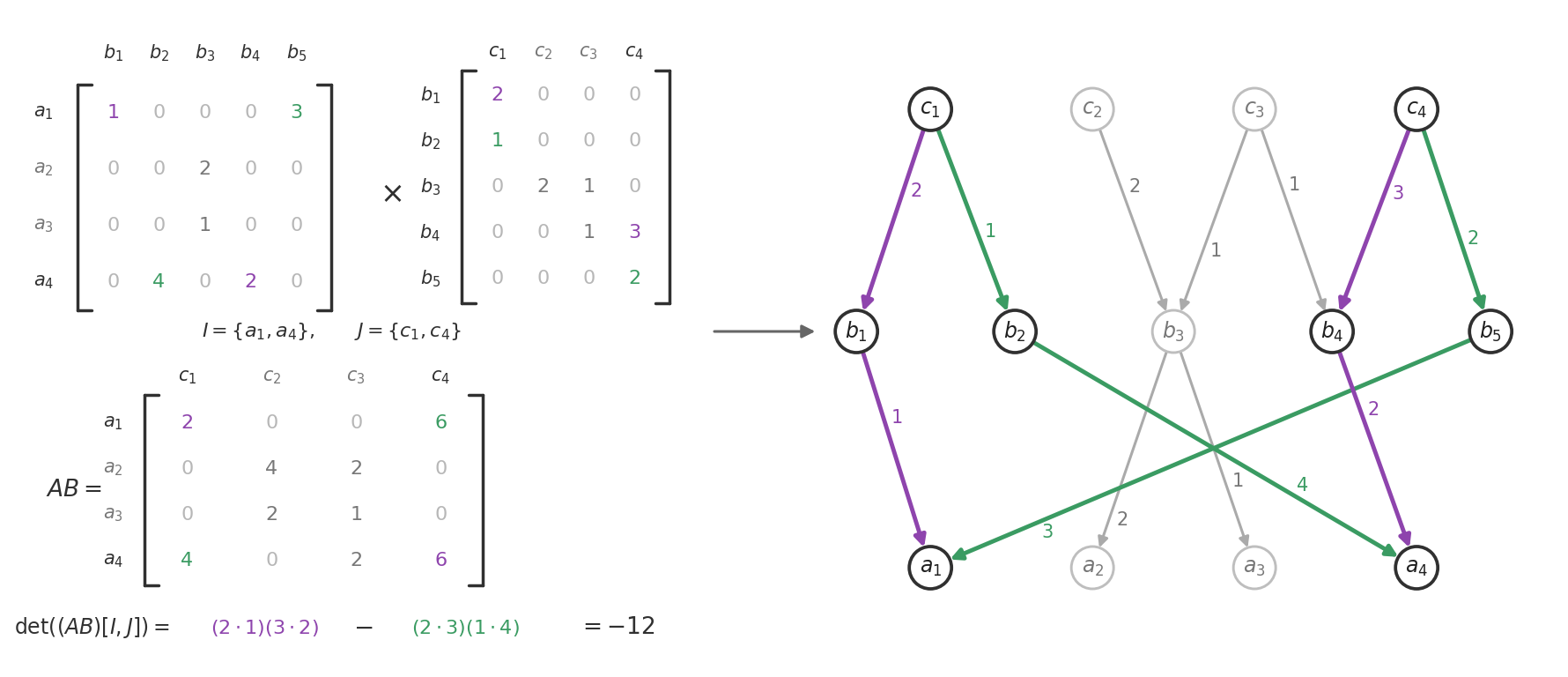}
\end{center}

In particular, if the $I \times J$ minor of a product is known to be nonzero then there exists at least one collection of disjoint paths connecting $I$ to $J$ in the stacked graph.  Or more generally, if the $I \times J$ submatrix of the product is known to have rank at least $k$, then we can find at least $k$ disjoint paths connecting $I$ to $J$.

\subsection{Peck posets} \label{peckintro}  We'll now introduce the main class of posets we will be interested in, so-called Peck posets.  If a poset admits an SCD then it is automatically Peck, and hence a failure to satisfy these weaker conditions is an obstruction to admitting an SCD.  However there are many examples of Peck posets which do not admit an SCD, and there seems to be very little in the literature towards understanding the failure of Peck posets to admit SCDs.

\subsubsection{The Sperner Property} We say a ranked poset has the \emph{Sperner property} if the every antichain has size at most the largest rank. 
\begin{center}
\includegraphics[scale=0.25]{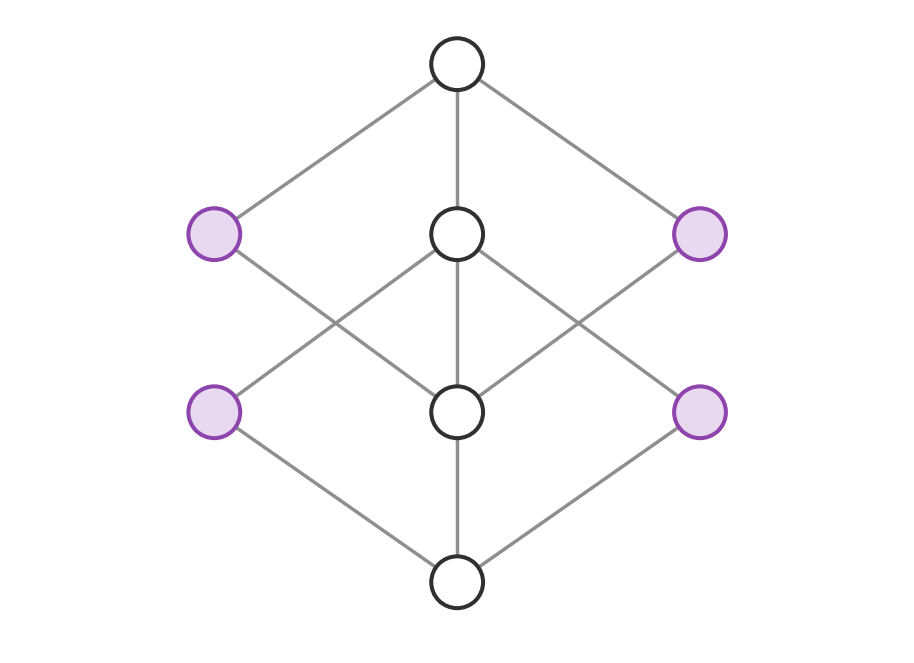}
\end{center}
In the above example the highlighted vertices form an antichain of size $4$, but the maximum rank size is $3$. Therefore this poset does not satisfy the Sperner property.  More generally, we say a ranked poset is \emph{strongly Sperner} if for all $k \ge 1$ every disjoint union of $k$ antichains has size at most the union of the $k$ largest ranks.

\begin{center}
\includegraphics[scale=0.25]{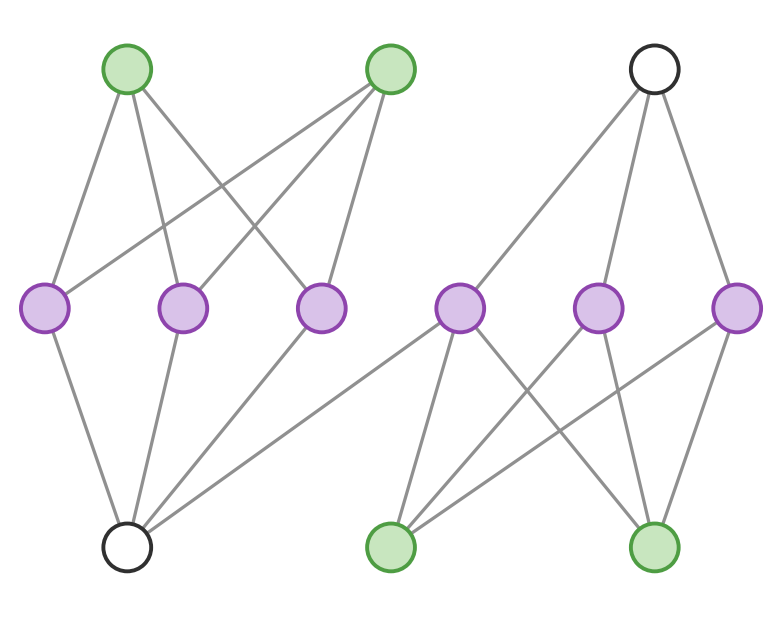}
\end{center}
In this example the largest antichain has size $6$, equal to the largest rank. However the largest union of two antichains has size $10$ which is larger than the union of the two largest ranks.  Therefore this poset is Sperner but not strongly Sperner.

\subsubsection{Griggs' property T} We say a rank symmetric poset $P$ satisfies \emph{Griggs' property T} if for each pair of opposite ranks $P_n$, $P_{-n}$ there exists a collection of disjoint, saturated chains beginning and ending at those ranks, containing every element in those ranks. 

\begin{center}
\includegraphics[scale=0.25]{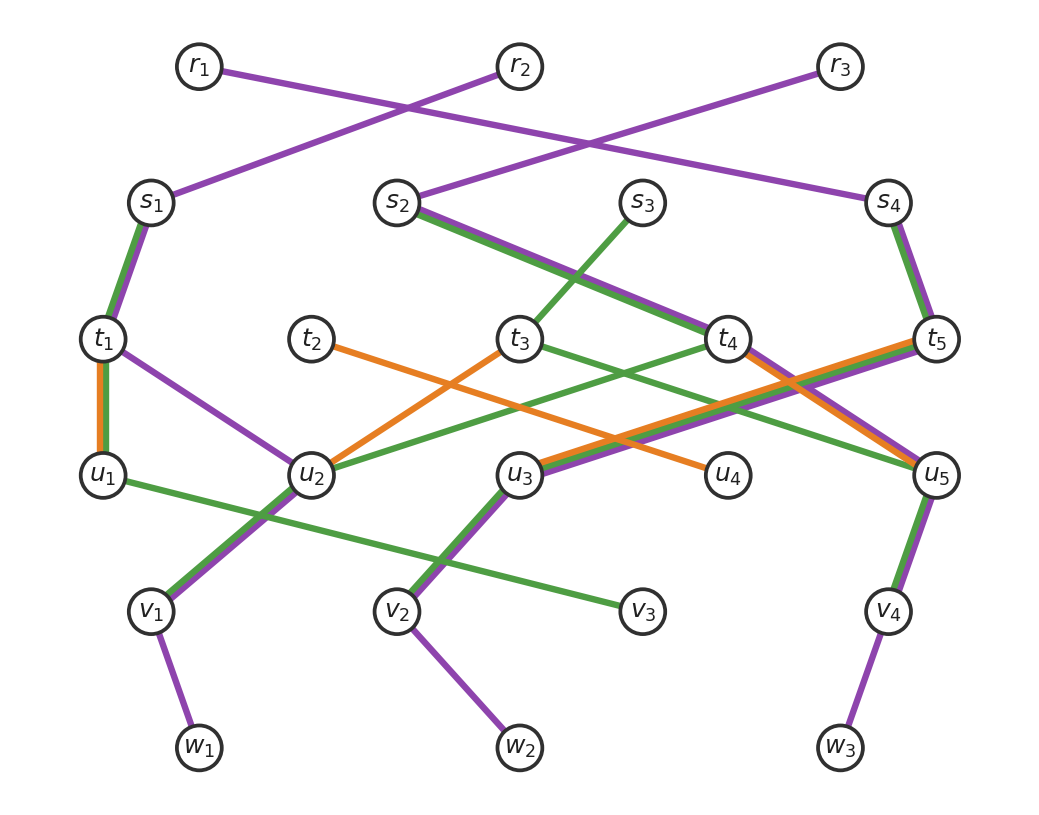}
\end{center}

Note that there does not need to be any compatibility between these collections of disjoint chains for different pairs of opposite ranks. In the above example, each color corresponds to a family of disjoint edges connecting opposite ranks.  The different families have some overlap but they do not form an SCD.

\subsubsection{Proctor's $\mathfrak{sl}_2$-actions} Following \cite{proctor82}, we say a ranked poset $P$ \emph{admits an $\mathfrak{sl}_2$-action} in the sense of Proctor (a weaker notion than the strong actions defined in the introduction) if there is an action of the Lie algebra $\mathfrak{sl}_2$ on the vector space $\mathbb{C}P$ such that:

\begin{enumerate}

\item $H$ scales elements $x \in P$ according to the rank: $Hx = \omega(x)x$, where $\omega$ is the rank function.

\item $E$ acts along cover relations:  $Ex = \sum_{y \gtrdot x} c_{x,y}y$ for some complex numbers $c_{x,y}$. 
\end{enumerate}
Note that $F$ is not required to act along cover relations (that would make it a strong action), but the $\mathfrak{sl}_2$-relations do ensure that if $x \in P_n$ then $Fx \in \mathbb{C}P_{n-2}$.  Here is an example of an $\mathfrak{sl}_2$-action on the poset from Example \ref{smallestexample}:

\begin{center}
\includegraphics[scale=0.3]{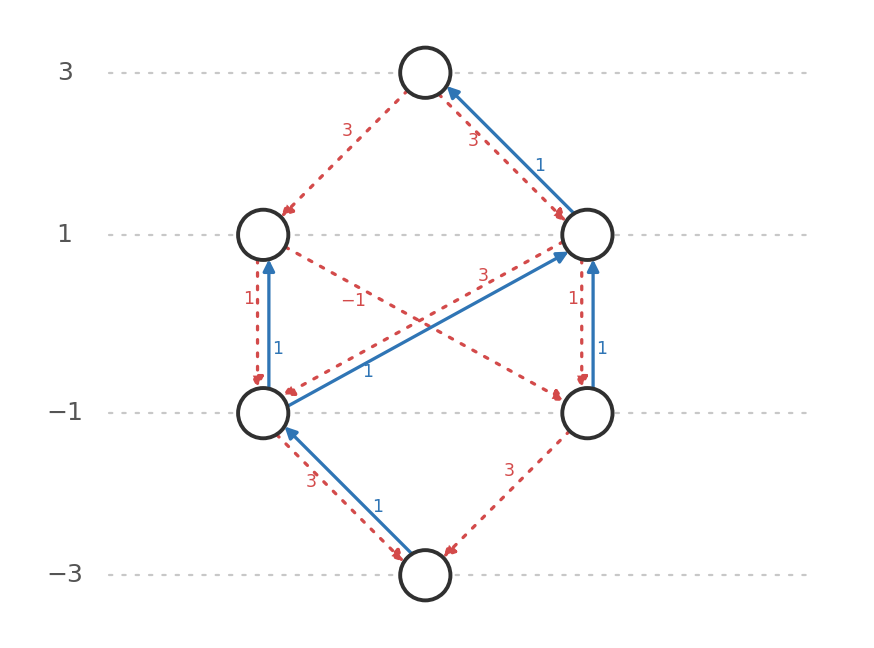}
\end{center}
The $F$ operator, represented with dotted red lines, does not follow the edges of the poset. Thus, the action depicted is not a strong $\mathfrak{sl}_2$-action on this poset.  Indeed, in Example \ref{smallestexample} we saw that this poset does not admit any strong action.  We'll note that it can be viewed as a strong action the larger poset obtained by adding the dotted red edges. Indeed any action in the sense of Proctor can be thought of as a strong action on a larger poset, by adding edges corresponding to nonzero $F$-coefficients. 

\medskip

\subsubsection{Peck poset definition}

\begin{defthm} \label{Peckdef} (Combining \cite{proctor82} and \cite{Griggs1980Chains}) Let $P$ be a rank-symmetric poset. The following are equivalent:

\begin{enumerate}

\item $P$ is strongly Sperner.

\item  $P$ satisfies Griggs' property $T$.

\item $P$ admits an $\mathfrak{sl}_2$-action in the sense of Proctor.

\end{enumerate}

\noindent Such posets are called Peck posets.
\end{defthm}

It is easy to see (using any of these conditions) that if a poset $P$ admits an SCD it is automatically Peck. However we see that Example \ref{smallestexample} is a Peck poset, which does not admit an SCD. We will see many more such examples in this paper.

\subsubsection{SCD-critical posets}

We say that a poset $P$ is an \emph{SCD-critical} poset if

\begin{enumerate}

\item $P$ is Peck. 

\item $P$ does not have an SCD

\item Adding any edge to the Hasse diagram creates an SCD. 

\end{enumerate}

Note that any Peck poset not admitting an SCD can always be enlarged to an SCD-critical poset.  So to prove Conjecture \ref{theconjecture} it would suffice to prove it for SCD-critical posets.   For small enough rank vectors we can sometimes just enumerate all of the SCD-critical posets to verify the conjecture,  although this quickly becomes computationally unfeasible.

\subsection{$\mathfrak{sl}_2$ support graphs} \label{supportgraphs} Let $V$ be a finite dimensional representation of $\mathfrak{sl}_2$ with a weight basis $\{b_1, b_2, \dots , b_n \}$. We define a ranked poset as follows:

\begin{itemize}

\item The elements of the poset are the basis elements $b_1, b_2, \dots, b_n$.

\item The rank function $\omega(b_i)$ is the weight of $b_i$ (its $H$-eigenvalue).

\item Cover relations $b_i \gtrdot b_j$ if either $b_j$ appears with nonzero multiplicity in the expansion of $Fb_i$, or $b_i$ appears with nonzero multiplicity in the expansion of $Eb_j$.

\end{itemize}

By construction we see this poset admits a strong $\mathfrak{sl}_2$-action.  These posets (or the underlying graphs of their Hasse diagrams), along with the numerical data of the $E$- and $F$- coefficients along each edge are called $\mathfrak{sl}_2$-\emph{support graphs}.  For example if we take a standard basis $\{v_1, v_{-1}\}$ for $L(1)$,  we get a natural tensor product basis $\{ v_{\pm 1} \otimes v_{\pm 1} \otimes v_{\pm 1} \}$ for $L(1)^{\otimes 3}$ with a support graph that looks like:

\begin{center}
\includegraphics[scale=0.3]{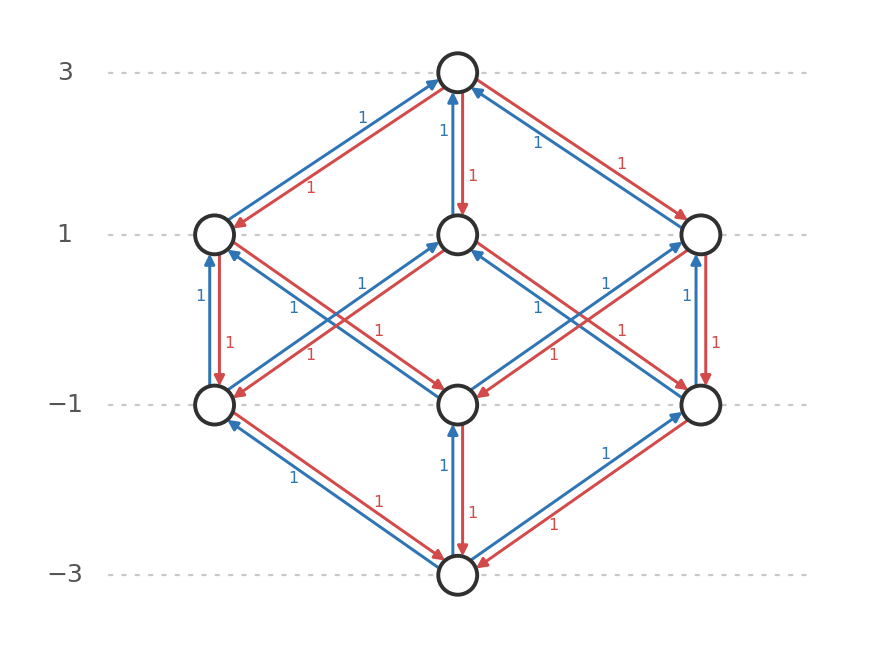}
\end{center}

 These were introduced by Donnelly in \cite{Donnelly1998Solitary} (and further developed in \cite{Donnelly2003Extremal}) along with analogs for all semisimple Lie algebras $\mathfrak{g}$, where instead of a poset we get a directed graph with multiple edge colors, one for each simple root in the root system associated to $\mathfrak{g}$.  Donnelly asked whether support graphs for a semisimple Lie algebra $\mathfrak{g}$ always contain a combinatorial $\mathfrak{g}$-crystal. For $\mathfrak{g} = \mathfrak{sl}_2$ that simply means every support graph admits an SCD.  Since any poset admitting a strong $\mathfrak{sl}_2$-action contains an $\mathfrak{sl}_2$-graph in its support (just delete any edges unused by the action), we see that an affirmative answer to the $\mathfrak{sl}_2$ case of Donnelly's question is equivalent to our formulation of Conjecture \ref{theconjecture}.

\subsection{Multicommodity Flows and Fractional SCDs} The existence of an SCD in a ranked poset can be formulated as a multicommodity flow problem as follows:

\begin{enumerate}

\item Make the Hasse diagram directed by orienting all edges downward.

\smallskip

\item For each $n \ge 0$ add an auxiliary vertex $x_n$ with edges to every element of $P_n$, and an auxiliary $y_n$ with edges from every element of $P_{-n}$.

\smallskip

\item Make $x_n$ and $y_n$ a source and sink respectively for commodity $n$, set the demand between them to $|P_n| - |P_{n+2}|$.

\smallskip

\item Give all vertices in the original poset total capacity $1$. 

\end{enumerate}

\begin{center}
\includegraphics[scale=0.3]{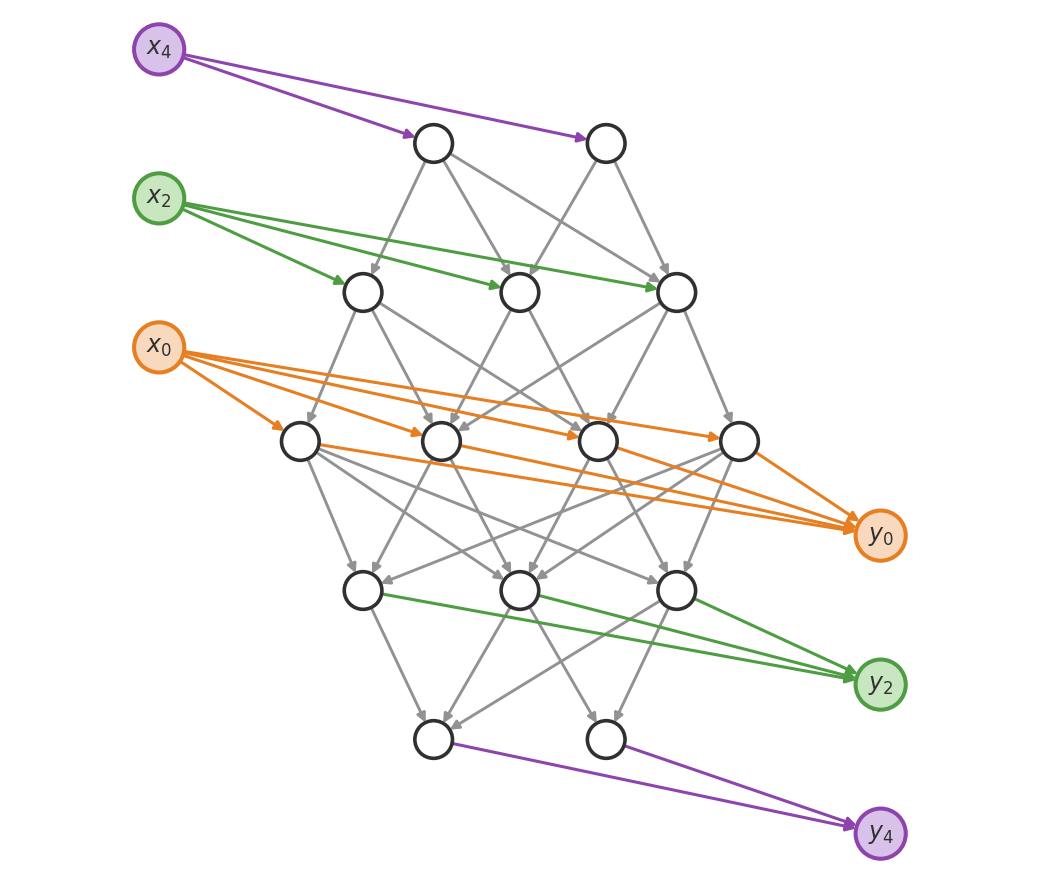}
\end{center}

The observation is that integer flow solutions to this flow problem exactly correspond to SCDs on the original poset. It's then natural to ask:  What do fractional flow solutions correspond to?

\medskip

\noindent \textbf{Definition:} Let $P$ be a ranked poset, and let $C(P)$ denote the set of symmetric saturated chains in $P$.  A \emph{fractional symmetric chain decomposition} is a function $\tau: C(P) \to \mathbb{R}_{\ge 0}$ such that for all $p \in P$
 $$\sum_{c \in C(P),  c \ni p} \tau(c) = 1.$$
 In other words, a fractional SCD of $P$ is a collection of weighted symmetric saturated chains in $P$ such that each vertex is covered by chains of total weight $1$. 
 
As far as we can tell this is the first time this notion is being defined explicitly. We will note though that the recent preprint \cite{DilliesJohnsonKowalska2025} uses coverings by weighted, but not necessarily symmetric, chains extensively. Moreover, in Appendix B they implicitly use fractional SCDs to give a new proof of Sperner's theorem. Our Proposition \ref{fracpeck} below is essentially a generalization of that argument. 
 
  Note that if all weights are rational we can renormalize and say that a fractional SCD is a collection of symmetric saturated chains covering every vertex of $P$ the same number of times.  We'll see that in fact if any fractional SCD exists then a rational one does too, so effectively we could have taken that as an alternative definition.

 \begin{proposition}
 
 For a ranked poset $P$ the associated multicommodity flow problem described above has a (not necessarily integral) solution if and only if $P$ admits a fractional SCD. 
 \end{proposition}
 
 \begin{proof}
 One direction is clear, given a fractional SCD we get a solution to the flow problem by flowing each commodity along the chains of the corresponding length according to the weight of that chain.  The fact that each vertex is covered by chains of weight $1$ translates to each vertex in the poset having one total unit of flow, satisfying the capacity.
 
 Going the other direction, starting with a flow we will construct a fractional SCD one weighted chain at a time.  Note that if there is flow from $x_i$ to $y_i$ we can find a path connecting them with nonzero flow.  This path intersects the poset in a symmetric chain with nonzero flow along each edge.  Here's the procedure:
 
 \begin{enumerate}
 
 \item Pick $x_i$ and $y_i$ with a nonzero amount of flow between them, and choose a path connecting them with nonzero flow along it. 
 
  \item Add the corresponding chain to the fractional SCD with weight equal to the smallest flow along any edges involved.
  
  \item Remove that amount of flow from each edge along the path, creating a new smaller flow.
  
  \item Repeat this process.
  
 \end{enumerate}
 
 At every step the flow of the chosen commodity becomes zero on at least one additional edge, so this process terminates after finitely many steps.
 \end{proof}

\noindent \textbf{Remark:} Note that this is not a bijection.  Every fractional SCD gives rise to a flow in a unique way, but many fractional SCDs can give rise to the same flow. The procedure above for converting a flow into a fractional SCD depends heavily on the choices of which commodity to look at and which path to choose at every step.

\subsubsection{An example of a fractional SCD} Every SCD is automatically a fractional SCD -- give the chains in the SCD weight $1$ and all others weight $0$.  Moreover, given two or more SCDs of the same poset, a convex combination of them will be a fractional SCD.  However not all fractional SCDs come from ordinary SCDs. Here is an example which cannot be obtained that way:

 \begin{center}
\includegraphics[scale=0.3]{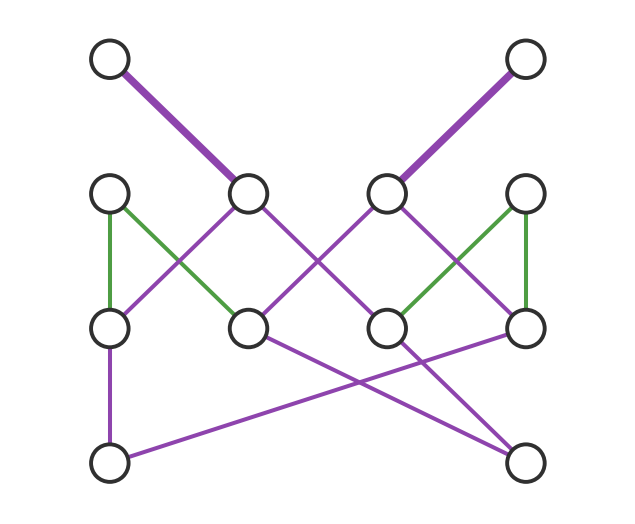}
\includegraphics[scale=0.3]{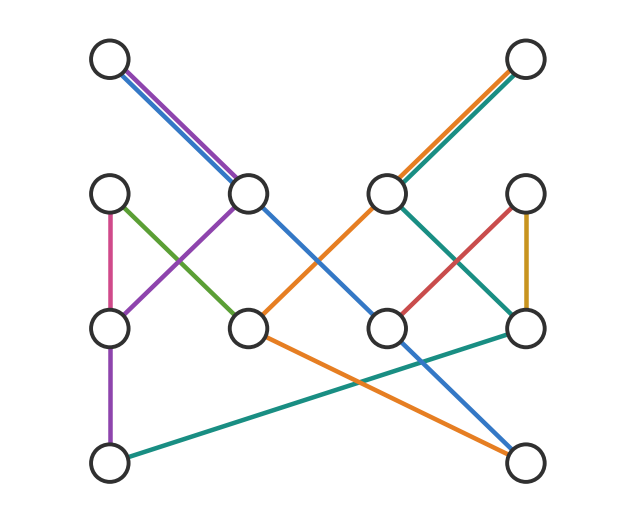}
\end{center}

On the left we have it depicted as a flow:  the bold edges at the top carry $1$ unit of flow, and the remaining edges carry $\frac{1}{2}$ a unit of flow.  We have a total of $2$ units of purple flow, which begin at rank $3$ and end at rank $-3$, and $2$ units of green flow going from rank $1$ to rank $-1$.   On the right we show it as a collection of $8$ symmetric saturated chains, which collectively cover each vertex exactly twice.   One can easily verify that this poset does not admit an ordinary SCD. We'll see in Example \ref{example2442} that this poset does not admit a strong $\mathfrak{sl}_2$-action either.

\subsubsection{Fractional SCD implies Peck}

Recall that if a poset $P$ admits an SCD then it is automatically Peck.  We have the following generalization:

\begin{proposition} \label{fracpeck}
If $P$ is a poset admitting a fractional SCD, then $P$ is Peck. 
\end{proposition}

\begin{proof}

First observe that $P$ must be rank-symmetric and unimodal, since the rank vector is the weighted sum of the rank vectors for the chains in the fractional SCD. We will check the strong Sperner condition directly.  Fix a fractional SCD $\tau: C(P) \to \mathbb{R}_{\ge 0}$ and let $S$ be a disjoint union of $k$ antichains. 

 Let $D_k$ denote the union of the largest $k$ ranks.  For any symmetric chain $c \in C(P)$ note that $|c \cap S| \le \min(|c|,k) =  |c \cap D_k|$.  Therefore:

$$|S| = \sum_{x \in S} \sum_{c \ni x} \tau(c) = \sum_{c \in C(P)} \tau(c)|S \cap c| \le \sum_{c \in C(P)} \tau(c)|c \cap D_k| = \sum_{y \in D_k} \sum_{c \ni y} \tau(c)  = |D_k|$$
\end{proof}

\subsubsection{The flow polytope} \label{flowpolytope} One advantage of this perspective is that fractional flows can naturally be identified with points inside a polytope, with SCDs corresponding to integer points.  For general references on integer programming and polyhedral realizations of multicommodity flow problems we refer to \cite{AhujaMagnantiOrlin1993} and \cite{Schrijver2003}. In our context, the \emph{flow polytope} is defined as follows:

\begin{itemize}

\item For every edge $e$ in the augmented poset and every commodity $c$ we have a variable $f_{e,c}$ representing the amount of that commodity flowing along that edge. 
$$ f_{e,c} \ge 0  \hspace{1cm} \text{and}  \hspace{1cm} \sum_c f_{e,c} \le 1  \text{ for all edges } e $$ 

\item For every vertex in the original poset the amount of flow in equals the amount of flow out for each commodity:

$$\sum_{e \to v} f_{e,c} = \sum_{v \to e} f_{e,c} \text{ for all vertices } v \text{ and commodities }  c$$

\item The total flow into each vertex in the original poset is $1$.

$$ \sum_c \sum_{e \to v} f_{e,c} = 1$$

\item The auxiliary vertices $x_n$ and $y_n$ are sources and sinks for commodity $c_n$ with appropriate demands $D_n$. 

$$ \sum_{x_n \to e} f_{e,c_n} = D_n  \hspace{1cm}  f_{e,c_m} = 0 \text{ if }  e  \text{ comes out of } x_n \text{ and } m \ne n$$

$$ \sum_{e \to y_n} f_{e,c_n} = D_n  \hspace{1.1cm}  f_{e,c_m} = 0 \text{ if }  e  \text{ goes in to } y_n \text{ and } m \ne n$$

\end{itemize}

In particular this means one can use linear programming and integer programming techniques to determine whether a poset admits a fractional or ordinary SCD.  For larger posets this was by far the most efficient way we found to test them, see Section \ref{SCDtesting} for more discussion. 

\subsubsection{Some questions about fractional SCDs} As far as we can tell this is the first time that this notion explicitly appears in the literature. While this isn't the main focus of this paper, we do think fractional SCDs are quite interesting so we'd like to include some questions about them we think are interesting: 

\medskip

\noindent \textbf{Open Problem:} \emph{Do the Young lattices $L(m,n)$, and more generally quotients $B_n / G$ of the Boolean lattice for $G \subseteq S_n$, always admit fractional SCDs?}

\smallskip

The analogous questions for ordinary SCDs are famously open, but fractional SCDs are more flexible so this could be an interesting intermediate result. Likely this could be verified computationally for more values of $m$ and $n$ too. 

\smallskip

\noindent \textbf{Open Problem:} \emph{Can the existence of a fractional SCD on a Peck poset be characterized by a max-flow-min-cut type criterion?}

\smallskip

In general for multicommodity flow problems we do not expect this to be true, but for this very restricted class of networks it seems conceivable that there is a nice characterization.

\smallskip 

\noindent \textbf{Open Problem:} \emph{For a poset which does admit a fractional SCD, can we bound the minimum denominator needed for a rational solution in terms of the size or number of ranks in the poset?}

\smallskip

Since the flow polytope is defined by integer equations, if any real point exists, so does a rational one. In all of the small examples we looked at, if there was a fractional SCD at all then there was one with denominator at most $2$, although we expect that may grow with the number of ranks.

\end{section}

\begin{section}{Combinatorics of strong $\mathfrak{sl}_2$-actions}\label{combinatorics}
In this section we will explore how the combinatorics of a poset can give us linear algebraic information about any strong $\mathfrak{sl}_2$-actions on the poset (and vice-versa).  Although perhaps a better title for this section might be ``how to tell when a Peck poset does not admit a strong  $\mathfrak{sl}_2$-action".

All examples in this section are Peck posets which do not admit symmetric chain decompositions, and we will be showing that they do not admit strong $\mathfrak{sl}_2$-actions.  Other than the last example in the section, these posets were all part of randomly generated batches, and then we selected ones to highlight particular proof ideas. 

By Proposition \ref{SCDimpliesstrong}, a poset not admitting a strong $\mathfrak{sl}_2$-action cannot admit a symmetric chain decomposition. Nevertheless, we encourage the reader to try to prove directly that the posets here do not admit symmetric chain decompositions. In particular, all of these proofs were inspired by similar (and usually simpler) arguments to show the posets do not admit a symmetric chain decomposition, and one can often reverse engineer that original argument. 

\subsection {Isotypic vectors}
We say that a nonzero vector $\vec{v}$ in an $\mathfrak{sl}_2$ representation $V$ is an \emph{$L(\lambda)$-isotypic vector} if $\vec{v}$ lies in the $L(\lambda)$-isotypic component.  If $\vec{v}$ is a weight vector, this is equivalent to $\vec{v}$ generating an irreducible subrepresentation of $V$ with highest-weight $\lambda$.  If we don't want to specify which component we are talking about we will just say ``$\vec{v}$ is an isotypic vector". 

A first observation is that if a poset $P$ admits a strong $\mathfrak{sl}_2$-action then we can tell that certain vectors must be isotypic vectors, even without finding the whole action.  The fact then that a nonzero vector cannot be in two different isotypic components gives us an easy obstruction. Here is a simple example of this type of reasoning:

\begin{example} \label{example1441} The poset  
\begin{center}
\includegraphics[scale=0.25]{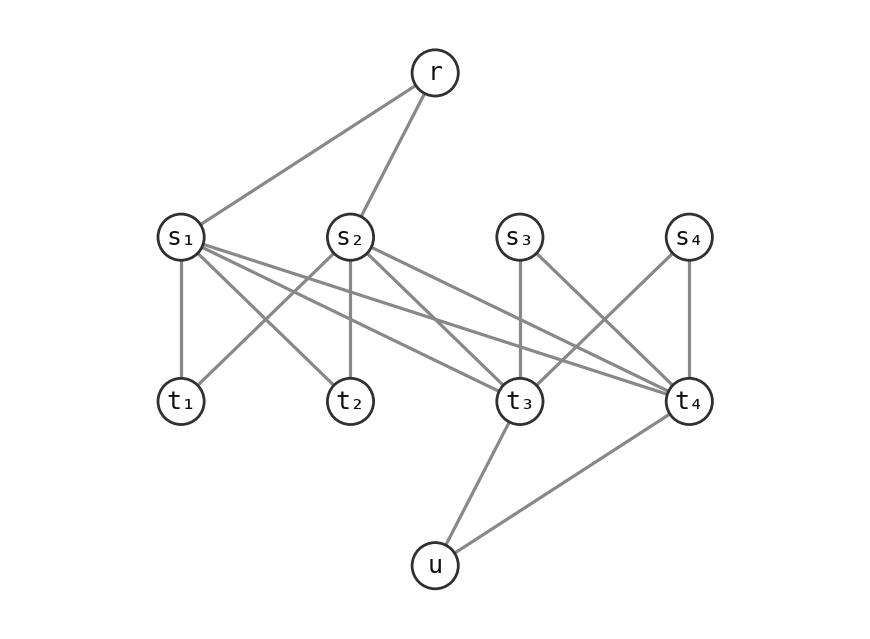}
\end{center}
 does not admit a strong $\mathfrak{sl}_2$-action
\end{example}

\begin{proof} Assume there is a strong $\mathfrak{sl}_2$-action on this poset.  The elements $t_1$ and $t_2$ have no downwards neighbors, so they are both lowest-weight vectors in the $L(1)$ component.  $E$ is injective between ranks $-1$ and $1$, and $t_1$ and $t_2$ have exactly two upwards neighbors altogether.  Hence, $\text{span}(E(t_1), E(t_2)) = \text{span}(s_1, s_2)$. Therefore $\text{span}(s_1, s_2)$ consists entirely of $L(1)$-isotypic vectors. 

On the other hand, $r$ has no upwards neighbors so it is a highest-weight $L(3)$-isotypic vector.  So then $F(r)$ is a nonzero $L(3)$-isotypic vector, which is a contradiction since $F(r) \in \text{span}(s_1, s_2)$ so we know it must be a $L(1)$-isotypic vector by the previous paragraph.
\end{proof}

\subsubsection{Poset elements as isotypic vectors}\label{isotypicelements}   Suppose $P$ is a poset which admits a strong $\mathfrak{sl}_2$-action. Here are some ways we can tell from the combinatorics of $P$ when certain elements must be isotypic vectors: 

\medskip

\begin{itemize}

\item \textbf{Highest-weight and lowest-weight vectors:}  If $\vec{v}$ is in rank $\lambda$ and $E(v) = 0$ then $\vec{v}$ is highest-weight vector, and therefore $L(\lambda)$-isotypic.  In particular if $p \in P$ has no upwards neighbors in $P$ then this must happen.   The story is similar for lowest-weight vectors. 

\medskip

\item \textbf{Unique neighbor propagation:} Suppose $\vec{v}$ is an $L(\lambda)$-isotypic vector in rank $k  > -\lambda$, then $F(\vec{v})$ is nonzero and $L(\lambda)$-isotypic.   So if $p \in P_k$ (with $k > -\lambda$) is known to be $L(\lambda)$-isotypic and $p$ has a unique neighbor $q$ in $P_{k-2}$, then $q$ is also $L(\lambda)$ isotypic.    Similarly we can propagate upwards along using $E$, with the opposite restriction on ranks.  

\medskip

\item \textbf{Hall-tight propagation:} More generally suppose $p_1, p_2, \dots, p_m \in P_k$ are$L(\lambda)$-isotypic with $k > -\lambda$. Suppose also that they collectively have exactly $m$ neighbors $q_1, q_2, \dots q_m$  in $P_{k-2}$. Then  $\text{span}(q_1, \dots q_m)  = \text{span}(F(p_1), \dots, F(p_m))$. Therefore, all of the $q_i$'s are $L(\lambda)$-isotypic too. Again, a similar rule applies in the other direction using $E$. 

\medskip

\item \textbf{Reverse propagation:} In general $F(\vec{v})$ being $L(\lambda)$-isotypic does not imply that $\vec{v}$ is also $L(\lambda)$-isotypic, because $\vec{v}$ could differ from the $L(\lambda)$-isotypic component by an element of the kernel of $F$.  However if $F$ is injective from rank $k$ to $k-2$ and $p \in P_k$ has all of its neighbors in $P_{k-2}$ known to be $L(\lambda)$-isotypic, then $p$ must also be $L(\lambda)$-isotypic.  Again, similar rules for $E$ apply.

\medskip

\item \textbf{Multi-step propagation:} We can also propagate using $F^n$ or $E^n$.  For example suppose $p \in P_k$ is a $L(\lambda)$-isotypic vector with $k > - \lambda + 2$,  and that $q \in P_{k-4}$ is its unique two-step downward neighbor. Then $q = F^2(p)$ up to scalar and therefore $q$ is also $L(\lambda)$-isotypic.  This can also be combined with Hall-tight or reverse propagation in suitable circumstances.

\medskip

\item \textbf{$EF$-eigenvectors:} A weight vector $\vec{v}$ is an eigenvector for $EF$ if and only if $v$ is an isotypic vector -- the eigenvalue will depend on the isotypic component as well as the rank.  Suppose  $p \in P_k$ has a unique neighbor $q \in P_{k-2}$ and $p$ is also the unique neighbor of $q$.  Then both $p$ and $q$ are $L(\lambda)$-isotypic vectors for some $\lambda$ (but we don't know which $\lambda$ a priori). 

\end{itemize}
 
\medskip

\subsubsection{Partial isotypic information:} Typically it is not the case that all of the poset elements themselves are forced to live in isotypic components.  More often we are only able to get partial information about the relationship between the poset elements and the isotypic components.  

\medskip

\begin{itemize}

\item \textbf{Isotypic vectors in spans:}  If $p \in P_k$ is $L(\lambda)$-isotypic with $k > -\lambda$,  then $F(p)$ is also $L(\lambda)$-isotypic.  Even if $p$ does not have a unique downwards neighbor, we can still conclude that there is a $L(\lambda)$-isotypic vector in the span of $p$'s downward neighbors. 

\medskip

\item \textbf{Sums of isotypic components:} If $p \in P_k$ has all of its downwards neighbors known to be $L(\lambda)$-isotypic, but $F$ is not injective from rank $k$ to $k-2$, we can't use the reverse propagation rule but we can still conclude that $p$ is in the direct sum of the $L(\lambda)$ and $L(k)$ components. 

\medskip

\item \textbf{Propagating partial information:} We can use the same methods for individual poset elements in isotypic components to propagate partial information about isotypic vectors in spans of poset elements and poset elements in sums of isotypic components.  For example if $p \in P_k$ is in the sum of two isotypic components,  and $q$ is the unique downward neighbor of $p$, then $q$ must also be in the sum of those components. 

\medskip

\item \textbf{Intersections and Unions:}  This partial information can sometimes be combined to get stronger results.  For example suppose $S$ and $T$ are subsets of $P_k$ where we know both $\text{span}(S)$ and $\text{span}(T)$ contain a $L(\lambda)$-isotypic vector. Then either $\text{span}(S \cup T)$ contains two linearly independent $L(\lambda)$-isotypic vectors or $\text{span}(S \cap T)$ contains a $L(\lambda)$-isotypic vector.   If say we additionally knew $L(\lambda)$ had multiplicity $1$, we could conclude we are in the latter case.

\end{itemize}

\subsubsection{Forbidden Configurations:} Once we know where the isotypic vectors live in the representation, certain combinatorial obstructions to strong $\mathfrak{sl}_2$-actions emerge.  Essentially these are contradictions one arrives at when trying to prove a given poset does not admit such an action.

\medskip

\begin{itemize}

\item \textbf{Multiple isotypic components:} The simplest obstruction is that a poset element is forced to be in two different isotypic components simultaneously.  This was what happened in Example \ref{smallestexample}.

\medskip

\item \textbf{Overcrowded spans:} More generally if $S \subseteq P_k$ is forced to contain more than $|S|$ linearly independent isotypic vectors in its span. In Example \ref{example1441} $\text{span}(s_1,s_2)$ was forced to contain two linearly independent $L(1)$-isotypic vectors, as well as one $L(3)$-isotypic vector. 

\medskip

\item \textbf{Too many isotypic vectors:}  If $L(\lambda)$ appears with multiplicity $m$,  then we can only have at most $m$ linearly independent $L(\lambda)$-isotypic vectors in any given rank.  So if say the multiplicity of $L(\lambda)$ is one and we know $p_1 \in P_k$ is $L(\lambda)$-isotypic, and we also know that there is a $L(\lambda)$-isotypic vector in $\text{span}(p_2, p_3)$ (also in rank $k$)  that is a contradiction. 

\medskip

\item \textbf{Rank conditions and Lindstr\"om–Gessel–Viennot:}  Suppose $S \subseteq P_k$ and $T \subseteq P_{k-2n}$ are such that the submatrix $F^n_{S,T}$ is known to have rank $m$.  Then the Lindstr\"om–Gessel–Viennot Lemma tells us there must be at least $m$ disjoint downwards paths connecting $S$ to $T$.   For example, suppose $\text{span}(S)$ and $\text{span}(T)$ are both known to contain all of the $L(\lambda)$-isotypic vectors in their respective ranks.  If there aren't enough disjoint paths connecting them,  we obtain a contradiction. We'll note that often this can be converted into one of the other contradictions listed above.

\end{itemize}

\subsection{Shrinking the support:} Given a poset $P$ with a strong $\mathfrak{sl}_2$-action, for $p \in P$ a priori we know that $F(p)$ is a linear combination of the downwards neighbors of $p$ (and similarly for $E(p)$ and upwards neighbors). However it could be that in fact $F(p)$ is a linear combination of some smaller subset of its downwards neighbors -- possibly just a single one or maybe even $F(p) = 0$. 

The next observation about the combinatorics of posets with a strong $\mathfrak{sl}_2$-action is that sometimes you can tell, just from the combinatorics of the poset, that any strong action on the poset cannot actually use certain edges in one direction or the other (or both).  Similarly, we can sometimes tell from the combinatorics that certain edge coefficients must be nonzero in any strong $\mathfrak{sl}_2$-action.  Here is an argument that uses this sort of reasoning:

\begin{example} \label{example24442} The poset  
\begin{center}
\includegraphics[scale=0.3]{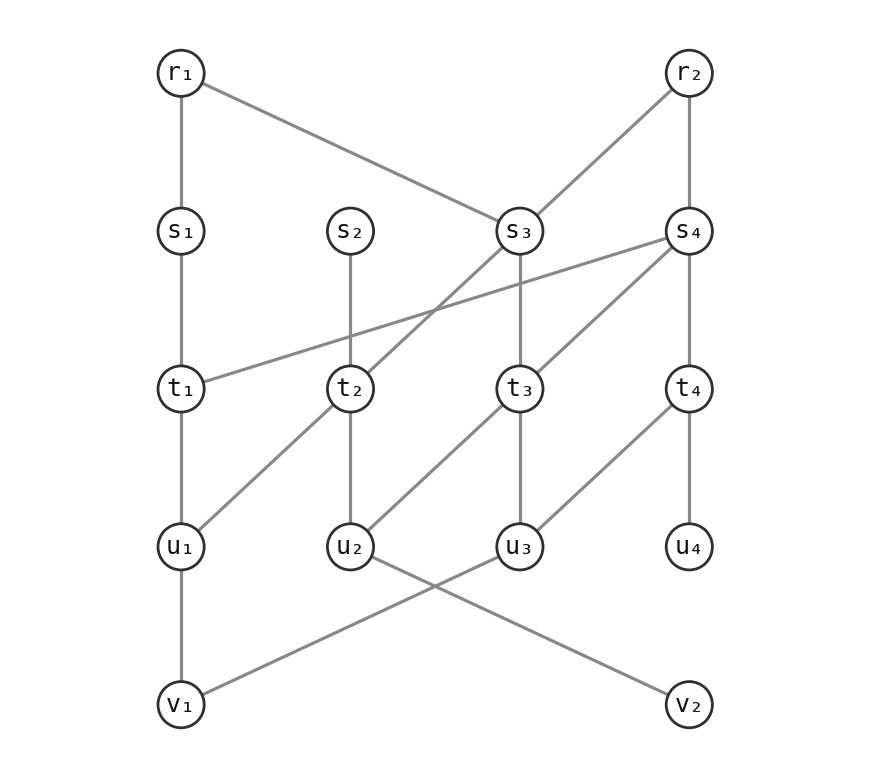}
\end{center}
 does not admit a strong $\mathfrak{sl}_2$-action.
\end{example}

\begin{proof} 

Assume this admits a strong $\mathfrak{sl}_2$-action. First note that $r_1, r_2, u_2, v_1,$ and $v_2$ are all $L(4)$-isotypic vectors, and that $s_2, s_4, t_2, t_4,$ and $u_4$ are all $L(2)$-isotypic vectors.  All of these follow from the highest/lowest-weight and the unique neighbor propagation deduction rules from the previous section.  

Now observe that $E(u_2)$ is in $\text{span}(t_2, t_3)$, but it cannot be a multiple of $t_2$ as they are in different isotypic components.  Therefore $E(u_2)$ must have a nonzero $t_3$ term.  This forces $F(t_2)$ not to have a $u_2$ term. Otherwise, $EF(t_2)$ would have a nonzero $t_3$ term but $EF(t_2) = 2t_2$ since $t_2$ is a $L(2)$-isotypic vector in rank $0$. 

\begin{center}
\includegraphics[scale=0.3]{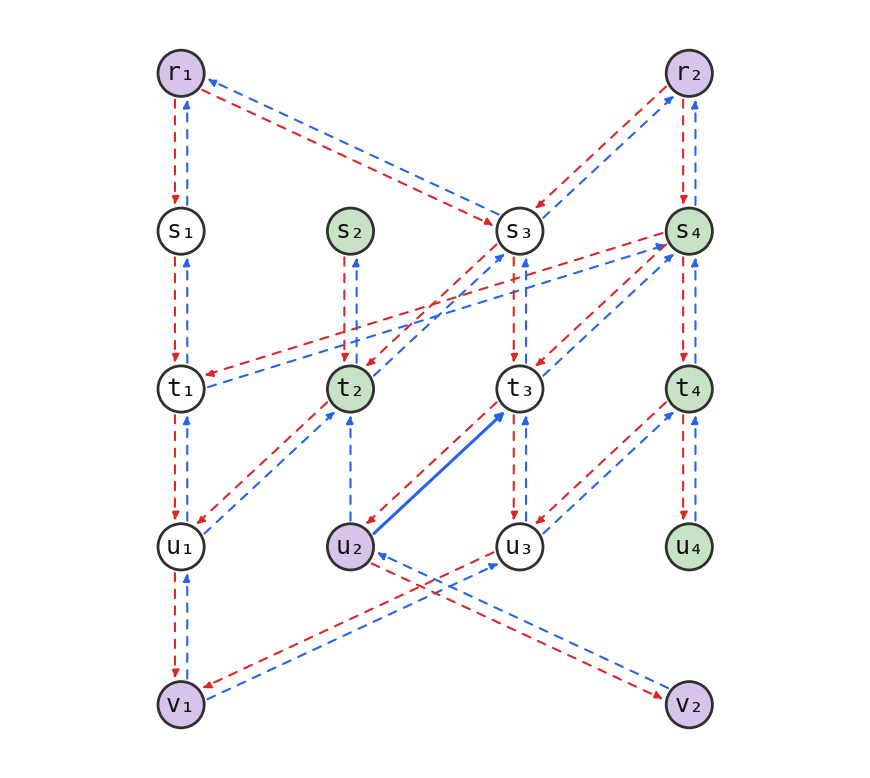}
\end{center}
The above picture is a depiction of what's known so far: The purple vertices are $L(4)$-isotypic vectors, and the green ones are $L(2)$-isotypic.  The solid arrow from $u_2$ to $t_3$ indicates the coefficient that we know must be nonzero, and the missing red arrow between $t_2$ and $u_2$ indicates the coefficient that we now know must be zero.  The remaining edges are left dotted as a priori they could be either zero or nonzero. 

From here we arrive at a contradiction since we see that $F(t_1)$ and $F(t_2)$ must both be multiples of $u_1$, but that can't happen since $F$ is injective between ranks $0$ and $-2$ in the representation $L(4)^2 \oplus L(2)^2$. 
\end{proof}

\subsubsection{Support deductions:} Suppose $P$ is a poset admitting a strong $\mathfrak{sl}_2$-action.  Here are some situations where we can make these sorts of deductions:

\begin{itemize}

\item \textbf{Highest and lowest-weight vectors:} Suppose because of one of our isotypic vector propagation rules we know that $p \in P_k$ is a $k$-isotypic vector.  Well then $E(p) = 0$ so we know that the $E$-coefficients along any upwards edge out of $p$ 
must all be $0$. Similarly for downward edges out of $p \in P_{-k}$. 

\medskip

\item \textbf{Unique neighbors:} Suppose $p \in P_k$ has a unique neighbor in $P_{k-2}$ and that $F$ is known to be injective between those ranks.  Then the $F$-coefficient along that edge must be nonzero.  This can also apply if $p$ has multiple neighbors in $P_{k-2}$ but we have already determined that the other $F$-coefficients along those edges must be zero. 

\medskip

\item \textbf{Unique non-isotypic neighbors:} Suppose $p \in P_k$ is known to be $L(\lambda)$-isotypic, and that all but one of its neighbors in $P_{k-2}$ are known to be in other isotypic components.  Then the $F$-coefficient along the edge to the last vertex must be nonzero. 

\medskip

\item \textbf{Unique isotypic neighbors:} Suppose $p \in P_k$ is known to be $L(\lambda)$-isotypic, it has a neighbor $q \in P_{k-2}$ which is also $L(\lambda)$-isotypic.  Suppose additionally that the span of all $p$'s neighbors in $P_{k-2}$ is known to have exactly one $L(\lambda)$-isotypic vector up to scalars.  Then $F(p)$ must be a scalar multiple of $q$.  Consequently, the $F$-coefficient from $p$ to $q$ is nonzero and all other $F$-coefficients out of $p$ are zero. 

\medskip

\item \textbf{Unique matchings in Hall-tight neighborhoods:} More generally suppose $S \subset P_k$ has neighborhood $T \subset P_{k-2}$ with $|S| = |T|$ and $F$ is injective between those ranks.   At least one perfect matching between $S$ and $T$ must have all nonzero $F$-coefficients as the determinant of the submatrix $F_{S,T}$ is nonzero.  If there is only one such matching, or if certain edges are contained in all such matchings, we can conclude those edges have nonzero $F$-coefficients. 

\medskip

\item \textbf{Down-up paths from isotypic vectors:}  Although $E$ and $F$ are not inverses to one another, for an isotypic weight vector $\vec{v}$ we have $EF(\vec{v}) = c\vec{v}$ for some explicit constant $c$.  This constant depends on the weight and isotypic component, and is nonzero provided $\vec{v}$ is not a lowest-weight vector.  In particular if $p \in P_k$ is an isotypic but not lowest-weight vector, and there is a unique two-step down-up path to some other $q \in P_k$ then one of the coefficients along that path must be zero. If one is already known to be nonzero, then the other must be zero.  We used this in Example \ref{example24442}.

\medskip

\item \textbf{Two-step paths from general vectors:}  For a non-isotypic vector in rank $k$ we don't have $EF(\vec{v}) = c\vec{v}$ but we do have $EF(\vec{v}) - FE(\vec{v}) = H (\vec{v}) = k \vec{v}$.  In particular if $p \in P_k$ is any element and there is a unique two-step path (both down-up and up-down) to some other $q \in P_k$ then one of the coefficients along that path must be zero.  Again, if one of those is already known to be nonzero the other must be zero. 

\medskip

\item \textbf{Longer paths and higher order relations:} Suppose $\vec{v}$ is an isotypic weight vector and $F^2(\vec{v}) \ne 0$, then $EF^2(\vec{v}) = c F(\vec{v})$ for an explicit constant $c$ depending on the weight and isotypic component.  If $p \in P_k$ is known to be $L(\lambda)$-isotypic with $k > -\lambda +2$ then for any $q \in P_{k-2}$ which is a neighbor of $p$ but not reachable by a down-down-up path from $p$, the $F$-coefficient from $p$ to $q$ must be zero.   Similar arguments can be made for different kinds of paths and $E^mF^n$ or $F^n E^m$. 

\end{itemize}

\subsubsection{Forbidden Configurations:} Once we have more information about which edge coefficients are forced to be zero or nonzero, there are new ways to see a strong $\mathfrak{sl}_2$-action cannot exist.

\medskip

\begin{itemize}

\item \textbf{Simultaneously zero and nonzero:} We may find that for one reason a coefficient is forced to be zero, but for another reason it is forced to be nonzero.  Obviously this is a contradiction. 

\medskip

\item \textbf{Breaking injectivity or surjectivity:}  For $\mathfrak{sl}_2$ representations we know exactly when the $E$ and $F$ operators are injective or surjective between adjacent ranks.  Since $P$ is Peck, the necessary Hall-matching conditions on the edges will all initially be satisfied automatically. However setting additional coefficients to zero might break these conditions, and making it so $E$ or $F$ cannot be injective/surjective in certain ranks. 

\end{itemize}

\subsection{Subrepresentations and quotients} Let $P$ be a poset with a strong $\mathfrak{sl}_2$-action.  Suppose $S \subseteq P$ is such that all of the $E$- and $F$-edge coefficients from elements of $S$ to elements of $S^c$ are zero (the coefficients from $S^c$ to $S$ may still be nonzero).   It follows that that $\text{span}(S)$ forms a submodule of $\mathbb{C}P$.  Thus $S$ itself inherits a strong $\mathfrak{sl}_2$-action.   Moreover though, the image of $S^c$ forms a basis for the quotient representation $\mathbb{C}P / \text{span}(S)$ with the same edge coefficients between them as in $P$. 

It turns out that sometimes you can tell from the combinatorics of $P$ that this must happen for any strong $\mathfrak{sl}_2$-action on $P$.  Here is an example of how we might make such an argument: 

\begin{example} \label{example3553} The poset  
\begin{center}
\includegraphics[scale=0.3]{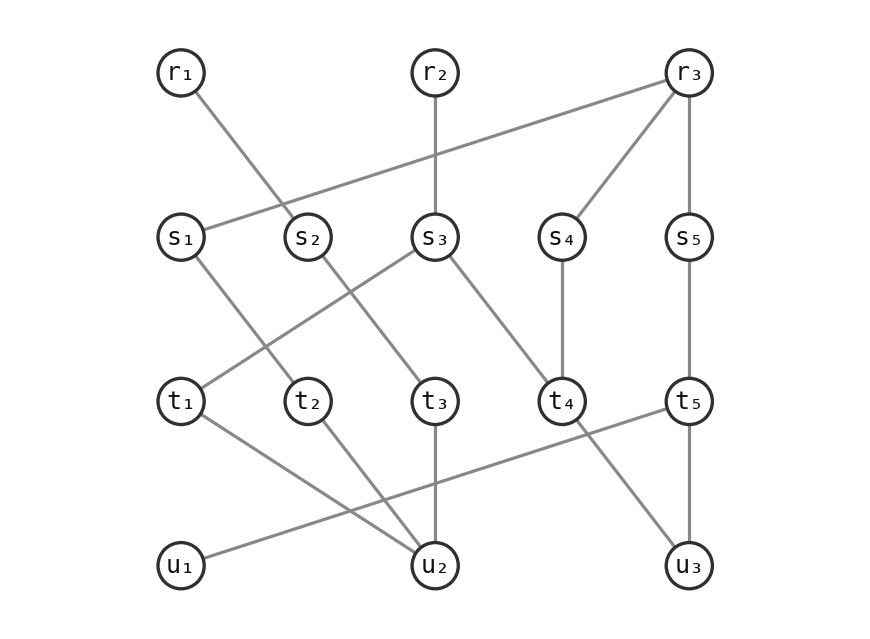}
\end{center}
 does not admit a strong $\mathfrak{sl}_2$-action.
\end{example}

\begin{proof} Suppose we have a strong $\mathfrak{sl}_2$-action on this poset.  The element $u_1$ is a lowest-weight $L(3)$-isotypic vector.  Unique neighbor propagation tells us $t_5$, $s_5$, and $r_3$ are all also $L(3)$-isotypic vectors.  We can also then conclude using the up-down path arguments that the $F$-edge coefficients from $r_3$ to $s_1$, from $r_3$ to $s_4$, and from $t_5$ to $u_3$ are all zero.   In particular we see that $r_3, s_5, t_5, u_1$ must span a subrepresentation. 

\begin{center}
\includegraphics[scale=0.25]{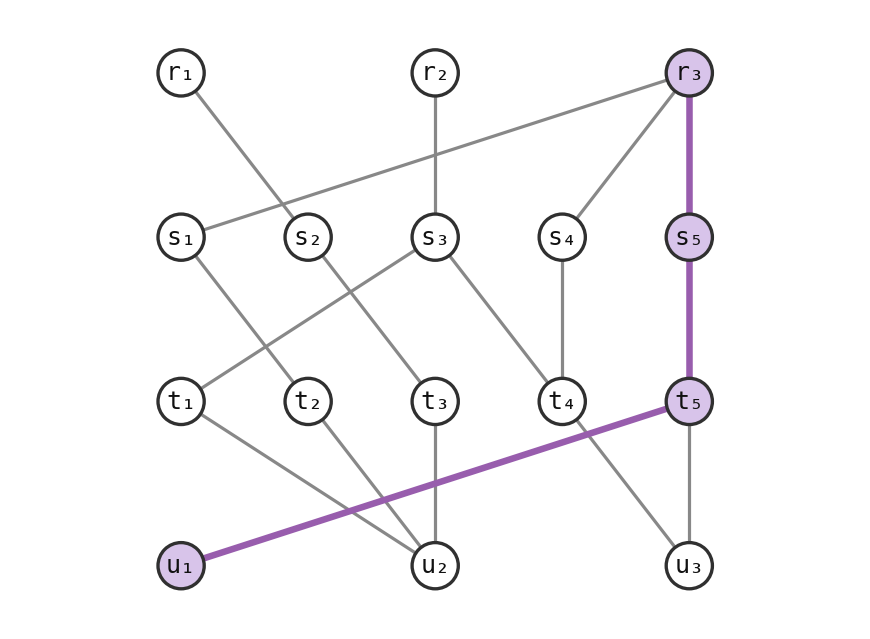}  \includegraphics[scale=0.25]{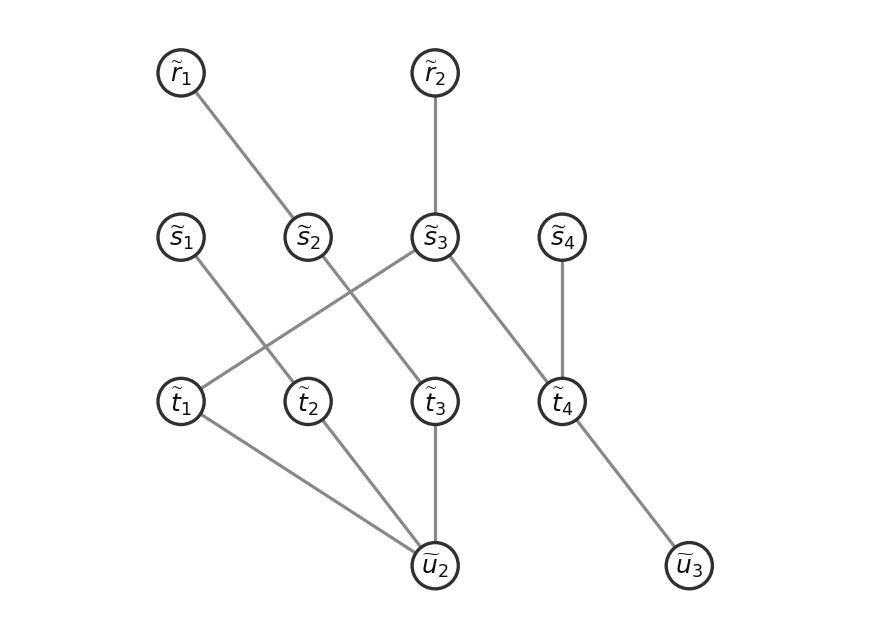}
\end{center}

Therefore if we remove these vertices the images of the remaining vectors in the quotient representation naturally define a strong $\mathfrak{sl}_2$-action on the complement.  However in this new smaller poset we can easily spot a familiar contradiction:  $\tilde{t}_4$ is forced to be both $L(1)$-isotypic and $L(3)$-isotypic. 
\end{proof}

Note that we can also go the other direction.  Given strong $\mathfrak{sl}_2$-actions on $S$ and $S^c$ we can define one on $P$ by setting all coefficients between $S$ and $S^c$, in both directions, to zero.   This gives us a recursive, divide and conquer approach to determining when strong $\mathfrak{sl}_2$-actions exist: Assuming you can show that all $E$- and $F$-coefficients from $S$ to $S^c$ must be zero then  a strong $\mathfrak{sl}_2$-action on $P$ exists if and only if such actions exist on $S$ and $S^c$. 

\subsection{Branching arguments}  When proving posets do not admit symmetric chain decompositions, inevitably one needs to start branching arguments  saying things like  ``The chain through vertex $A$ must connect to either $B$ or $C$, if it connects to $B$..."   or ``Either $D$ is the bottom of its chain, or it connects to $E$...".

Even though strong $\mathfrak{sl}_2$-actions are more flexible than symmetric chain decompositions in that $E$ and $F$ can be supported along multiple edges, there are still times where these sorts of binary choices occur and we need to use similar branching arguments.  Here is an example of such an argument: 

\begin{example} \label{example2442} The poset  
\begin{center}
\includegraphics[scale=0.3]{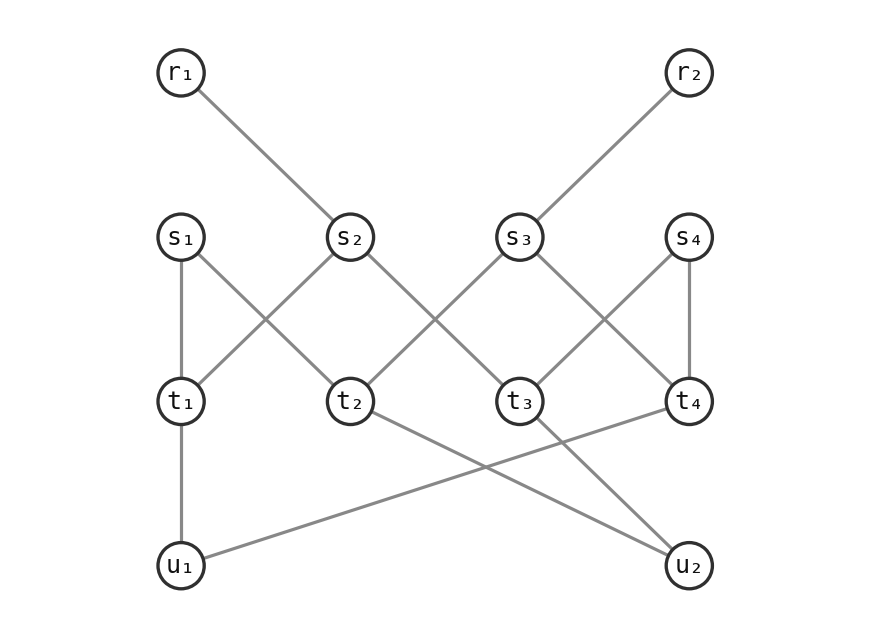}
\end{center}
 does not admit a strong $\mathfrak{sl}_2$-action.
\end{example}

\begin{proof} 
Assume we have a strong $\mathfrak{sl}_2$-action.  We can see immediately that $r_1,r_2, s_2,$  $s_3, u_1,$ and $u_2$ are $L(3)$-isotypic,  and that $s_1$ and $s_4$ are $L(1)$-isotypic.  We then know by injectivity that at least one of the $F$-coefficients along the edges down from $s_1$ must be nonzero so let's branch based on that.


\medskip

\noindent \textbf{Case 1:}  $F(s_1)$ has a nonzero $t_1$ term. 

\begin{center}
\includegraphics[scale=0.25]{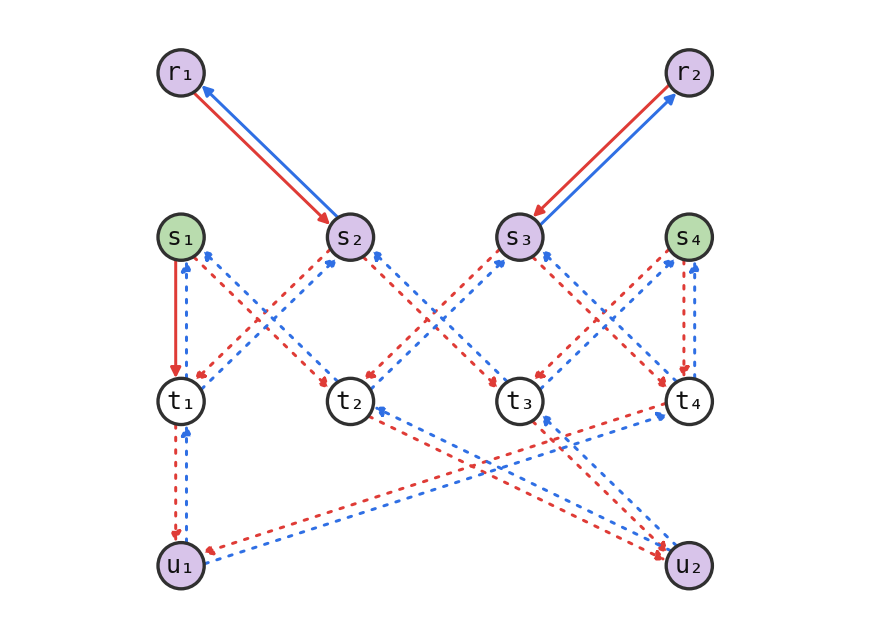}
\includegraphics[scale=0.25]{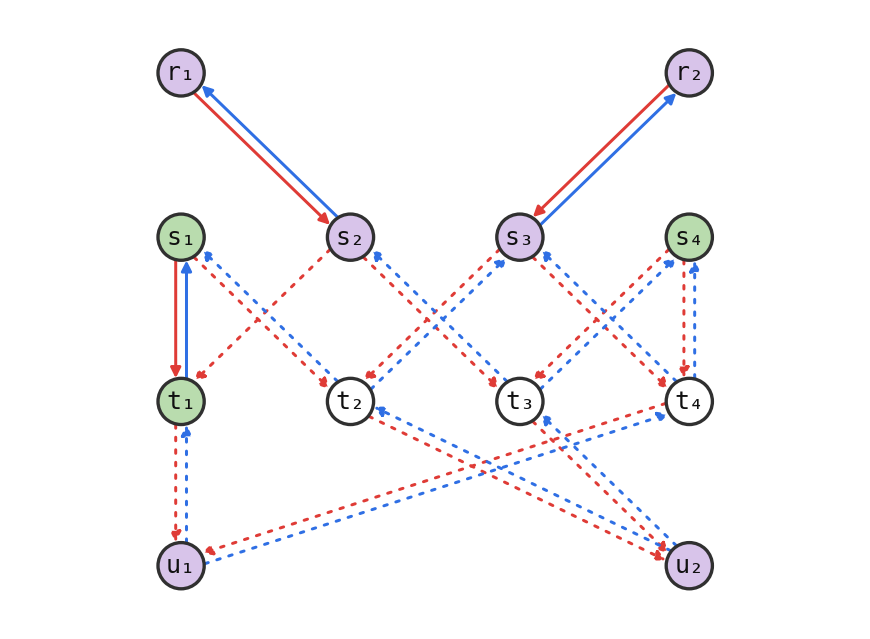}
\includegraphics[scale=0.25]{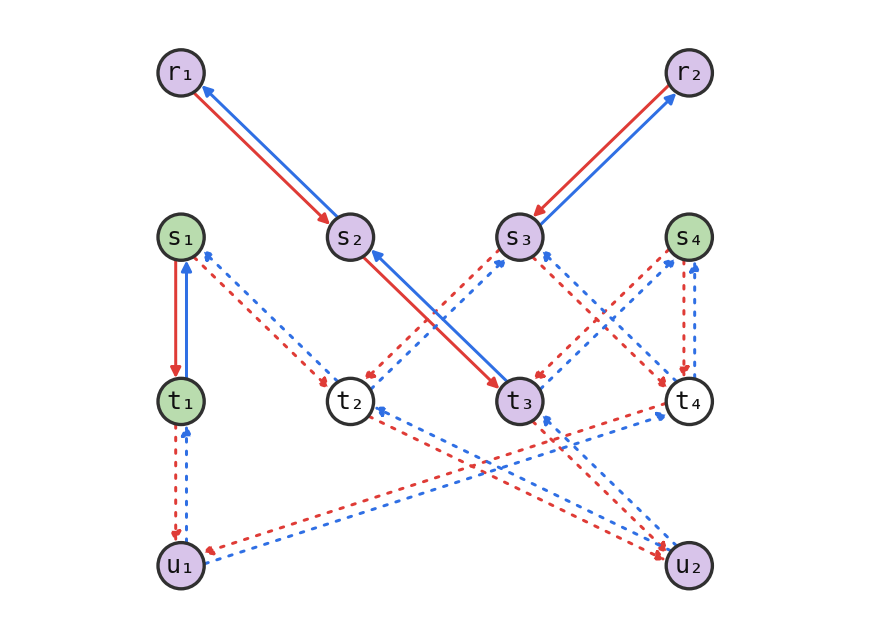}
\includegraphics[scale=0.25]{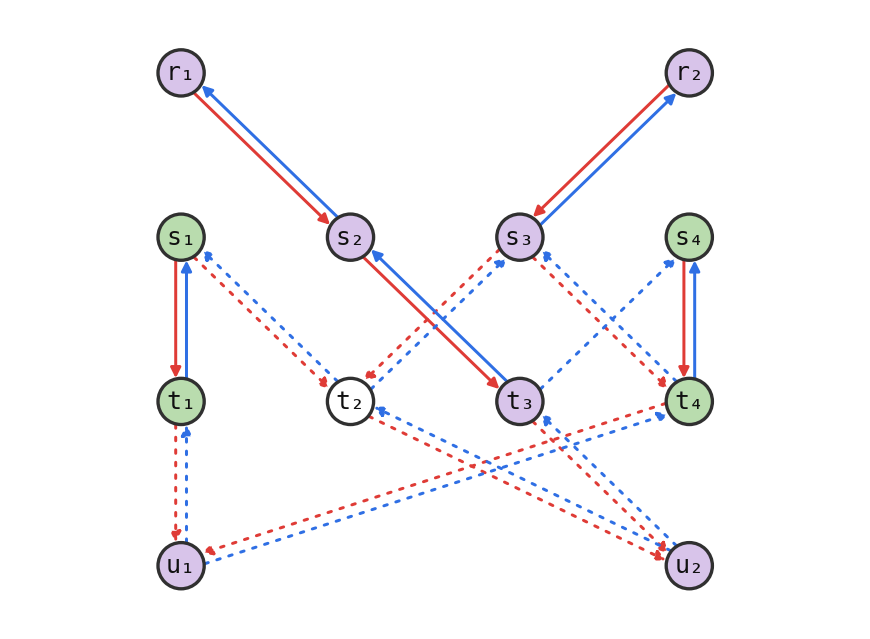}
\end{center}
\vspace{-.5cm} This causes a chain reaction in the middle layer:
\begin{itemize}
\item  First we see that since $EF(s_1) = s_1$ we must have $E_{t_1,s_1} \ne 0$ and that $E_{t_1,s_2} = 0$, and therefore we see that $t_1$ is $1$ isotypic as $E(t_1) = cs_1$. 
\vspace{.2cm}

\item  Next since $EF(s_2) = 4s_2$ we must have $F_{s_2, t_1} = 0$ as well as $F_{s_2, t_3} \ne 0$ and $E_{t_3, s_2} \ne 0$.  We see $F(s_2) = ct_3$ so $t_3$ is $L(3)$-isotypic.  
\vspace{.2cm}

\item Applying the same reasoning one more time, $EF(s_4) = s_4$ we must have $F_{s_4, t_3} = 0$ as well as $F_{s_4, t_4} \ne 0$ and $E_{t_4, s_4} \ne 0$.  We see $F(s_4) = ct_4$ so $t_4$ is $L(1)$-isotypic.

\end{itemize}
We arrive at a contradiction, since $\text{span}(t_1, t_4)$ consists entirely of $L(1)$-isotypic vectors but $E(u_1) \in \text{span}(t_1, t_4)$ is $L(3)$-isotypic. \\

\noindent \textbf{Case 2:}  $F(s_1)$ has a nonzero $t_2$ term.  \\

A completely analogous chain reaction happens, this time supported on the complementary set of edges on the middle row, and an analogous contradiction is arrived at, this time involving $t_2, t_3,$ and $u_2$. Since both cases lead to a contradiction we therefore conclude there cannot be a strong $\mathfrak{sl}_2$-action. \end{proof}

\subsubsection{Things to branch on:} Here are a few situations where branching arguments have proven useful:

\begin{itemize}

\item \textbf{Two neighbors:} In Example \ref{example2442} the element $s_1$ had two downwards neighbors,  and we knew at least one had to appear with nonzero coefficient in $F(s_1)$.  This seems to be particularly useful when the source vertex is isotypic, and especially if there are other nearby elements with two neighbors to cause a chain reaction like in that example. 

\medskip

\item \textbf{Two isotypic components:} Suppose we know an element $p \in P$ is in the direct sum of the $k$-isotypic component and the $j$-isotypic component.  It is sometimes useful to condition on whether or not $p$ actually has a $j$-isotypic component or not.  This is particularly useful if $p$ is in rank $k$ (or similar for $-k$), so then either $p$ is a highest-weight vector or $E(p)$ is nonzero and $j$-isotypic, either of which can be useful information. 

\medskip

\item \textbf{Isotypic vectors in a span:} Suppose $S \subseteq P_k$ is known to contain both $F(p)$ and $E(q)$ where $p$ and $q$ are $L(\lambda)$-isotypic vectors in ranks $k+2$ and $k-2$ respectively.   Well, either $F(p) = E(q)$ up to a scalar, or they are linearly independent.   Either of these options provides new and potentially useful information. This is particularly useful if $S$ is small or if it is known to contain other isotypic vectors.

\end{itemize}

\subsection{Without loss of generality arguments} Sometimes one can show ``If any strong $\mathfrak{sl}_2$-action exists on $P$,  then there must exist an action such that [insert additional property] holds".  So without loss of generality we may assume the additional condition holds. This can in some circumstances simplify things greatly.  Here is an example of such an argument:

\begin{example} \label{example235532} The poset  
\begin{center}
\includegraphics[scale=0.3]{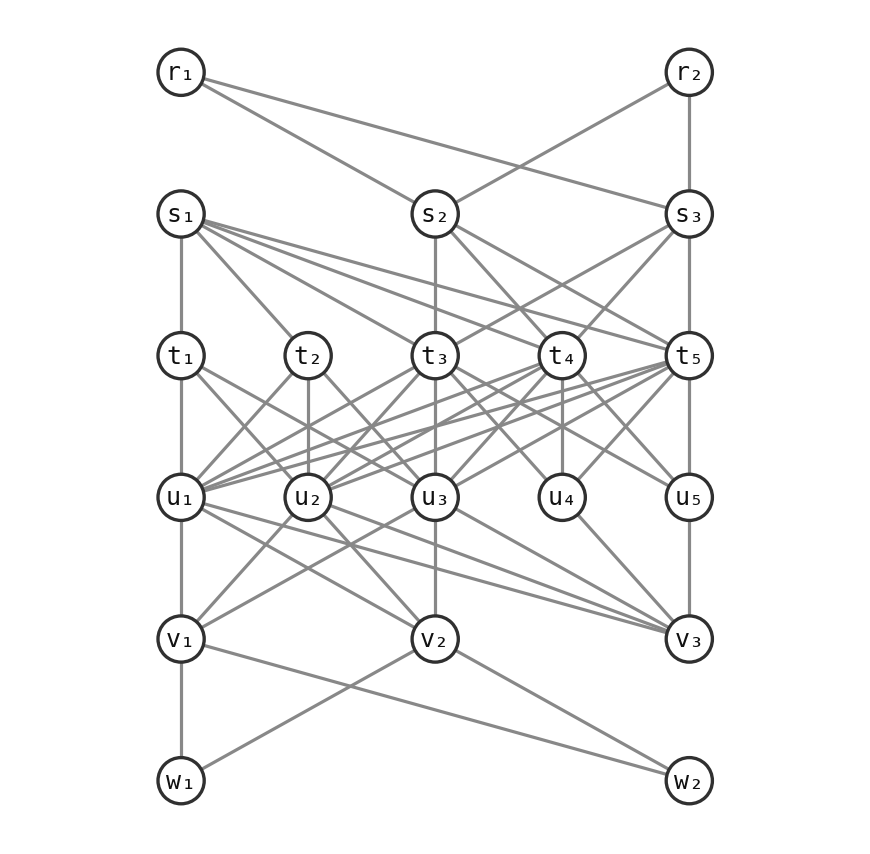}
\end{center}
 does not admit a strong $\mathfrak{sl}_2$-action.
\end{example}

\begin{proof}
First observe that  $r_1$, $r_2$, and $s_1$ are all highest-weight vectors.  Moreover, $\text{span}(F(r_1),F(r_2)) = \text{span}(s_2, s_3)$ so $s_2$ and $s_3$ are $L(5)$-isotypic.  This poset has $180^\circ$ rotational symmetry, so analogous claims hold on the bottom ranks as well. 
 
There is actually much more symmetry here than just the rotational symmetry.  There are a number of pairs of vertices that have the exact same neighbors.  Here is a new picture with the isotypic data and symmetry highlighted: 

\begin{center}
\includegraphics[scale=0.3]{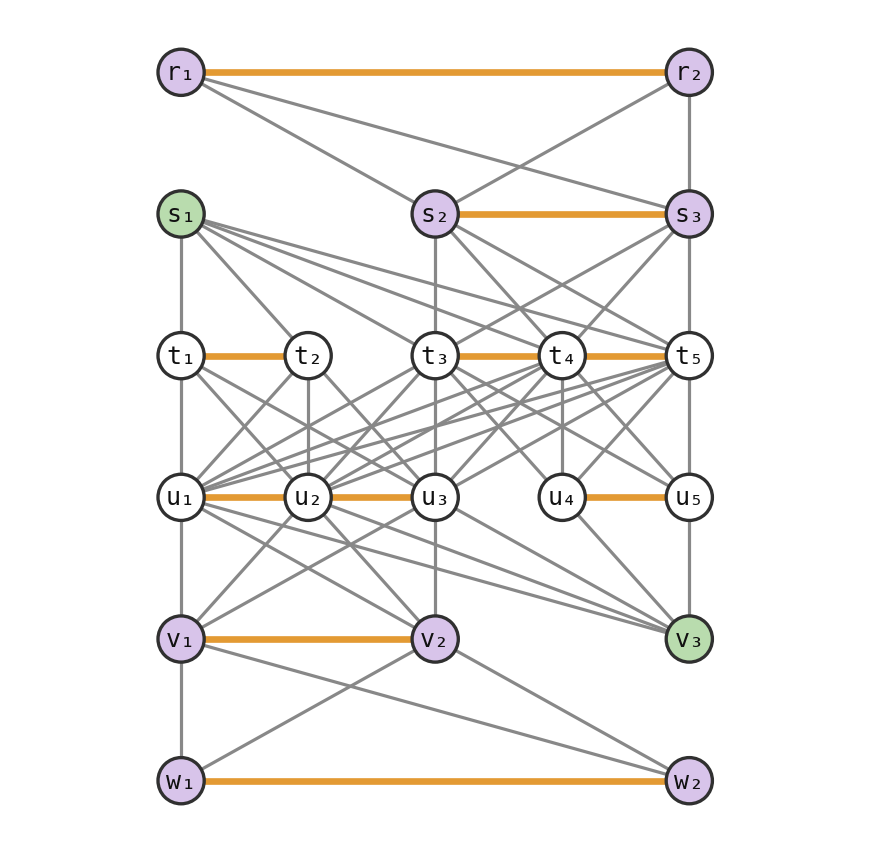}
\end{center}

Here is the key observation: Suppose we have a basis for $L(5)^2 \oplus L(3) \oplus L(1)^2$ with support graph contained in this poset. For each of these groups of ``twin" vertices (connected in orange) we may replace those basis vectors with any other basis for the same span and the support graph will still be contained in the poset. 

So for example since $\text{span}(t_3, t_4, t_5)$ contains $F(s_2)$ and $F(s_3)$,  we may assume without loss of generality that $t_3$ and $t_4$ are $L(5)$-isotypic.   Using the $180^\circ$ rotational symmetry again we can make a similar argument to assume $u_2$ and $u_3$ are $L(5)$-isotypic.   Once we make those assumptions we have found the entire $L(5)$-isotypic component and can quotient it out: 

\begin{center}
\includegraphics[scale=0.25]{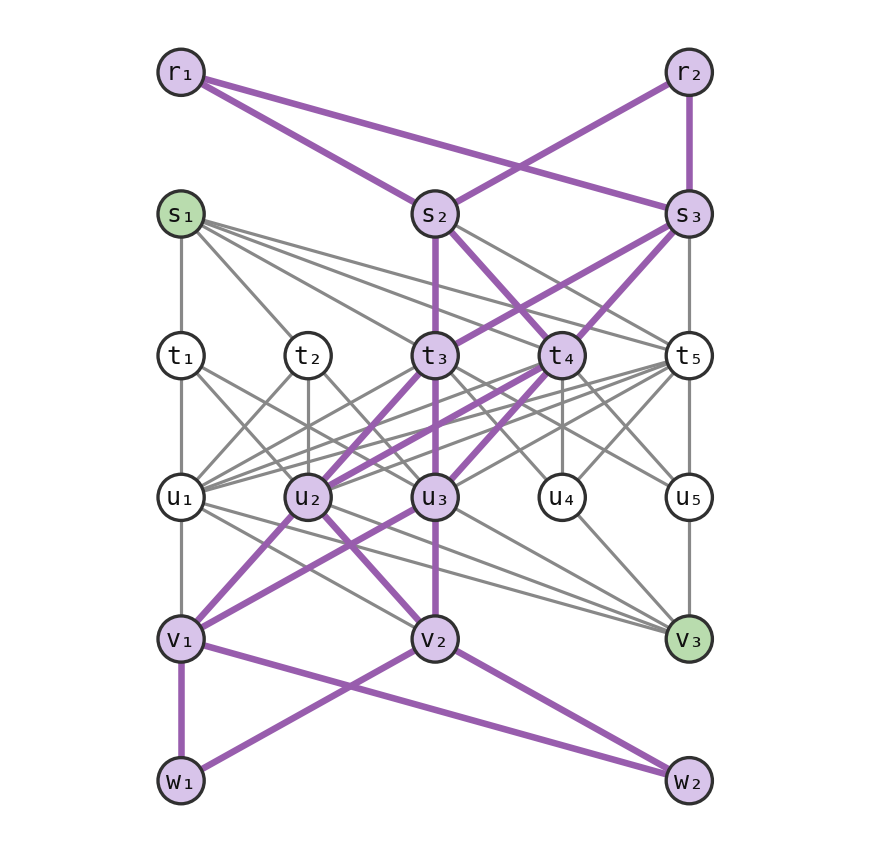}
\includegraphics[scale=0.25]{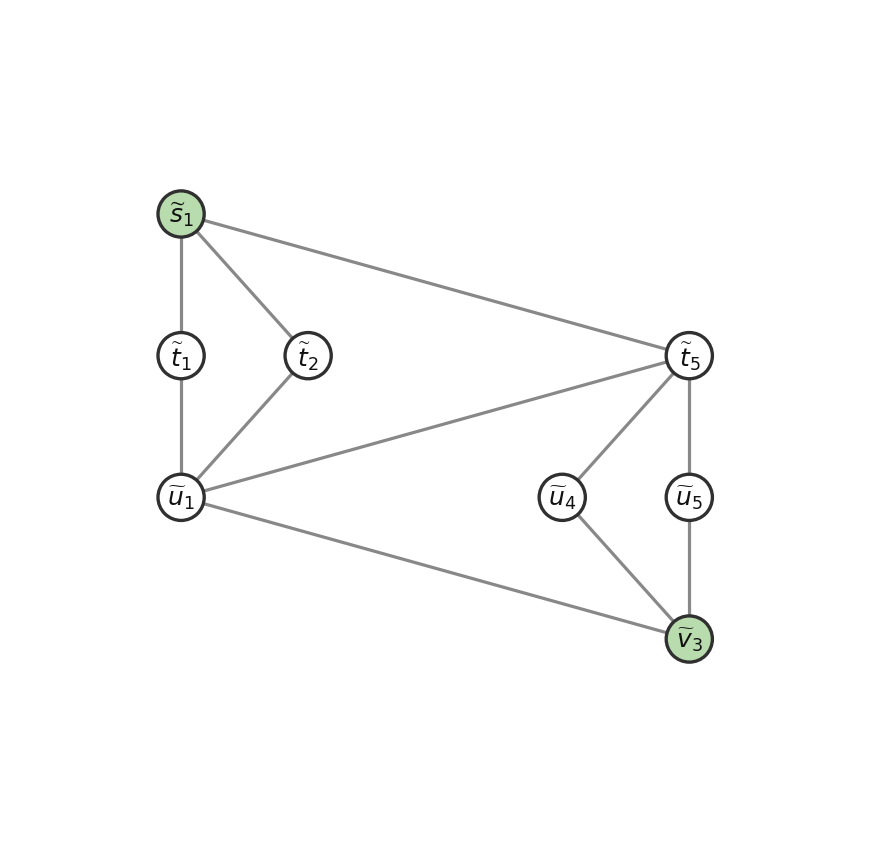}
\end{center}

In this quotient we can see that $F(\tilde{t}_1) = cF(\tilde{t}_2)$, which is a contradiction as $F$ needs to be injective between rank $1$ and $-1$.  \end{proof}

\noindent \textbf{Remark:} The $180^\circ$ rotational symmetry in the above example is fairly rare. It was just convenient as an example to simplify the exposition.   However we'll note that having some combinatorial symmetry coming from ``twin" vertices empirically seems fairly common among SCD-critical Peck posets. 

\subsection{The weak Bruhat order for $H_3$}  Finally, we'll close this section out with an example of a Peck poset which does not admit an SCD which occurs ``in the wild".  Leclerc proved in \cite{leclerc94} that the weak Bruhat order for the Coxeter group of type $H_3$ is Peck but does not admit an SCD.  We have the following strengthening:

\begin{theorem}
The weak Bruhat order for the Coxeter group of type $H_3$ does not admit a strong $\mathfrak{sl}_2$-action. 
\end{theorem}

\begin{proof}

Suppose a strong $\mathfrak{sl}_2$-action exists.  Mimicking Leclerc's argument (and reproducing his diagram), here are the middle four ranks poset plus a few relevant elements from the ranks above and below them: 

\begin{center}
\includegraphics[scale=0.3]{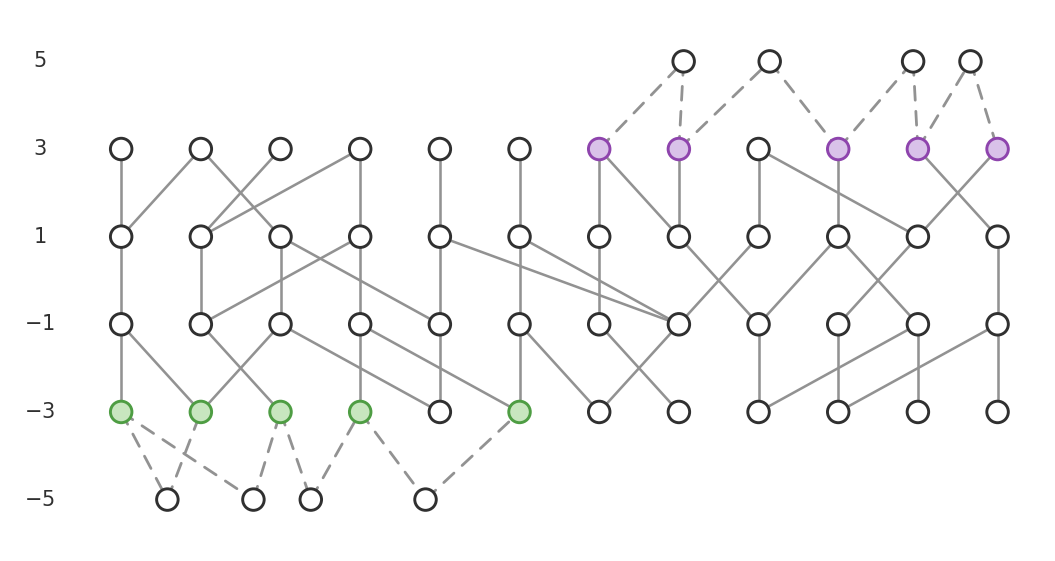}
\end{center}

Let $S$ denote the set of $5$ vertices in rank $3$ highlighted in purple.  Collectively the elements of $S$ only have $4$ neighbors in rank $5$.  Therefore, $\text{span}(S)$ must contain a highest-weight vector $\vec{v}$ of weight $3$. Similarly let $T$ denote the set of $5$ vertices in rank $-3$ highlighted in green.  Collectively the elements of $T$ only have $4$ neighbors in rank $-5$.  Therefore, $\text{span}(T)$ must contain a lowest-weight vector $\vec{w}$. 

In this poset $|P_3| - |P_5| = 1$. Thus, $\vec{v}$ and $\vec{w}$ are the unique $L(3)$-isotypic vectors in their respective ranks, up to scaling.  Consequently, $E^3(\vec{w}) = \lambda \vec{v}$ for some nonzero scalar $\lambda$.  Therefore $E^3_{S,T}$ has rank $\ge 1$ so by the  Lindstr\"om–Gessel–Viennot Lemma  (the combinatorial version of the Cauchy-Binet Theorem, see Section \ref{graphdictionary}) there must be at least one upward path connecting $T$ to $S$. However by inspection we see no such path exists, a contradiction.
\end{proof}


\end{section}

\begin{section}{Special cases and partial results} \label{specialcases}
In this section we'll prove some special cases and partial results toward Conjecture \ref{theconjecture}.

\subsection{A single isotypic component}

We'll prove Conjecture \ref{theconjecture} in the case where all of the ranks of $P$ have the same size.  Such posets correspond to $\mathfrak{sl}_2$-representations $L(\lambda)^k$ with a single isotypic component. In fact we will only need an $\mathfrak{sl}_2$-action in the sense of Proctor for this.

\begin{theorem}\label{onecomponent}
Let $P$ be a ranked poset where all of the non-empty ranks have the same size. If $P$ admits an $\mathfrak{sl}_2$-action in the sense of Proctor then $P$ admits an SCD. 
\end{theorem}

\begin{proof}
In this case, the $E$ operator is injective between every pair of adjacent ranks, so on each layer we can pick a perfect matching.  Those matchings combine to form an SCD.  (Alternatively, we can observe that the outermost family of chains guaranteed by Griggs' property $T$ form an SCD in this case)
\end{proof}

Indeed in this case we can get away only using edges corresponding to nonzero $E$ coefficients.  We'll note that if we have a strong $\mathfrak{sl}_2$-action, we could also use all edges corresponding to nonzero $F$-coefficients instead.  Here are a couple more easy variations of this theorem:

\begin{corollary}\label{almostisotypic}
Let $P$ be an RSU poset where all of the non-empty ranks have the same size except for $P_0$. If $P$ admits an $\mathfrak{sl}_2$-action in the sense of Proctor then $P$ admits an SCD. 
\end{corollary}

\begin{proof}
The same argument works: The chains connecting the outer ranks will miss some vertices on the middle rank. Taking those missed vertices to be one vertex chains gives an SCD. 
\end{proof}

\begin{corollary}\label{allisotypic}
Let $V$ be a finite dimensional representation of $\mathfrak{sl}_2$ with either all even or all odd highest-weights.   If $\{b_1, b_2, \dots b_n\}$ is a weight basis for $V$ consisting entirely of isotypic vectors, then the corresponding support graph admits an SCD.
\end{corollary}

\begin{proof}
The support graph decomposes as a disjoint union of support graphs for its isotypic components. Each of those components has an SCD, so the whole support graph does too. 
\end{proof}

\subsection{Diamond-free Posets} Four elements $x, y, z, w$ in a ranked poset form a \emph{diamond} if $z \gtrdot x \gtrdot w$ and $z \gtrdot y \gtrdot w$.  In the Hasse diagram this looks like:

\begin{center}
\includegraphics[scale=0.3]{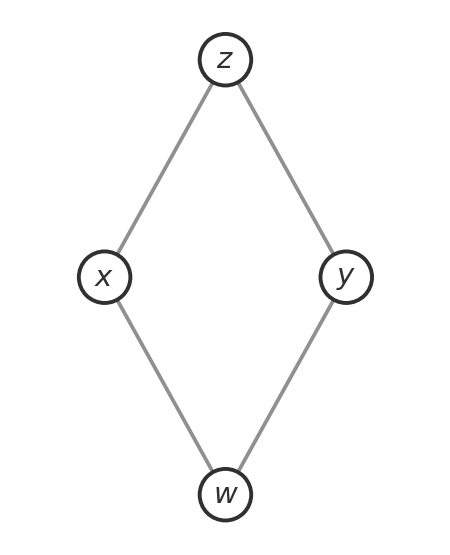}
\end{center}

We say a poset is \emph{diamond-free} if no four elements form a diamond.  Example \ref{smallestexample} is an example of a diamond free Peck poset, which we saw does not admit a strong $\mathfrak{sl}_2$-action.  The following theorem generalizes that result:

\begin{theorem}
If $P$ is a diamond-free ranked poset admitting a strong $\mathfrak{sl}_2$-action, then $P$ admits an SCD. 
\end{theorem}

\begin{proof}
Suppose a strong $\mathfrak{sl}_2$-action exists, and consider the action $EF$ on a rank $P_n$.  Suppose the matrix of $EF$ has a nonzero non-diagonal entry $(EF)_{x,y}$. Since $H = EF-FE$ acts diagonally on $P_n$, the corresponding entry $(FE)_{x,y}$  of $FE$ must also be nonzero.

A nonzero entry $(EF)_{x,y}$ must come from at least one pair of nonzero entries $E_{x,w}$ and $F_{w, y}$, and similarly nonzero entry $(FE)_{x,y}$ must come from at least one pair of nonzero entries $F_{x,z}$ and $E_{z, y}$.  However then $x,y,z$ and $w$ would form a diamond.  Therefore we conclude that $EF$ must act diagonally on $P_n$. 

As discussed in Section \ref{isotypicelements}, the eigenvectors $EF$ in a weight space are exactly the isotypic vectors.  Therefore every element of the poset must be an isotypic vector for this action, which puts us into the situation of Corollary \ref{allisotypic} and therefore we know an SCD exists. 
\end{proof}

\subsection{Width two posets} Next up we'll look at posets of width $2$, meaning the maximum rank size is $2$.  These correspond to representations $L(\lambda_1) \oplus L(\lambda_2)$.

\begin{theorem}
Let $P$ be a ranked poset with maximal rank size $2$.  $P$ admits a strong $\mathfrak{sl}_2$-action if and only if $P$ admits an SCD.
\end{theorem}

\begin{proof}

First note that if any pair of adjacent ranks $P_{n}$ and $P_{n-2}$, at least one of which has size $2$, is complete (meaning all possible edges between them exist) then $P$ admits an SCD.  To see this pick a matching on every layer.  If these matchings form an SCD we are done. If not, we can swap the matching at the complete layer to produce an SCD. 

Suppose $P$ is a width $2$ poset with a strong action of $\mathfrak{sl}_2$ and no complete layers (besides between ranks of size $1$).  Every layer looks like one of the following:

\begin{center}
\includegraphics[scale=0.2]{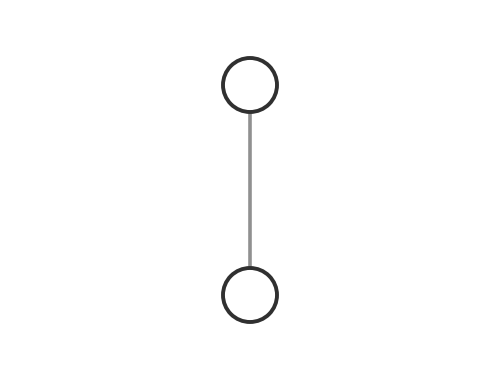}
\includegraphics[scale=0.2]{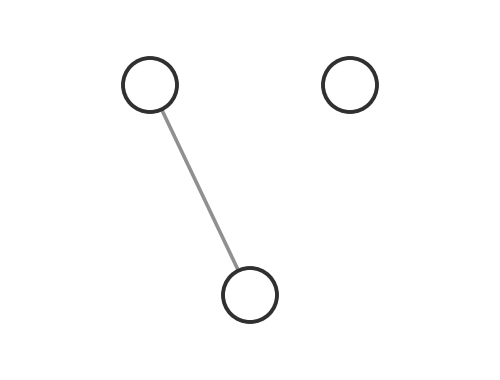}
\includegraphics[scale=0.2]{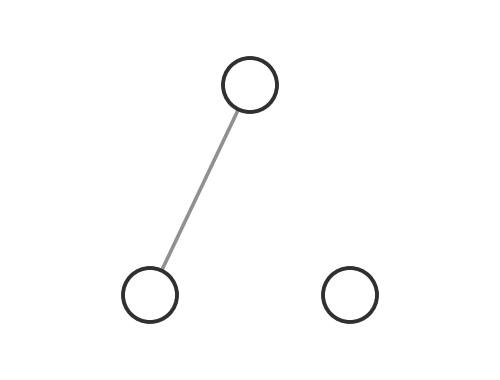}
\includegraphics[scale=0.2]{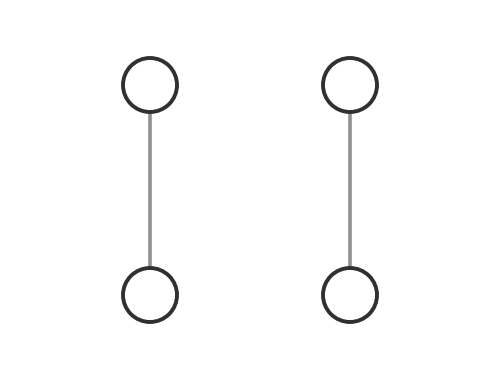}
\includegraphics[scale=0.2]{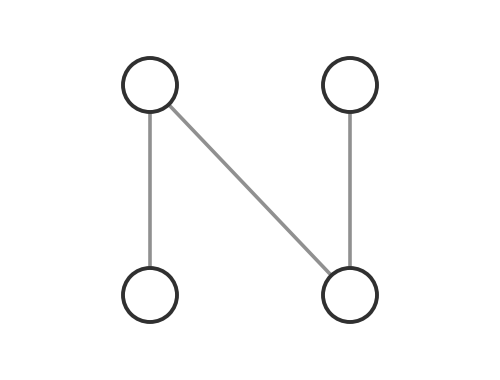}
\end{center}
All of these configurations have the property that if the top rank is known to consist of isotypic vectors then the bottom rank must also consist of isotypic vectors. In particular since the maximal rank in the poset always consists of isotypic vectors, this means the entire poset consists of isotypic vectors, so we know an SCD exists by Corollary \ref{allisotypic}.
\end{proof}

\noindent \textbf{Remark:} A variation of this argument can also be used to prove Conjecture \ref{theconjecture} for posets which are width $2$ except for a wide middle layer.  These correspond to representations $L(\lambda_1) \oplus L(\lambda_2) \oplus L(0)^k$.

\subsection{Rank 4 posets} 

The rank $2$ and $3$ cases of Conjecture \ref{theconjecture} are covered by Theorem \ref{onecomponent} and Corollary \ref{almostisotypic}. Indeed for those cases all Peck posets admit an SCD.  So in a sense rank $4$ is the first real case where the distinction between Proctor's $\mathfrak{sl}_2$-actions and our strong actions matters, and we have already seen several examples of rank $4$ posets which are Peck but do not admit an SCD (Examples \ref{smallestexample}, \ref{example1441},  \ref{example3553}, and \ref{example2442}). For the remainder of this section we will prove some special cases and partial results in the rank $4$ case. 

\subsubsection{Saturated and cosaturated sets}

  Let $P$ be a poset equipped with a strong $\mathfrak{sl}_2$-action.  A set $S \subseteq P_k$ is \emph{$L(\lambda)$-saturated} if the $\mathfrak{sl}_2$-equivariant projection $\pi_\lambda$  onto the $L(\lambda)$-isotypic component restricts to a bijection from $\text{span}(S)$ to the degree $k$ part of the $L(\lambda)$ component.   In other words, $S \subseteq P_k$ is $L(\lambda)$-saturated if $|S|$ is equal to the multiplicity of $L(\lambda)$ and $\text{span}(S)$ intersects the sum of the non-$L(\lambda)$-isotypic components trivially.

\begin{lemma}\label{saturated}
Let $S \subseteq P_k$ be a $L(\lambda)$-saturated set.

\begin{enumerate}[(a)]
\item If $k < \lambda$ then we can find a $L(\lambda)$-saturated set $S' \subseteq P_{k+2}$ such that there is a matching between $S$ and $S'$ consisting of edges with nonzero $E$-coefficients. 

\item If $k >  -\lambda$ then we can find a $L(\lambda)$-saturated set $S'' \subseteq P_{k-2}$ such that there is a matching between $S$ and $S''$ consisting of edges with nonzero $F$-coefficients. 
\end{enumerate}

\end{lemma}

\begin{proof}
We'll just prove $(a)$, as the proof for $(b)$ is similar. 

 Since $S$ is $L(\lambda)$-saturated, $\pi_\lambda$ is full rank from $\text{span}(S)$ to the degree $k$ part of the $L(\lambda)$ component.  Moreover, since $k < \lambda$ we know $E$ has full rank between ranks $k$ and $k+2$ on the $L(\lambda)$ component.  Therefore $E \circ \pi_\lambda$ is a full rank map from $\text{span}(S)$ to the degree $k+2$ part of the $L(\lambda)$ component.
 
 However since $\pi_\lambda$ is an equivariant map we have $E \circ \pi_\lambda = \pi_\lambda \circ E$. So Cauchy-Binet tells us that if $\pi_\lambda \circ E$ has full rank we can find $S' \subset P_{k+2}$ such that both $E_{S',S}$ and $(\pi_{\lambda})_{L(\lambda), S'}$ have full rank.  Those conditions correspond to the existence of a matching along edges with nonzero $E$-coefficients and $S'$ being $L(\lambda)$-saturated respectively. \end{proof}

Dually we say that $S \subseteq P_k$ is \emph{$L(\lambda)$-cosaturated} if the coordinate projection $\pi_S$ restricts to a bijection from the weight $k$ part of the $L(\lambda)$-isotypic component to $\text{span}(S)$.  In other words, $S \subseteq P_k$ is $L(\lambda)$-cosaturated if  $|S|$ is equal to the multiplicity of $L(\lambda)$ and the span of its complement $\text{span}(S^c)$ intersects the $L(\lambda)$-isotypic component trivially.  We get a similar result:

\begin{lemma}\label{cosaturated}
Let $S \subseteq P_k$ be a $L(\lambda)$-cosaturated set.

\begin{enumerate}[(a)]
\item If $k < \lambda$ then we can find a $L(\lambda)$-cosaturated set $S' \subseteq P_{k+2}$ such that there is a matching between $S$ and $S'$ consisting of edges with nonzero $F$-coefficients. 

\item If $k >  -\lambda$ then we can find a $L(\lambda)$-cosaturated set $S'' \subseteq P_{k-2}$ such that there is a matching between $S$ and $S''$ consisting of edges with nonzero $E$-coefficients. 
\end{enumerate}

\end{lemma}

\begin{proof}
Nearly identical to the proof of \ref{saturated}, just using the inclusion map $\iota_\lambda$ from the $L(\lambda)$ component into the representation instead of the projection map $\pi_\lambda$.
\end{proof}

The point is that these are  linear algebraic conditions which ensure that $S$ could be the intersection of a collection of $L(\lambda)$-chains with $P_k$. 

\begin{corollary}
If $S \subseteq P_k$ is either $L(\lambda)$-saturated or $L(\lambda)$-cosaturated, then we can find $m_\lambda$ disjoint saturated chains connecting $P_\lambda$ to $P_{-\lambda}$ which intersect $P_k$ at $S$, where $m_\lambda$ is the multiplicity of $L(\lambda)$.
\end{corollary}
\begin{proof}
Use Lemma \ref{saturated} or \ref{cosaturated} to grow the paths outward from $S$. 
\end{proof}

 Now let's turn our attention back to rank $4$.  We only have two isotypic components in this case, so we get a tighter relationship between saturated and cosaturated sets: 

\begin{lemma} \label{rank4complements}

Let $P$ be a rank $4$ poset with a strong $\mathfrak{sl}_2$-action.

\begin{enumerate}

\item $S \subseteq P_{\pm 1}$ is $L(3)$-saturated if and only if $S^c = P_{\pm1} - S$ is $L(1)$-cosaturated.

\item $S \subseteq P_{\pm 1}$ is $L(3)$-cosaturated if and only if $S^c = P_{\pm1} - S$ is $L(1)$-saturated.

\end{enumerate}
\end{lemma}

\begin{proof}
In general a set $T \subseteq P_k$ is $L(\lambda)$-saturated if $\text{span}(T^c)$ contains no nonzero $L(\lambda)$-isotypic vectors,  and is $L(\lambda)$-cosaturated if $\text{span}(T)$ intersects the sum of the other isotypic components trivially.  In particular if there are only two isotypic components then the cosaturated condition just becomes ``$T$ is $L(\lambda)$-cosaturated if $\text{span}(T)$ intersects the other component trivially" -- but that exactly means $T^c$ is saturated for the other component.
\end{proof}

\subsubsection{Rank vectors $[1,a+1,a+1,1]$ and $[a, a+1, a+1, a]$} We will now prove Conjecture \ref{theconjecture} in the case where the rank vector of the poset is either $[1,a+1,a+1,1]$ or $[a, a+1, a+1, a]$.  In other words, when the underlying representation is either $L(3) \oplus L(1)^a$ or $L(3)^a \oplus L(1)$. 

 As noted in Section \ref{graphdictionary}, if a square matrix has nonzero determinant then the corresponding bipartite graph has a perfect matching.  The following lemma gives us a mild strengthening of this when we have a block decomposition satisfying certain rank conditions. 

\begin{lemma}\label{matrixlemma}
Suppose $n = k+ \ell = p+q$, and $M$ is an invertible $n \times n$ matrix with a $(k+\ell) \times (p+q)$ block decomposition
$$M =  \begin{bmatrix}
A & B\\
C & D 
\end{bmatrix}$$
such that $\text{rank}(C) < p$ . Then the corresponding graph has a perfect matching using at least one edge corresponding to an entry in $A$. 

\begin{proof}
Expand the determinant along the first $p$ columns.  Because $\det(M) \ne 0$,  some $p \times p$ minor using those columns and its complementary minor must both be nonzero. Since $\text{rank}(C) < p$ the first minor must use at least one row from $A$.  The two complementary minors therefore give a matching using an edge corresponding to an entry of $A$. \end{proof}

\end{lemma} 

\begin{theorem}\label{multiplicityone}
Let $P$ be a poset with rank vector $[1,a+1,a+1,1]$.    $P$ admits a strong $\mathfrak{sl}_2$-action if and only if $P$ admits a symmetric chain decomposition. 
\end{theorem}

\begin{proof}
Fix a strong $\mathfrak{sl}_2$-action on $P$.  Let $P_1 = S \sqcup T$ where $S$ is the set of vertices with an edge to $P_3$, and $T$ is the set of vertices which do not.  Similarly let $P_{-1} = S' \sqcup T'$ where $S'$ is the set of vertices with an edge to $P_{-3}$, and $T'$ is the set of vertices which do not.

\begin{center}
\includegraphics[scale=0.25]{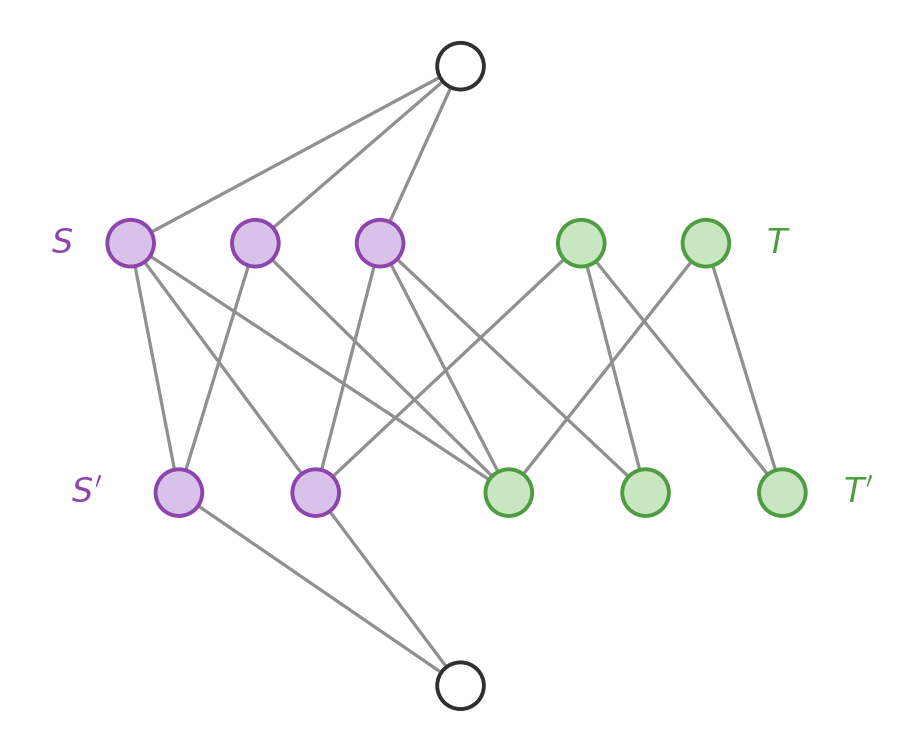}
\end{center}

It suffices to show that we can find a perfect matching on the middle layer such that at least one element of $S$ gets matched with an element of $S'$.   Consider the block decomposition of $E_{P_{1},P_{-1}}$ corresponding to the decompositions $P_1 = S \sqcup T$ and $P_{-1} = S' \sqcup T'$:  

$$E_{P_{-1},P_1} =  \begin{bmatrix}
E_{S,S'} & E_{S,T'} \\
 E_{T, S'} & E_{T,T'}
\end{bmatrix}$$

In particular if we can show the conditions of Lemma \ref{matrixlemma} hold that would prove the theorem.  We already know $E$ is invertible between those ranks, so we just need to check that $\text{rank}( E_{T, S'}) < |S'|$.

To see this, notice that $\text{span}(S')$ contains the unique (up to scaling) $L(3)$-isotypic vector in rank $-1$.  This vector gets mapped by $E$ to the unique $L(3)$-isotypic vector in rank $1$, which we similarly know is contained in $\text{span}(S)$.   The matrix $E_{T, S'}$ corresponds to applying $E$ to $\text{span}(S')$ and then projecting onto $\text{span}(T)$ killing $\text{span}(S)$.  In particular the unique  $L(3)$-isotypic vector gets killed in the second step,   so $\text{rank}( E_{T, S'}) < |S'|$. 
\end{proof} 

\begin{theorem}
Let $P$ is a poset with rank vector $[a,a+1,a+1,a]$.    $P$ admits a strong $\mathfrak{sl}_2$-action if and only if $P$ admits a symmetric chain decomposition. 
\end{theorem}

\begin{proof}
The proof will proceed similarly to the $[1, a+1, a+1, 1]$ case. We define the decompositions $P_1 = S \sqcup T$ and $P_{-1} = S' \sqcup T'$ slightly differently.  
Let $S$ be the set of those $x \in P_1$ that either have nonzero projection onto the $L(1)$-isotypic component, or those such that the $L(1)$-isotypic component is not killed by the coordinate projection onto $x$. In other words those $x$ such that $\{x\}$ is either $L(1)$-saturated or $L(1)$-cosaturated.   Similarly define $S' \subseteq P_{-1}$ and let $T$ and $T'$ be their respective complements. 

\begin{center}
\includegraphics[scale=0.25]{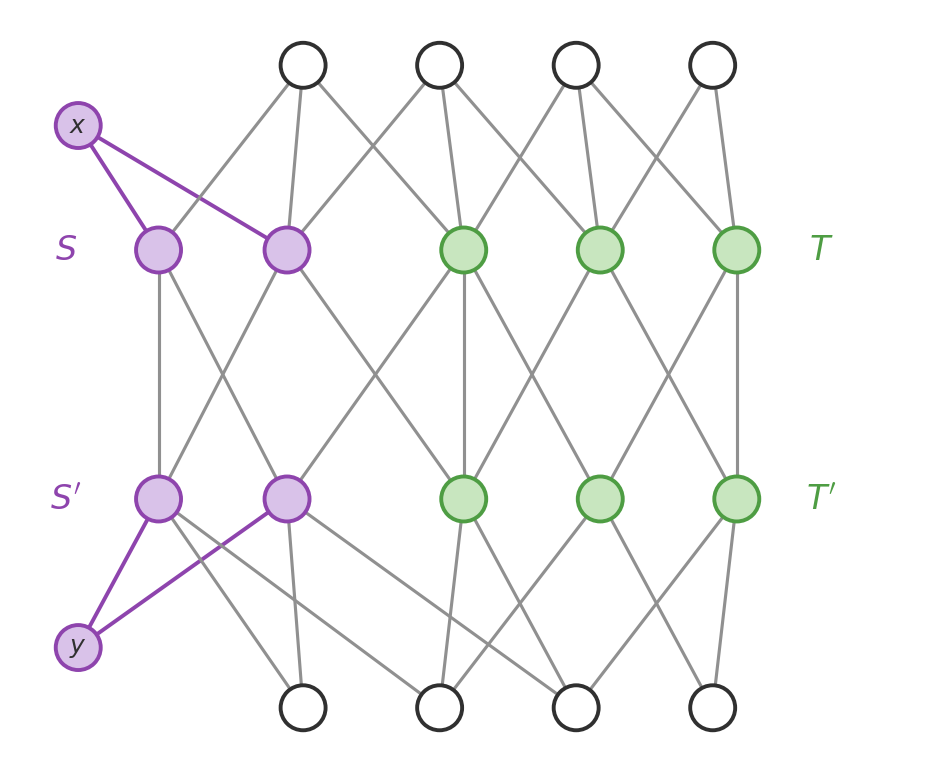}
\end{center}

In the above picture the purple edges are auxiliary, representing the nonzero matrix coefficients for the inclusion and projection maps between the $L(1)$ component and the full representation.  An element in the poset having such an edge exactly corresponds to it being $L(1)$-saturated or $L(1)$-cosaturated.

Now we proceed exactly as in Theorem \ref{multiplicityone}, using Lemma \ref{matrixlemma} to show that we can find a matching where at least one $a \in S$ is matched with $b \in S'$.  Then by Lemma \ref{rank4complements}, if $\{a\}$ is $L(1)$-saturated, then $P_1 - \{a\}$ is $L(3)$-cosaturated. Similarly, if $\{a\}$ is $L(1)$-cosaturated then $P_1 - \{a\}$ is $L(3)$-saturated.  In either case $P_1 - \{a\}$ has a matching with $P_3$, and $P_{-1} - \{b\}$ has a matching with $P_{-3}$. Therefore the matching on the middle layer from Lemma \ref{rank4complements} can be extended to an SCD.
\end{proof}

\subsubsection{Fractional SCDs for Rank 4 Posets} Now we will turn our attention to rank $4$ posets with arbitrary rank vectors, but relax things to consider fractional symmetric chain decompositions. We have the following theorem: 

\begin{theorem}\label{fractionalthm}
If $P$ is a rank $4$ poset admitting a strong $\mathfrak{sl}_2$-action then $P$ admits a fractional symmetric chain decomposition. 
\end{theorem}

\begin{proof} Let $P$ be a rank $4$ poset with a strong $\mathfrak{sl}_2$-action. For this argument we will think of a fractional symmetric chain decomposition as a collection of symmetric saturated chains covering every vertex of $P$ the same number of times.

Choose a $L(3)$-saturated set $S_0 \subseteq P_1$. By Lemma \ref{saturated}, we can find a $L(3)$-saturated set $S_0' \subseteq P_{-1}$ it matches with and we can extend that matching to a collection of disjoint length $4$ chains.  Set $T_0 = S_0'^c \subseteq P_{-1}$. By Lemma \ref{rank4complements} $T_0$ is $L(1)$-cosaturated,  so therefore we can match it with a $L(1)$-cosaturated set $T_0' \subseteq P_1$ by Lemma \ref{cosaturated}.

\begin{center}
\includegraphics[scale=0.4]{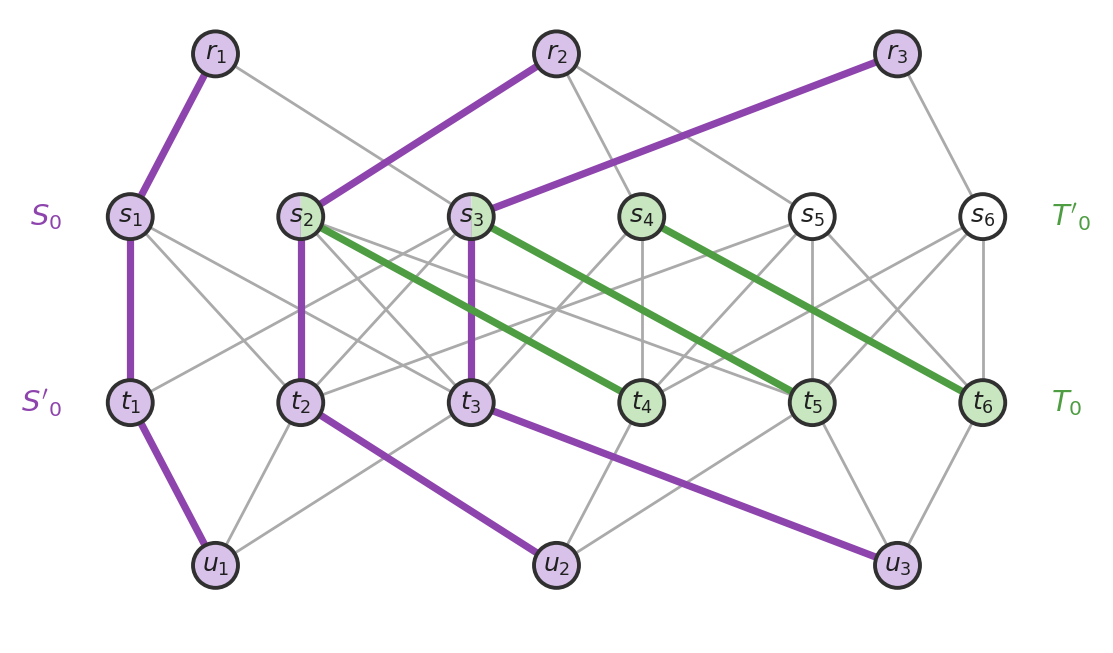}
\end{center}

 Notice that if $T_0' = S_0^c$ then we have found a symmetric chain decomposition.  If $T_0' \ne S_0^c$ we let $S_1 = T_0'^c$, which is $L(3)$-saturated by Lemma  \ref{rank4complements}, and we repeat the process. For each $n \ge 0$, we get  $S_n \subseteq P_1$ and $S_n' \subseteq P_{-1}$ which are $L(3)$-saturated sets with a matching between them, and  $T_n = S_n'^c \subseteq P_{-1}$  and $T_n' \subseteq P_1$ are $L(1)$-cosaturated sets with a matching between them, with $S_{n+1} = T_n'^c$.

\begin{center}
\includegraphics[scale=0.22]{alternating-proof-step-0-no-header.png}
\includegraphics[scale=0.22]{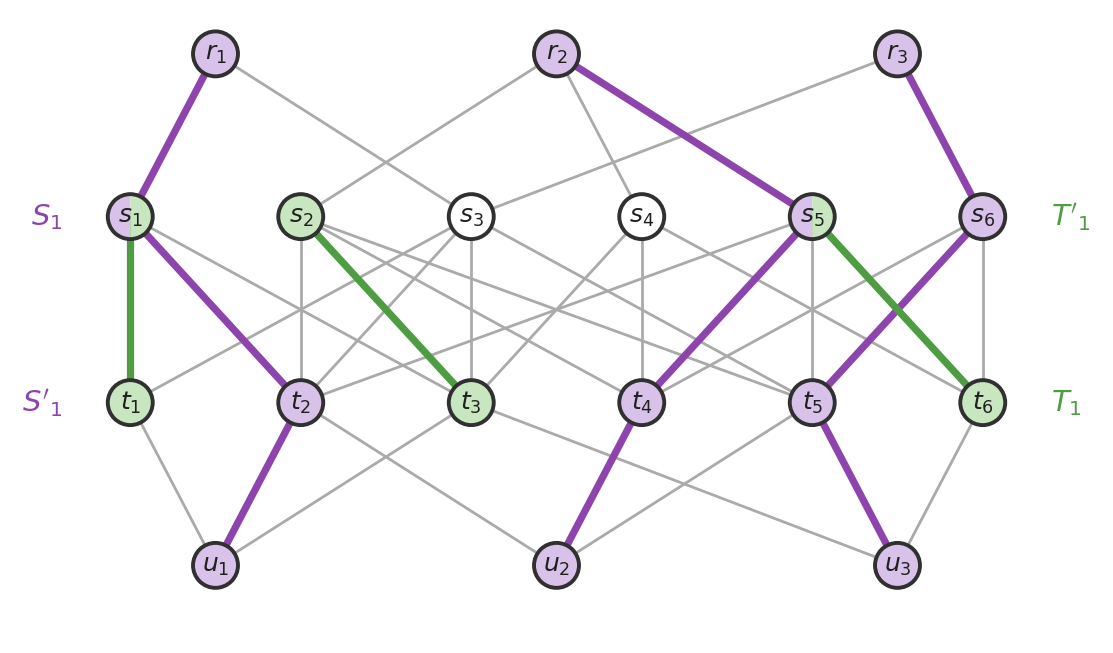}
\includegraphics[scale=0.22]{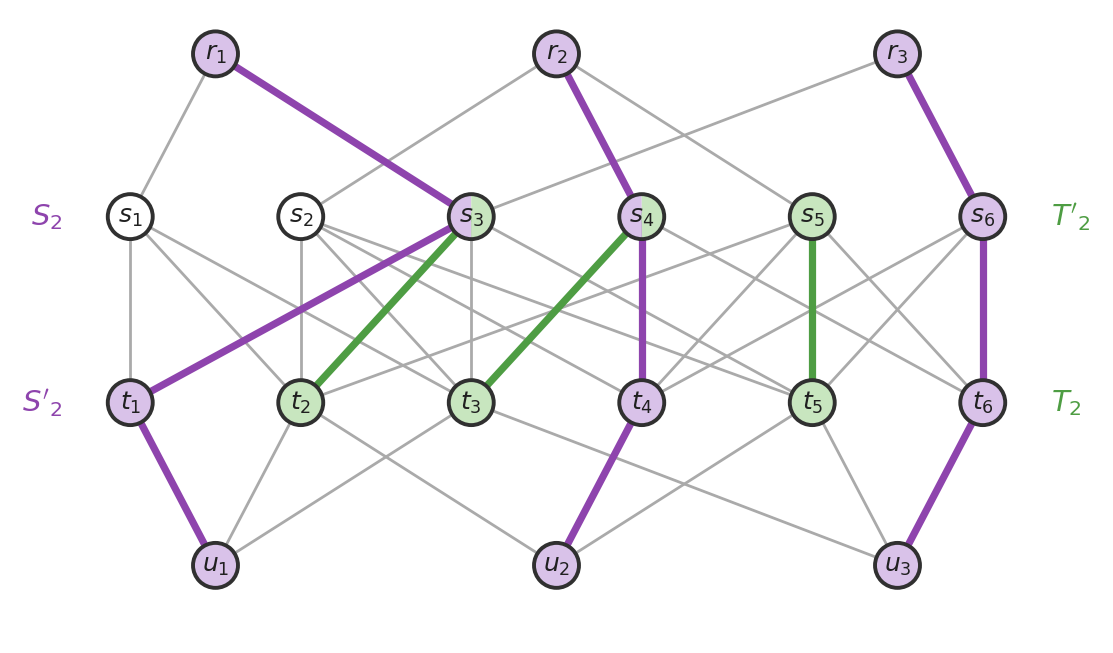}
\includegraphics[scale=0.22]{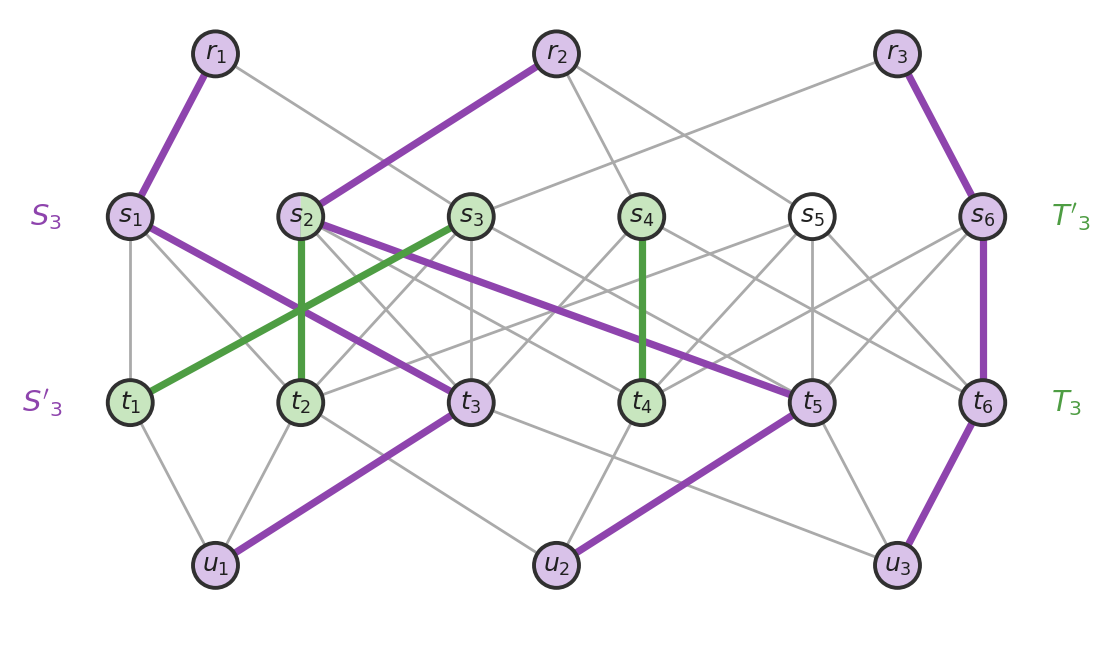}
\end{center}

There are only finitely many options for the $S_n$'s so at some point this process must repeat.  Suppose $T_n'^c = S_m$ for some $m < n$. We claim that all of the chains for steps $m$ through $n$ form a fractional symmetric chain decomposition. It is clear from the construction that each vertex in ranks $3$, $-1$ and $-3$ is covered once in each step.  In rank 1 for $a$ between $m$ and $n-1$  $T'_a$ and $S_{a+1}$ together cover all of the vertices once,  and then $T'_n$ and $S_m$ cover them one final time. 

In the example above we stopped after $4$ steps since $T_3'^c = S_1$. One can easily verify that the chains in the last 3 pictures collectively cover each vertex exactly $3$ times.

\end{proof}

\subsection{More questions and special cases:}   Of course a full solution to Conjecture \ref{theconjecture} would be great, but there are many interesting special cases and variations one could consider beyond those we just did.   Moreover we think many of those directions would still be interesting even if the full conjecture ends up being false.  Here are a few we'd like to highlight:

\medskip

\noindent \textbf{Open Problem:}  \emph{For which rank vectors does Conjecture \ref{theconjecture} hold?}

\smallskip

Of course we think the answer is all of them, but there is lots of room for partial progress.  We gave a few infinite families in this section, and have checked it computationally for rank vectors with fewer than $14$ elements.  All rank $4$ posets or width $3$ posets might be natural candidates for the next families to look at. 

\medskip

\noindent \textbf{Open Problem:}  \emph{Do rank 5 posets always admit fractional SCDs?}

\smallskip

The rank $4$ argument breaks down in rank 5. It still produces a family of chains covering every element outside of the middle rank the same number of times. However its possible that one of the middle elements gets used too many times to to be able turn that collection into a fractional SCD.  Perhaps one can prevent this from happening though with a more careful argument.

\medskip

\noindent  \textbf{Open Problem:}  \emph{Do posets with positive $\mathfrak{sl}_2$-actions always admit (fractional) SCDs?}

\smallskip

The Young lattices $L(m,n)$ have the additional property that all of the $E$ and $F$ coefficients can be taken to be non-negative real numbers. We call these \emph{positive actions}.  These greatly restrict the combinatorics of the poset by preventing certain terms from cancelling in the $\mathfrak{sl}_2$ relations, and seem to be quite rare.

\end{section}

\begin{section}{Computational Experiments}\label{experiments}

Our main computational experiment can be summarized as follows: 

\begin{enumerate}

\item Generate either a random Peck poset or a random SCD-critical Peck poset.

\item Test whether it admits an SCD or a fractional SCD.

\item If it does not admit an SCD try to prove it does not admit a strong $\mathfrak{sl}_2$-action. 

\item If it does not admit an SCD but we fail to prove it cannot admit a strong $\mathfrak{sl}_2$-action, search for a strong action numerically.

\item Repeat.

\end{enumerate}

We'll briefly explain how we did these steps, and explain certain choices we made in the process, and then we will report the outcome of these experiments.

\subsection{Generating random Peck posets}

Recall the characterization of Peck posets due to Griggs (Definition/Theorem \ref{Peckdef}):  An RSU poset is Peck if and only if between any pair of opposite ranks there exists a family of disjoint saturated chains connecting them.  This allows one to construct random Peck posets via the following procedure:

\begin{enumerate}
\item Fix a symmetric and unimodal rank vector, and put the appropriate number of elements in each rank.

\begin{center}
\includegraphics[scale=0.25]{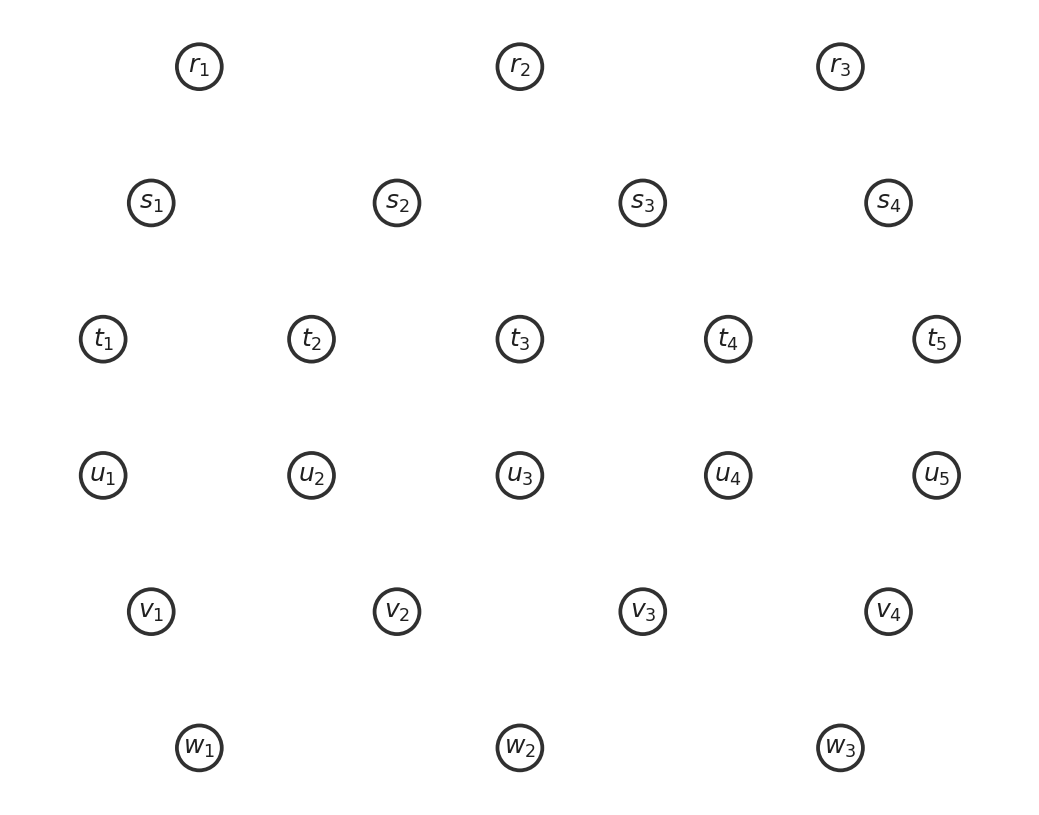}
\end{center}

\item For each pair of opposite ranks randomly select a collection of disjoint saturated chains connecting them, and add those cover relations to the poset. 

\begin{center}
\includegraphics[scale=0.25]{random-peck-construction-2-chain-families.png}
\end{center}

\item Between each pair of adjacent ranks add some additional cover relations with some probability. 

\begin{center}
\includegraphics[scale=0.25]{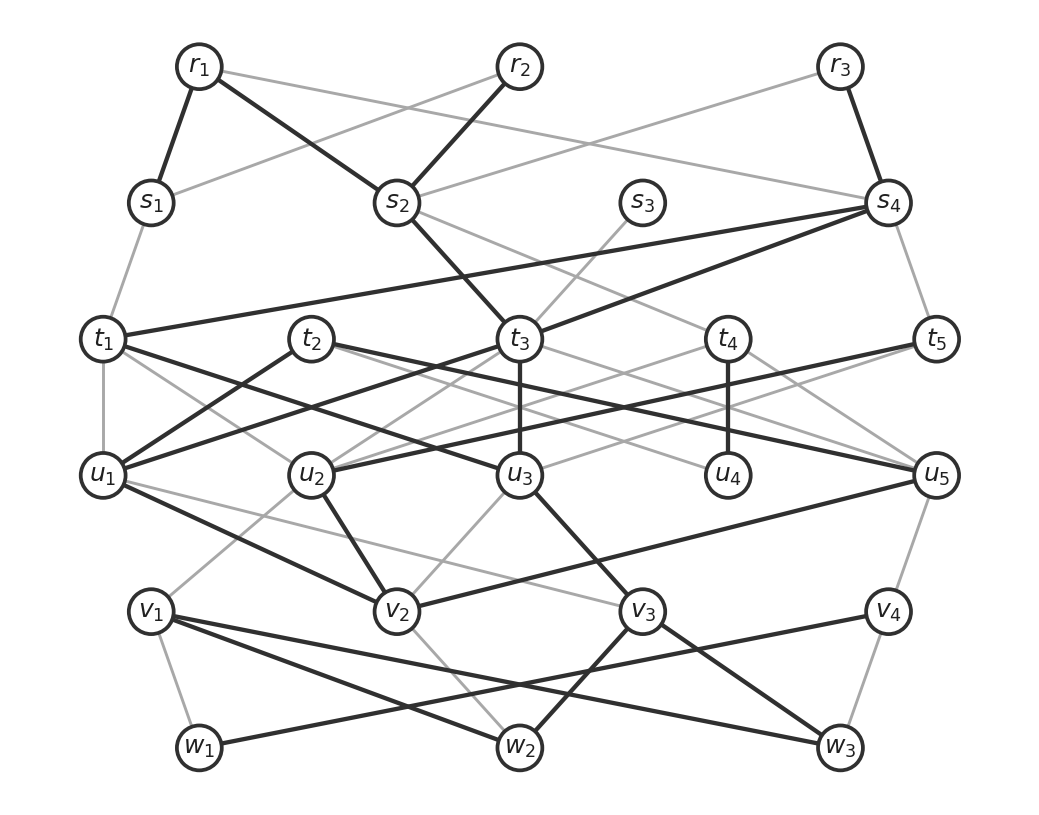}
\end{center}

\end{enumerate}

We'll note that this procedure does not sample from the set of Peck posets with given edge densities uniformly.  A Peck poset occurs with probability roughly proportional to the number of families of symmetric saturated chains satisfying the Griggs condition (the number of ways we could have gotten it from this process).

We also tried a procedure that samples uniformly, and it was far less computationally efficient. Moreover, Peck posets with very few Griggs witness families tend to not admit SCDs, fractional SCDs, or strong $\mathfrak{sl}_2$-actions very often, and the bottlenecks/obstructions tend to be very simple.  The procedure described above seems to produce many more interesting examples which can fail to admit these structures for more subtle reasons.

\subsection{Testing for (fractional) SCDs} \label{SCDtesting} Once we have a Peck poset the first thing we want to do is check whether or not it admits an SCD.  We experimented with 3 ways to do this:

\medskip

\noindent \textbf{Method 1:} Enumerate all symmetric saturated chains and perform a brute force search.  

\medskip

\noindent \textbf{Method 2:}  First run a program to find certain easy deductions like ``$x$ must be the start of its chain" or ``the chain through $y$ must also go through $z$", then use these to trim the poset before performing a brute force search as in method 1. 

\medskip

\noindent \textbf{Method 3:} Use built-in integer programming tools in Python to search for integer points in the flow polytope (see Section \ref{flowpolytope}). 

\medskip

Method 1 was the most efficient for relatively small rank vectors or low edge densities.  Method 2 was best for certain middle-complexity posets. Once the posets were sufficiently large and complicated, method 3 was the most efficient.  Another advantage of method 3 is that it also tells us whether the poset admits a fractional SCD essentially for free since the integer program solvers typically solve the linear program along the way. 

\subsection{Generating random SCD-critical posets}

Starting with a Peck poset which does not admit an SCD one can add in edges at random one at a time,  verifying each time that the added edge does not create an SCD.  Once you cannot add any more edges, the resulting poset is an SCD-critical Peck poset. 

For small enough rank vectors this is reasonably computationally tractable (in fact for very small rank vectors there were few enough SCD-critical Peck posets that we could just enumerate them all).  However for larger rank vectors this becomes computationally intensive because one needs to keep checking whether the poset admits an SCD, so we were able to test far fewer of these cases for larger rank vectors.  

\subsection{The strong $\mathfrak{sl}_2$-action autodisprover}  Given a Peck poset which we have already verified does not admit an SCD,  we then want to prove it also does not admit a strong $\mathfrak{sl}_2$-action.  To do this we wrote a Python program to automatically generate proofs that a given poset does not admit a $\mathfrak{sl}_2$-action. 

This is in some sense the most interesting thing we built for the experiment, but we've essentially already explained how it works: Section \ref{combinatorics} lists a series of deductions, implications, reductions, and contradictions one can make when trying to prove a given poset does not have a strong $\mathfrak{sl}_2$-action.   The program just iteratively applies these rules until it either finds a contradiction or gets stuck. 

Importantly though we'll note that if the autodisprover does not find a proof, that does not mean the poset admits a strong $\mathfrak{sl}_2$-action.  Indeed, we analyzed many posets that earlier versions of the autodiprover failed to catch. Developing proofs for those examples is how we came up with many of the techniques described in Section \ref{combinatorics}.  The failure rate of the autodisprover is mostly a reflection of where we decided to stop this process rather than some inherent limit on these methods.

\subsection{Searching for numerical solutions}
Whenever our disprover fails we have a candidate counterexample to the conjecture. We then want to just see if we can construct a strong $\mathfrak{sl}_2$ action, by directly solving the equations involved.  For small enough posets we could use Gr\"obner bases to get a definitive answer one way or another,  however for larger posets that became computationally infeasible.  Instead here's what we did: 

 If we fix the raising operator $E$, finding the corresponding $F$ operator satisfying the $\mathfrak{sl}_2$-relations is just a linear system of equations.  However the solution will typically not be supported along the edges of the poset,  so instead we force it to be supported on the poset and find a least squares solution to the $\mathfrak{sl}_2$-relations. Once we have that $F$ we flip things around and find a new $E$ that is supported on the edges we want and is a least squares solution to the $\mathfrak{sl}_2$-relations. We repeat this process, alternating solving for $E$ and $F$, getting operators supported on our edges which are closer to satisfying the $\mathfrak{sl}_2$-relations with every step.  
 
  This is by no means guaranteed to converge to a solution, and indeed often it get's stuck at local minima even when we know solutions exist.  However in practice, with enough starting seeds and a bit of random noise added in,  it seems to do pretty well at finding solutions when they exist.

\subsection{Choosing rank sizes and densities}

We wanted to test the conjecture on a variety of posets of different shapes and edge densities.  Here's a bit more on how we decided what to test:

\begin{itemize}

\item The number of samples tested is a function of time.  For each rank  (rank $4$ to rank $10$) we ran one $\sim\hspace{-.1cm}8$ hour sweep;  $\sim\hspace{-.1cm}6$ hours constructing and testing random posets, and $\sim\hspace{-.1cm}2$ hours constructing and testing SCD-critical posets.  In each rank we cycled through the rank vectors being tested so they each got roughly the same number of samples. This did mean that the larger rank vectors tended to use up most of the time.  Had we instead given the rank vectors equal time we would have tested far more small cases, but fewer large ones.  Since larger rank vectors tended to produce harder more interesting examples, we prioritized those. 

\smallskip

\item The rank vectors tested were chosen by what seemed computationally feasible in test runs, as well as a desire to test posets of a variety of shapes (e.g. thin at the extremes and wide in the middle, or relatively boxy with almost equal rank sizes). 

\smallskip

\item We skipped any cases covered by results in Section \ref{specialcases}.

\smallskip

\item We skipped all cases with fewer than $14$ elements. \textbf{Bonus Theorem:} \emph{Conjecture \ref{theconjecture} holds for ranked posets with fewer than $14$ elements.}  During the process of building and testing the autodisprover we completely settled these cases by enumerating all SCD-critical posets with fewer then $14$ elements which are not covered by the results in Section \ref{specialcases} and checking they don't admit a fractional SCD.  At $14$ elements the enumeration process became too computationally intensive for some rank vectors, which is why we stopped there.

\smallskip

\item For each rank vector we ran a quick sweep of different edge densities to get an estimate of how often they produce posets with SCDs vs without. We then chose three density profiles. The sparse profile was not so sparse that SCDs never appeared. The medium profile produced a mixture of SCD and non-SCD examples. The dense profile was not so dense that every example admitted an SCD.  These different regimes tended to produce examples that feel qualitatively different, but the totals reported below are these 3 combined. 

\smallskip

\item For each chosen rank vector we also tested SCD-critical posets, which tend to be much more dense than the ``dense" random regime but extremely unlikely to show up in random searches.  These are much more computationally intensive to both construct and test, so even though we dedicated a quarter of the total run time to these, they make up a far smaller portion of the total.

\smallskip

\end{itemize}

Our main interest was in providing evidence for, or finding a counterexample to, Conjecture \ref{theconjecture}, although we did also decide to track how often the posets had fractional SCDs but not ordinary SCDs. There are many other interesting questions one could ask about the distributions of Peck posets produced this way -- including a more detailed analysis of how the density affects the probability the poset will have an SCD.

\subsection{Experimental Results:}

Below are tables summarizing the results of this computational experiment.  The totals in the random poset columns are combined across the sparse, medium, and dense bands.  The ``Fractional" columns record how many posets we found which admit a fractional SCD but not an ordinary SCD.  The ``Open" columns record the number of posets where the autodisprover was not able to certify that a strong $\mathfrak{sl}_2$-action does not exist.   For every open case, we also conducted a numerical search for a solution. None of these searches found one.

\vfill

\begin{center}
{\large\bfseries Rank 4\par}
\vspace{0.5em}

\small
\setlength{\tabcolsep}{4.5pt}
\begin{tabular}{lrrrrrrr}
\toprule
& \multicolumn{4}{c}{Random posets}
& \multicolumn{3}{c}{SCD-critical posets} \\
\cmidrule(lr){2-5}
\cmidrule(lr){6-8}
Rank vector
& Total
& Non-SCD
& Fractional
& Open
& Tested
& Fractional
& Open \\
\midrule
$(2,5,5,2)$      & 457,722 & 186,592 &  2 & 0 & 49 & 0 & 0 \\
$(2,9,9,2)$      & 457,722 & 177,861 & 12 & 0 & 49 & 0 & 0 \\
$(3,5,5,3)$      & 457,722 & 182,887 &  2 & 0 & 49 & 0 & 0 \\
$(3,8,8,3)$      & 457,722 & 180,970 & 16 & 0 & 49 & 0 & 0 \\
$(3,10,10,3)$    & 457,722 & 177,540 & 12 & 0 & 49 & 0 & 0 \\
$(4,7,7,4)$      & 457,722 & 176,066 &  3 & 0 & 49 & 0 & 0 \\
$(4,11,11,4)$    & 457,721 & 187,226 & 21 & 0 & 49 & 1 & 1 \\
$(6,8,8,6)$      & 457,719 & 177,697 &  1 & 0 & 49 & 0 & 0 \\
$(6,11,11,6)$    & 457,719 & 179,812 & 18 & 0 & 49 & 0 & 0 \\
$(6,13,13,6)$    & 457,719 & 174,528 & 24 & 0 & 49 & 1 & 1 \\
$(8,10,10,8)$    & 457,719 & 175,549 &  3 & 0 & 48 & 0 & 0 \\
$(10,15,15,10)$  & 457,719 & 186,703 &  5 & 0 & 48 & 0 & 0 \\
\midrule
\textbf{Total}
& \textbf{5,492,648}
& \textbf{2,163,431}
& \textbf{119}
& \textbf{0}
& \textbf{586}
& \textbf{2}
& \textbf{2} \\
\bottomrule
\end{tabular}

\end{center}

\vfill

\begin{center}
{\large\bfseries Rank 5\par}
\vspace{0.5em}

\small
\setlength{\tabcolsep}{4.5pt}
\begin{tabular}{lrrrrrrr}
\toprule
& \multicolumn{4}{c}{Random posets}
& \multicolumn{3}{c}{SCD-critical posets} \\
\cmidrule(lr){2-5}
\cmidrule(lr){6-8}
Rank vector
& Total
& Non-SCD
& Fractional
& Open
& Tested
& Fractional
& Open \\
\midrule
$(1,4,4,4,1)$ & 556,761 & 240,016 & 0  & 0 & 418 & 0 & 0 \\
$(1,5,7,5,1)$ & 556,761 & 263,012 & 7  & 0 & 418 & 0 & 0 \\
$(2,4,7,4,2)$ & 556,761 & 255,389 & 10 & 0 & 418 & 0 & 0 \\
$(2,5,5,5,2)$ & 556,761 & 253,323 & 17 & 0 & 418 & 2 & 0 \\
$(2,6,8,6,2)$ & 556,761 & 247,867 & 37 & 0 & 418 & 0 & 17 \\
$(3,4,7,4,3)$ & 556,758 & 230,141 & 7  & 0 & 418 & 0 & 0 \\
$(3,5,6,5,3)$ & 556,758 & 277,833 & 18 & 0 & 418 & 0 & 1 \\
$(3,6,8,6,3)$ & 556,758 & 242,923 & 32 & 0 & 418 & 3 & 10 \\
$(3,7,7,7,3)$ & 556,758 & 290,867 & 38 & 0 & 418 & 3 & 19 \\
$(4,5,5,5,4)$ & 556,758 & 249,211 & 0  & 0 & 418 & 0 & 0 \\
$(4,6,8,6,4)$ & 556,758 & 269,685 & 38 & 0 & 417 & 1 & 3 \\
$(4,7,8,7,4)$ & 556,758 & 263,773 & 48 & 0 & 417 & 3 & 18 \\
\midrule
\textbf{Total}
& \textbf{6,681,111}
& \textbf{3,084,040}
& \textbf{252}
& \textbf{0}
& \textbf{5,014}
& \textbf{12}
& \textbf{68} \\
\bottomrule
\end{tabular}

\end{center}

\medskip

\begin{center}
{\large\bfseries Rank 6\par}
\vspace{0.5em}

{\small
\setlength{\tabcolsep}{4.5pt}
\begin{tabular}{lrrrrrrr}
\toprule
& \multicolumn{4}{c}{Random posets}
& \multicolumn{3}{c}{SCD-critical posets} \\
\cmidrule(lr){2-5}
\cmidrule(lr){6-8}
Rank vector
& Total
& Non-SCD
& Fractional
& Open
& Tested
& Fractional
& Open \\
\midrule
$(1,2,4,4,2,1)$ & 310,740 & 146,225 & 0  & 0 & 172 & 1 & 0 \\
$(1,3,6,6,3,1)$ & 310,740 & 158,245 & 9  & 0 & 172 & 1 & 1 \\
$(1,4,5,5,4,1)$ & 310,740 & 146,348 & 14 & 0 & 172 & 0 & 0 \\
$(1,5,7,7,5,1)$ & 310,740 & 141,901 & 13 & 0 & 172 & 0 & 3 \\
$(2,3,4,4,3,2)$ & 310,740 & 142,896 & 3  & 0 & 172 & 0 & 0 \\
$(2,4,6,6,4,2)$ & 310,740 & 142,227 & 23 & 0 & 172 & 1 & 1 \\
$(2,5,5,5,5,2)$ & 310,740 & 146,453 & 29 & 0 & 172 & 2 & 0 \\
$(3,4,5,5,4,3)$ & 310,740 & 143,640 & 3  & 0 & 172 & 0 & 1 \\
$(3,5,5,5,5,3)$ & 310,740 & 150,871 & 19 & 0 & 172 & 0 & 2 \\
$(3,5,7,7,5,3)$ & 310,740 & 139,177 & 55 & 0 & 172 & 0 & 3 \\
$(3,6,9,9,6,3)$ & 310,739 & 124,519 & 87 & 0 & 172 & 3 & 25 \\
$(4,6,9,9,6,4)$ & 310,737 & 145,872 & 62 & 0 & 171 & 2 & 9 \\
\midrule
\textbf{Total}
& \textbf{3,728,876}
& \textbf{1,728,374}
& \textbf{317}
& \textbf{0}
& \textbf{2,063}
& \textbf{10}
& \textbf{45} \\
\bottomrule
\end{tabular}
}

\end{center}

\medskip


\begin{center}
{\large\bfseries Rank 7\par}
\vspace{0.5em}

{\small
\setlength{\tabcolsep}{4.5pt}
\begin{tabular}{lrrrrrrr}
\toprule
& \multicolumn{4}{c}{Random posets}
& \multicolumn{3}{c}{SCD-critical posets} \\
\cmidrule(lr){2-5}
\cmidrule(lr){6-8}
Rank vector
& Total
& Non-SCD
& Fractional
& Open
& Tested
& Fractional
& Open \\
\midrule
$(1,2,3,4,3,2,1)$     & 51,357 & 21,779 &  2 & 0 & 18 & 0 & 0 \\
$(1,2,5,6,5,2,1)$     & 51,357 & 22,847 &  0 & 0 & 18 & 0 & 1 \\
$(1,4,4,7,4,4,1)$     & 51,357 & 17,054 &  4 & 0 & 18 & 0 & 0 \\
$(1,5,7,9,7,5,1)$     & 51,357 & 22,415 &  6 & 0 & 18 & 0 & 2 \\
$(2,3,4,5,4,3,2)$     & 51,357 & 22,681 &  1 & 0 & 18 & 0 & 0 \\
$(2,4,6,6,6,4,2)$     & 51,357 & 25,942 & 10 & 0 & 18 & 0 & 0 \\
$(2,6,8,9,8,6,2)$     & 51,357 & 23,971 & 36 & 0 & 18 & 1 & 6 \\
$(3,4,6,8,6,4,3)$     & 51,357 & 23,222 & 10 & 0 & 18 & 0 & 1 \\
$(3,6,6,9,6,6,3)$     & 51,357 & 22,766 & 16 & 0 & 18 & 0 & 2 \\
$(3,7,10,11,10,7,3)$  & 51,357 & 20,568 & 73 & 0 & 18 & 1 & 6 \\
$(4,5,7,7,7,5,4)$     & 51,357 & 22,107 &  2 & 0 & 18 & 0 & 0 \\
$(4,7,9,11,9,7,4)$    & 51,356 & 22,607 & 39 & 0 & 18 & 2 & 7 \\
\midrule
\textbf{Total}
& \textbf{616,283}
& \textbf{267,959}
& \textbf{199}
& \textbf{0}
& \textbf{216}
& \textbf{4}
& \textbf{25} \\
\bottomrule
\end{tabular}
}

\end{center}

\medskip


\begin{center}
{\large\bfseries Rank 8\par}
\vspace{0.5em}

{\small
\setlength{\tabcolsep}{4.5pt}
\begin{tabular}{lrrrrrrr}
\toprule
& \multicolumn{4}{c}{Random posets}
& \multicolumn{3}{c}{SCD-critical posets} \\
\cmidrule(lr){2-5}
\cmidrule(lr){6-8}
Rank vector
& Total
& Non-SCD
& Fractional
& Open
& Tested
& Fractional
& Open \\
\midrule
$(1,2,3,4,4,3,2,1)$       & 29,496 & 11,731 &  1 & 0 & 8 & 0 & 0 \\
$(1,2,5,6,6,5,2,1)$       & 29,496 & 15,077 &  3 & 0 & 8 & 0 & 1 \\
$(1,4,4,7,7,4,4,1)$       & 29,496 & 12,665 & 11 & 0 & 8 & 0 & 0 \\
$(1,5,7,9,9,7,5,1)$       & 29,496 & 15,019 & 10 & 0 & 8 & 0 & 0 \\
$(2,3,4,5,5,4,3,2)$       & 29,496 & 11,267 &  2 & 0 & 8 & 0 & 0 \\
$(2,4,6,6,6,6,4,2)$       & 29,496 & 14,549 & 11 & 0 & 8 & 0 & 1 \\
$(2,6,8,9,9,8,6,2)$       & 29,495 & 14,892 & 42 & 0 & 8 & 1 & 3 \\
$(3,4,6,8,8,6,4,3)$       & 29,493 & 13,420 & 10 & 0 & 8 & 0 & 1 \\
$(3,6,6,9,9,6,6,3)$       & 29,493 & 10,827 & 50 & 0 & 8 & 1 & 3 \\
$(4,5,7,7,7,7,5,4)$       & 29,493 &  9,539 &  2 & 0 & 8 & 0 & 0 \\
$(4,7,9,10,10,9,7,4)$     & 29,493 & 13,118 & 49 & 0 & 8 & 2 & 4 \\
$(4,8,10,10,10,10,8,4)$   & 29,493 & 13,924 & 79 & 0 & 8 & 0 & 3 \\
\midrule
\textbf{Total}
& \textbf{353,936}
& \textbf{156,028}
& \textbf{270}
& \textbf{0}
& \textbf{96}
& \textbf{4}
& \textbf{16} \\
\bottomrule
\end{tabular}
}

\end{center}

\subsection{Takeaways and Caveats:} Across \(16{,}872{,}854\) randomly generated Peck posets, every non-SCD example was certified not to admit a strong action. Among \(7{,}975\) SCD-critical posets, the autodisprover certified \(7{,}819\), leaving \(156\) unresolved. Numerical searches found no strong action in any unresolved case.  While this did produce a small number of unresolved open cases, where we could neither prove nor disprove the existence of a strong $\mathfrak{sl}_2$-action, overall we think these results are compelling computational evidence for the conjecture.  Here are a few additional remarks on these results: 

\begin{itemize}

\item The autodisprover has essentially completely saturated testing posets which appear in the random searches.  This is likely due in part to the fact that we essentially trained it on this data. We developed the autodisprover iteratively, adding new arguments whenever it failed on an example. By the end, we were training it on examples much rarer and more complicated than those arising in ordinary random searches.  It is possible that potential counterexamples have to be far more structured than random posets, and we are just missing them entirely.

\smallskip

\item On SCD-critical posets (the worst case scenario for showing an $\mathfrak{sl}_2$-action does not exist) the success rate drops. In total we had $156$ open cases out of $7975$ SCD-critical posets tested -- about a $98 \%$ success rate overall. This is still a high overall success rate. For some larger shapes, however, the success rate dropped significantly. The worst was  $(4, 7, 9, 10, 10, 9, 7, 4)$ with a $50\%$ success rate. 

\smallskip

\item Not shown in these tables is the intermediate edge-density regime in between where random searches start producing posets which have SCDs  nearly $100\%$ of the time, and the maximally dense SCD-critical posets.   One can test these by starting with a random non-SCD poset and iteratively adding edges while making sure it doesn't create an SCD, but stopping short of making it SCD-critical. We decided to only test the purely random, and maximally difficult regimes for this experiment, but it's possible that some interesting behavior occurs in this intermediate regime that is not captured by this data.

\smallskip

\item The numerical solver is not guaranteed to find a solution if one exists, so all of the open cases are potential counterexamples to the conjecture.  That being said, we ran versions of this experiment many times and tested many more open cases than just the $156$ in this official run, and we have yet to find an actual counterexample.

\smallskip

\item This is just a snapshot of where the autodisprover is right now.  As we worked on this we watched the success rate on both random and SCD-critical posets creep up with every iteration.  We have no reason to think that it couldn't be improved further to cover more of those remaining open cases. 

\smallskip


\end{itemize}

\subsection{Gallery of Failures} To close things out, here are a few of the ``open" cases in which our autodisprover failed to disprove the existence of a strong $\mathfrak{sl}_2$-action.  Our searches for a numerical solution on these all failed, but they are all potential counterexamples to Conjecture \ref{theconjecture}.

\begin{center}
\includegraphics[scale=0.4]{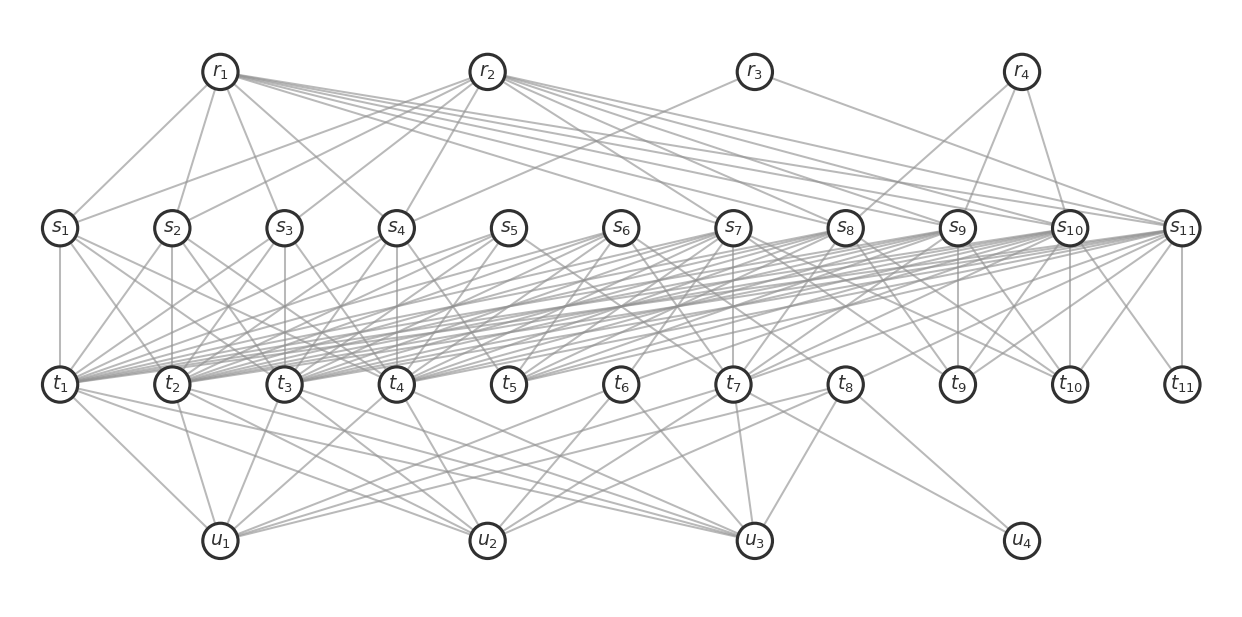}
\end{center}

\begin{center}
\includegraphics[scale=0.4]{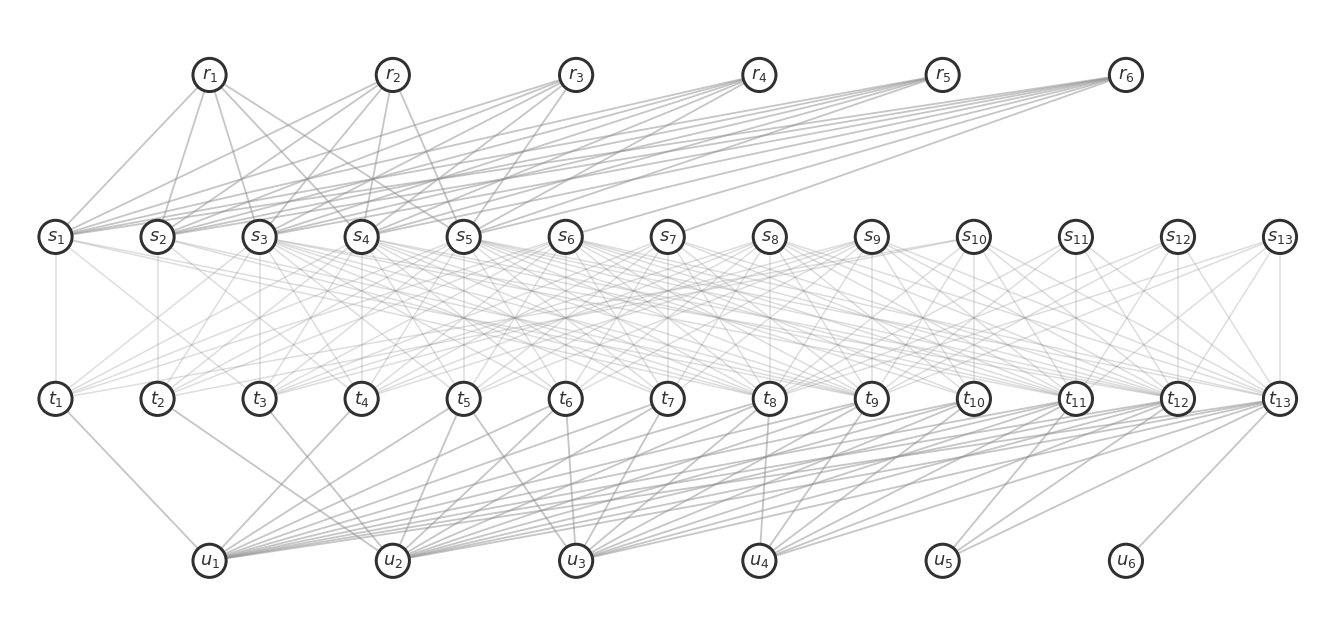}
\end{center}

\begin{center}
\includegraphics[scale=0.4]{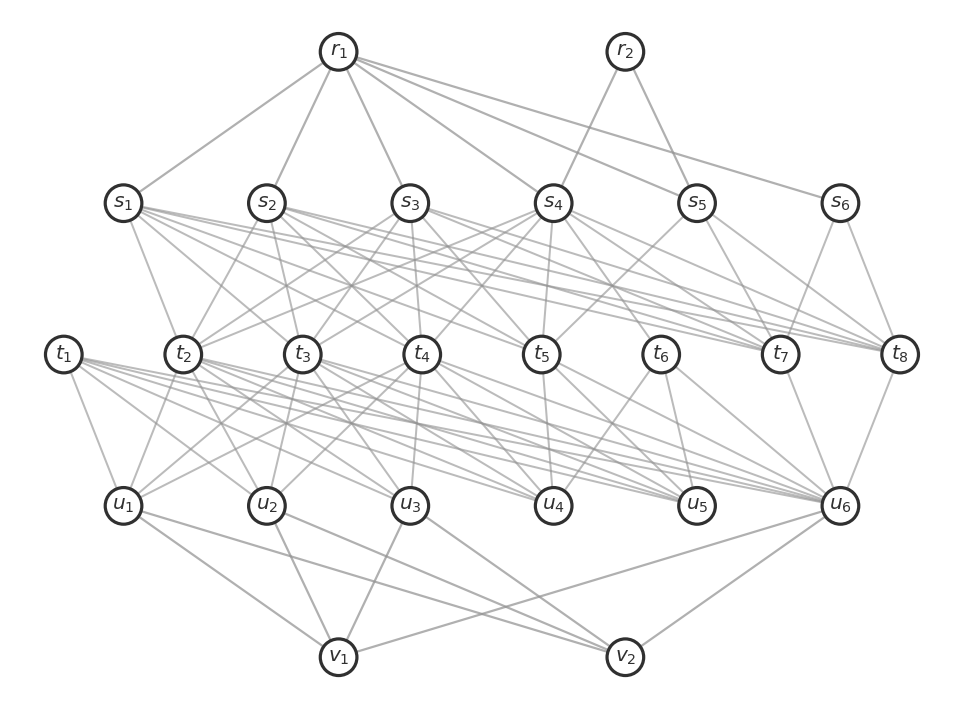}
\end{center}

\begin{center}
\includegraphics[scale=0.4]{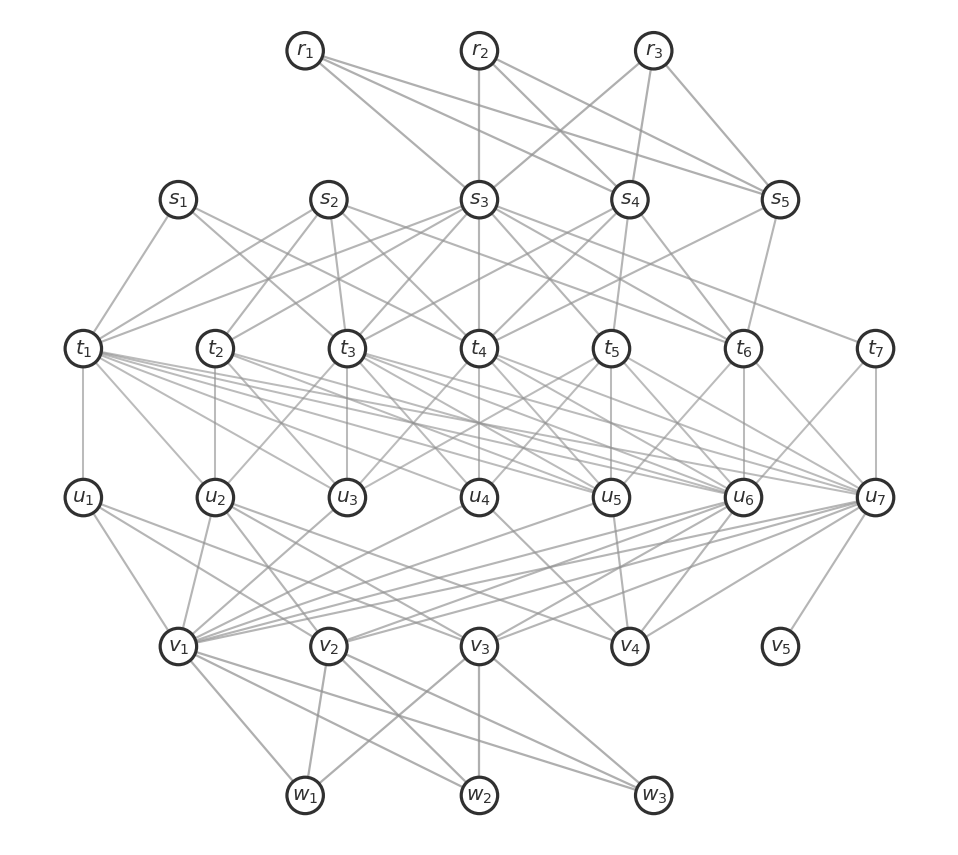}
\end{center}

\begin{center}
\includegraphics[scale=0.4]{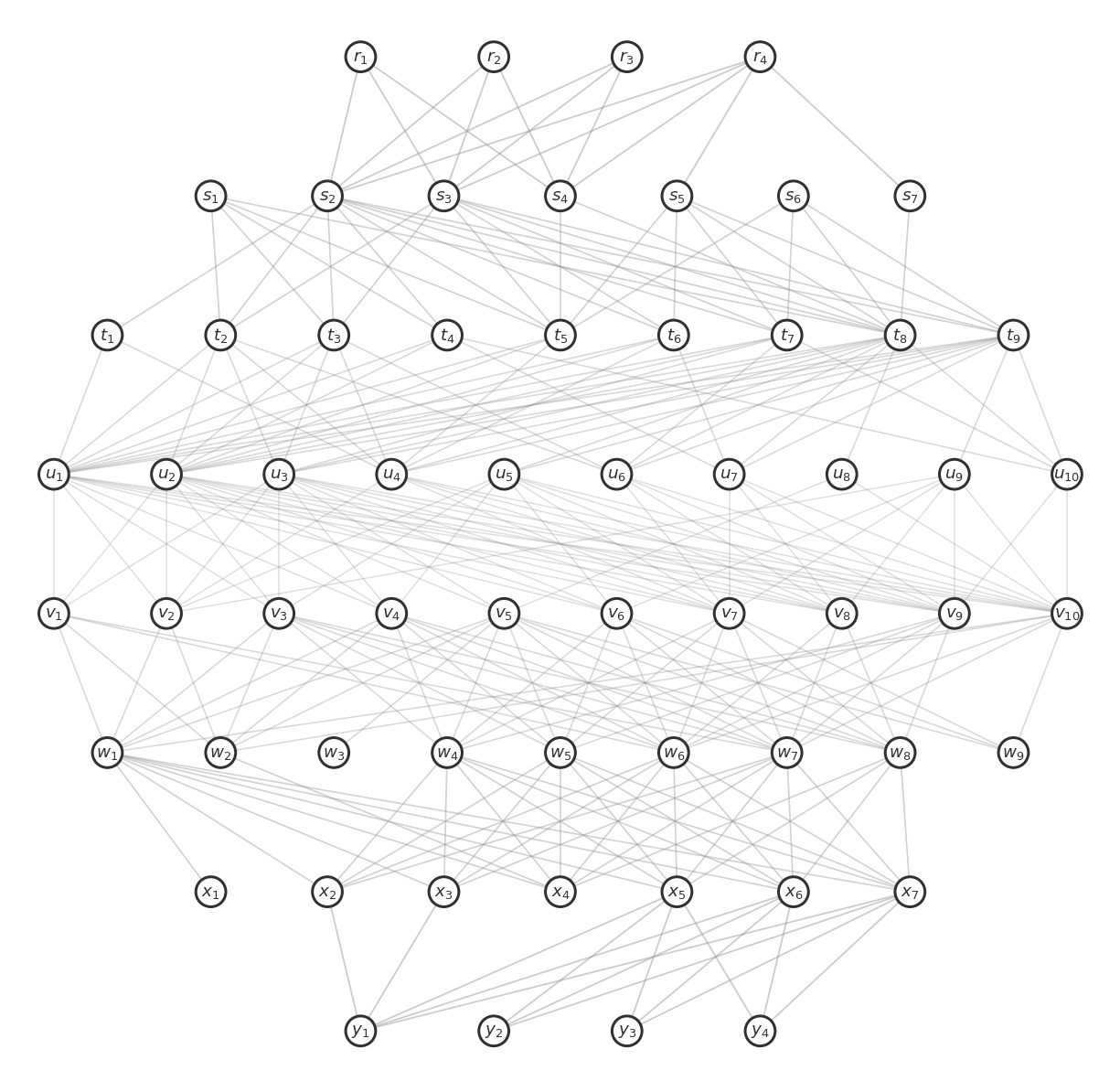}
\end{center}

Do any of these admit a strong $\mathfrak{sl}_2$-action?

\end{section}

\printbibliography

@article{leclerc94,
title = {A Finite Coxeter Group the Weak Bruhat Order of which is not Symmetric Chain},
author = {Bruno Leclerc},
journal = {European Journal of Combinatorics},
volume = {15},
number = {2},
pages = {181-185},
year = {1994},
issn = {0195-6698},
doi = {https://doi.org/10.1006/eujc.1994.1020}
}

@article{proctor82,
author = {Proctor, Robert A.},
title = {Representations of $\mathfrak{sl}( 2,\mathbb{C} )$ on Posets and the Sperner Property},
journal = {SIAM Journal on Algebraic Discrete Methods},
volume = {3},
number = {2},
pages = {275-280},
year = {1982},
doi = {10.1137/0603026}
}

@inproceedings{Donnelly1998Solitary,
  author    = {Donnelly, Robert G.},
  title     = {Solitary Bases for Irreducible Representations of
               Semisimple Lie Algebras},
  booktitle = {FPSAC '98},
  year      = {1998},
  note      = {Extended abstract, 12 pages},
  url       = {https://campus.murraystate.edu/academic/faculty/rdonnelly/Research/FPSAC98-Abstract.pdf}
}

@article{Donnelly2003Extremal,
  author  = {Donnelly, Robert G.},
  title   = {Extremal Properties of Bases for Representations of
             Semisimple {Lie} Algebras},
  journal = {Journal of Algebraic Combinatorics},
  volume  = {17},
  number  = {3},
  pages   = {255--282},
  year    = {2003},
  doi     = {10.1023/A:1025096704917}
}

@misc{DilliesJohnsonKowalska2025,
  author        = {Dillies, Ya{\"e}l and Johnson, Matthew
                   and Kowalska, Aleksandra},
  title         = {A Generalisation of {Sperner}'s Theorem
                   Using Weighted Chains},
  year          = {2025},
  eprint        = {2509.26493},
  archivePrefix = {arXiv},
  primaryClass  = {math.CO},
  url           = {https://arxiv.org/abs/2509.26493}
}

@book{AhujaMagnantiOrlin1993,
  author    = {Ahuja, Ravindra K. and Magnanti, Thomas L.
               and Orlin, James B.},
  title     = {Network Flows: Theory, Algorithms, and Applications},
  publisher = {Prentice Hall},
  address   = {Englewood Cliffs, NJ},
  year      = {1993}
}

@book{Schrijver2003,
  author    = {Schrijver, Alexander},
  title     = {Combinatorial Optimization: Polyhedra and Efficiency},
  series    = {Algorithms and Combinatorics},
  volume    = {24},
  publisher = {Springer},
  address   = {Berlin},
  year      = {2003}
}

@book {humphreys72,
    AUTHOR = {Humphreys, James E.},
     TITLE = {Introduction to {L}ie algebras and representation theory.},
    SERIES = {},
 PUBLISHER = {Springer-Verlag, New York-Berlin,, },
      YEAR = {1972},
     PAGES = {xii+169},
}

@article{Lindstrom1973,
  author  = {Lindstr{\"o}m, Bernt},
  title   = {On the vector representations of induced matroids},
  journal = {Bulletin of the London Mathematical Society},
  volume  = {5},
  number  = {1},
  pages   = {85--90},
  year    = {1973},
  doi     = {10.1112/blms/5.1.85}
}

@article{GesselViennot1985,
  author  = {Gessel, Ira and Viennot, G{\'e}rard},
  title   = {Binomial determinants, paths, and hook length formulae},
  journal = {Advances in Mathematics},
  volume  = {58},
  number  = {3},
  pages   = {300--321},
  year    = {1985},
  doi     = {10.1016/0001-8708(85)90121-5}
}

@article{Griggs1980Chains,
  author  = {Griggs, Jerrold R.},
  title   = {On chains and {Sperner} {$k$}-families in ranked posets},
  journal = {Journal of Combinatorial Theory, Series A},
  volume  = {28},
  number  = {2},
  pages   = {156--168},
  year    = {1980},
  doi     = {10.1016/0097-3165(80)90082-5}
}

\end{document}